\UseRawInputEncoding
\documentclass[11pt]{article}
\usepackage{bbm}
\usepackage{amssymb, amsthm, amsmath, amscd}
\newtheorem{thm}{Theorem}[section]
\newtheorem{lem}[thm]{Lemma}
\newtheorem{prop}[thm]{Proposition}
\newtheorem{cor}[thm]{Corollary}
\theoremstyle{definition}
\newtheorem{NN}[thm]{}
\theoremstyle{definition}
\newtheorem{df}[thm]{Definition}
\theoremstyle{definition}
\newtheorem{rem}[thm]{Remark}
\newtheorem{nota}[thm]{Notation}
\theoremstyle{definition}
\newtheorem{exm}[thm]{Example}

\renewcommand{\phi}{\varphi}

\newcommand{\N}{\mathbb{N}}

\newcommand{\R}{\mathbb{R}}
\newcommand{\C}{\mathbb{C}}

\numberwithin{equation}{section}

\newcommand{\Aff}{\operatorname{Aff}}

\newcommand{\Tw}{\overline{T(A)}^w}

\newcommand{\cpc}{c.p.c.~map}

\newcommand{\hm}{homomorphism}
\newcommand{\dt}{\delta}
\newcommand{\ep}{\varepsilon}

\newcommand{\td}{\tilde}

\newcommand{\DT}{\Delta}

\newcommand{\andeqn}{\,\,\,{\rm and}\,\,\,}
\newcommand{\rforal}{\,\,\,{\rm for\,\,\,all}\,\,\,}
\newcommand{\CA}{$C^*$-algebra}
\newcommand{\SCA}{$C^*$-subalgebra}

\newcommand{\af}{{\alpha}}
\newcommand{\bt}{{\beta}}

\newcommand{\diag}{{\rm diag}}

\newcommand{\wtd}{\widetilde}

\newcommand{\wilog}{without loss of generality}
\newcommand{\Wlog}{Without loss of generality}

\newcommand{\beq}{\begin{eqnarray}}
\newcommand{\eneq}{\end{eqnarray}}
\newcommand{\tforal}{\,\,\,\text{for\,\,\,all}\,\,\,}
\newcommand{\tand}{\,\,\,\text{and}\,\,\,}

\newcommand{\Her}{\mathrm{Her}}
\newcommand{\Cu}{\mathrm{Cu}}

\newcommand{\Qw}{\overline{QT(A)}^w}

\newcommand{\st}{such that}

\newcommand{\eD}{\partial_e(\Delta)}
\newcommand{\eT}{\partial_e(T)}

\title{Tracial approximation and ${\cal Z}$-stability}

\author{Huaxin Lin}
\date{}

\begin{document}

\maketitle
\begin{abstract}
Let $A$ be a  unital separable non-elementary amenable simple stably finite \CA\, such that 
its tracial state space  has a   $\sigma$-compact
countable-dimensional extremal boundary. 
We show that $A$ is ${\cal Z}$-stable if and only if it has strict comparison and stable rank one.
We show that  this result also holds for non-unital cases (which may not be Morita equivalent to 
unital ones).

\end{abstract}


\section{Introduction}
The Jiang-Su algebra ${\cal Z}$ -- an infinite-dimensional, simple, unital \CA\, with unique tracial state and ordered $K$-theory matching exactly that of  the complex field $\C$ (see \cite{JS})-- plays a pivotal role in the Elliott classification program. For a separable simple \CA\, $A$ with weakly unperforated $K_0(A),$  the Elliott invariant of $A$ and $A\otimes {\cal Z}$ coincide (\cite{GJS}). Consequently, ${\cal Z}$-stability (i.e., $A\cong A\otimes {\cal Z}$) 
is the natural assumption 
in the classification of separable amenable \CA s.  A central problem is determining when a simple \CA\,  is ${\cal Z}$-stable, a question  highlighted in the Toms-Winter conjecture. For a non-elementary, separable, stably finite, simple, amenable \CA\, $A,$  the conjecture posits the equivalence of:

(a) Strict comparison of positive elements,

(b) ${\cal Z}$-stability,

(c) Finite nuclear dimension.

The equivalence of (b) and (c) has since been established (see \cite{CE}, \cite{CETWW}, 
\cite{Winter-Z-stable-02}, \cite{T-0-Z}, \cite{MS2}).  
The implication (b) $\Rightarrow$ (a) is established earlier (\cite{Rrzstable}).

The remaining direction is the implication (a) $\Rightarrow$ (b) (or (c)). 
Progress began with Matui and Sato's breakthrough (\cite{MS}), resolving the case for 
unital simple \CA s with finitely many extremal traces. Subsequent work extended this to 
unital simple \CA s with tracial state spaces forming Bauer simplices with finite-dimensional extremal boundaries (see \cite{KR}, \cite{S-2} and \cite{TWW-2}),
and later to those with tight, finite-dimensional extremal traces (by Wei Zhang, see \cite{Zw}).
 Three   critical barriers
  persist:
 \vspace{0.1in}
 
1)   {\em Non-Bauer simplices:}  Moving beyond 
compact extremal boundary $\partial_e(T(A)),$
\vspace{0.1in}

2)  {\em Infinite dimensional extremal boundaries:}  Moving beyond  finite-dimensional  $\partial_e(T(A)),$

\vspace{0.1in}
 3)   {\em Non-unital algebras:}  Addressing stably projectionless simple \CA s (not stably isomorphic to unital ones).

This paper confronts  these  challenges, unifying and generalizing prior results. 

A key insight arises from the interplay between stable rank one and ${\cal Z}$-stability. 
M. R\o rdam showed that a unital finite  simple ${\cal Z}$-stable \CA\, has {\em stable rank one} (see \cite{Rrzstable}), while 
L. Robert later showed that any stably projectionless simple ${\cal Z}$-stable \CA\, has {\it almost} stable rank one (see \cite{Rlz}).  
Recent work \cite[Corollary 6.8]{FLL} improves these results:
{\em all finite simple ${\cal Z}$-stable \CA\,--unital or not --have  stable rank one.}
This improvement positions stable rank one as both a consequence of ${\cal Z}$-stability
and a complementary condition  to strict comparison in the implication of 
(a) $\Rightarrow$ (b).
Our strategy  in this paper is to replace condition (a) 
above by that $A$ has strict comparison and stable rank one, leveraging their intrinsic 
connection via $T$-tracial approximate oscillation zero \cite[Theorem 1.1]{FLosc}
(see  also the last line of the text).

\vspace{0.1in}

Main Result
\vspace{0.1in}


We establish:

\begin{thm}\label{TM-2}
Let $A$ be a non-elementary separable amenable  simple \CA\, with $\wtd T(A)\setminus \{0\}\not=\emptyset$
such that $\wtd T(A)$ has a
$\sigma$-compact countable-dimensional extremal boundary (see Definition \ref{Dscdim}).
Then the  following are equivalent.

(1) $A$ has strict comparison and T-tracial approximate oscillation zero,

(2) $A$ has strict comparison and stable rank one,

(3) $A\cong A\otimes {\cal Z}.$
%
\end{thm}

Key Advancements
\vspace{0.1in}

Theorem \ref{TM-2} generalizes prior work in three directions:
\vspace{0.1in}

$\bullet$ {\em Non-Bauer simplexes:}  The extremal boundary $\partial_e(\wtd T(A))$
need not be compact.
\vspace{0.1in}

$\bullet$  {\em Infinite-dimensional boundaries:}   $\partial_e(\wtd T(A))$ 
can be countable-dimensional (equivalently, of transfinite dimension if compact 
\cite[Corollary 7.1.32]{En}).

\vspace{0.1in}
$\bullet$ {\em Non-unital algebras:}  $A$  may be stably projectionless.

\vspace{0.1in}

One notices that Theorem \ref{TM-2} covers the case that  the cone $\wtd T(A)$ has a basis $S$ whose extremal boundary $\partial_e(S)$  is
compact and countable-dimensional (in this special case, Theorem \ref{TM-2}  
generalizes the results in  \cite{KR}, \cite{S-2} and \cite{TWW-2} under the  additional  (but necessary)
condition that $A$ has stable rank one).
%
If $\partial_e(S)$ has only countably many points, 
then $A$ has T-tracial approximate oscillation zero. By \cite[Theorem 1.1]{FLosc}, the condition that $A$ has stable rank one 
is automatic. In other words,  the original Toms-Winter conjecture holds when the cone $\wtd T(A)$ has a    
basis $S$ which has countably many extremal points (but not necessarily closed) which
 also generalizes the result in \cite{MS}. 
 %

\vspace{0.1in}

Technical Insight

\vspace{0.1in}
To prove ${\cal Z}$-stability, we refine the  tracial approximate divisibility-- a strategy tracing to Matui-Sato \cite{MS}.
But circumvent
the central sequence algebra 
$\pi_\infty^{-1}(A')/I_{_{\varpi}}$
by working directly in $l^\infty(A)/I_{_{\varpi}}.$
Under $T$-tracial oscillation zero, $l^\infty(A)/I_{_{\varpi}}$
has real rank zero, enabling matrix algebra constructions that approximate elements in trace norm. This framework accommodates non-unital algebras and non-Bauer simplexes,
addressing the three challenges mentioned above.
\vspace{0.1in}

Organization
\vspace{0.1in}

Section 2  serves as preliminaries. Section 3 revisits quotient algebras and establishes a key norm condition (Theorem 3.14). Section 4 analyzes non-Bauer simplexes with countable-dimensional boundaries. Sections 5 -- 7 develop  some stability results in trace
2-norm and matrix approximations. Section 8 proves Theorem 1.1, while Section 9 discusses open questions.

\section {Notations}

\begin{df}\label{D1}
Let $A$ be a \CA.
Denote by $A^{\bf 1}$ the closed unit ball of $A,$  and 
by $A_+$ the set of all positive elements in $A.$
Put $A_+^{\bf 1}:=A_+\cap A^{\bf 1}.$
Denote by $\wtd A$ the minimal unitization of $A.$ Let $S\subset A$   be a subset of $A.$
Denote by
${\rm Her}(S)$
the hereditary $C^*$-subalgebra of $A$ generated by $S.$

Denote by $T(A)$ the tracial state space of $A.$
For $r>0,$ set
$T_{[0,r]}(A)=\{r\tau: \tau\in T(A)\andeqn r\in [0,r]\}$
and $T_{(0,r]}(A)=\{r\tau: \tau\in T(A)\andeqn r\in (0,r]\}.$

 Let  ${\rm Ped}(A)$ denote
the Pedersen ideal of $A$ and ${\rm Ped}(A)_+:= {\rm Ped}(A)\cap A_+$.  

Unless otherwise stated (except  for Pedersen ideals), an ideal of a \CA\, is {\it always} a closed and  two-sided ideal.
\end{df}

\begin{df}\label{Defj}
Let $A$ and $B$ be \CA s and 
$\phi: A\rightarrow B$ be a  linear map.
The map $\phi$ is said to be positive if $\phi(A_+)\subset B_+.$  
The map $\phi$ is said to be completely positive contractive, abbreviated to c.p.c.,
if $\|\phi\|\leq 1$ and  
$\phi\otimes \mathrm{id}: A\otimes M_n\rightarrow B\otimes M_n$ is 
positive for all $n\in\mathbb{N}.$ 
A c.p.c.~map $\phi: A\to B$ is called order zero, if for any $x,y\in A_+,$
$xy=0$ implies $\phi(x)\phi(y)=0$ (see  Definition 2.3 of \cite{WZ}).
If $ab=ba=0,$ we also write $a\perp b.$

In what follows, $\{e_{i,j}\}_{i,j=1}^n$ (or just $\{e_{i,j}\},$ if there is no confusion) stands for  a system of matrix units for $M_n$ and $\jmath\in C_0((0,1])$ 
is the identity function on $(0,1],$  i.e., $\jmath(t)=t$ for all $t\in (0,1].$

If $D$ is a finite dimensional \CA, and $\phi: D\to B$ is an order zero \cpc, 
then there exists a unique \hm\, $\phi_c: C_0((0,1])\otimes D\to B$ such that 
$\phi_c(\jmath\otimes d)=\phi(d)$ for all $d\in D$ (\cite[Proposition 1.2.1]{WcovII}).
Conversely, if $\phi_c:C_0((0,1])\otimes D\to B$ is a \hm, then $\phi: D\to B$ defined by
$\phi(d)=\phi_c(\jmath\otimes d)$ for all $d\in D$ is an order zero \cpc.
This fact will be used frequently (without further warning).
\end{df}

\begin{nota}
Throughout the  paper,
the set of all positive integers is denoted by $\N.$ 
%
Let  $A$ 
be a normed space and ${\cal F}\subset A$ be a subset. For any  $\epsilon>0$ and 
$a,b\in A,$
we  write $a\approx_{\epsilon}b$ if
$\|a-b\|< \epsilon$.
We write $a\in_\ep{\cal F}$ if there is $x\in{\cal F}$ such that
$a\approx_\ep x.$


\end{nota}

\begin{df}\label{Dcuntz}
Denote by ${\cal K}$ the \CA\, of compact operators on $l^2.$ 
Let $A$ be a \CA\, and 
$a,\, b\in (A\otimes {\cal K})_+.$ 
We 
write $a \lesssim b$ if there are
$x_i\in A\otimes {\cal K}$  for all $i\in \N$ 
such that
$\lim_{i\rightarrow\infty}\|a-x_i^*bx_i\|=0$.
We write $a \sim b$ if $a \lesssim b$ and $b \lesssim a$ both  hold.
The Cuntz relation $\sim$ is an equivalence relation.
Set $\Cu(A)=(A\otimes {\cal K})_+/\sim.$
Let $[a]$ denote the equivalence class of $a$. 
We write $[a]\leq [b] $ if $a \lesssim  b$.
\end{df}

\begin{nota}\label{Nfg}
Let $\epsilon>0.$ Define a  continuous function
$f_{\epsilon}
: {{[0,+\infty)}}
\rightarrow [0,1]$ by
\beq\nonumber
f_{\epsilon}(t)=
\begin{cases}
0  &t\in {{[0,\epsilon/2]}},\\
1 &t\in [\epsilon,\infty),\\
2(t-\ep/2)/\ep&{t\in[\epsilon/2, \epsilon].}
\end{cases}
\eneq
\end{nota}

\begin{df}\label{Dqtr}
{{Let $A$ be a $\sigma$-unital \CA. 
A densely  defined 2-quasi-trace  is a 2-quasi-trace defined on ${\rm Ped}(A)$ (see  Definition II.1.1 of \cite{BH}). 
Denote by ${\widetilde{QT}}(A)$ the set of densely defined quasi-traces 
on 
$A\otimes {\cal K}.$  
%
Denote by ${\widetilde{T}}(A)$ the set of densely defined traces 
on 
$A\otimes {\cal K}.$  
 In what follows we will identify 
$A$ with $A\otimes e_{1,1},$ whenever it is convenient. 
Let $\tau\in {\widetilde{QT}}(A).$  Note that $\tau(a)\not=\infty$ for any $a\in {\rm Ped}(A)_+\setminus \{0\}.$

We endow ${\widetilde{QT}}(A)$ 
{{with}} the topology  in which a net 
${{\{}}\tau_i{{\}}}$ 
 converges to $\tau$ if 
${{\{}}\tau_i(a){{\}}}$ 
 converges to $\tau(a)$ for all $a\in 
 {\rm Ped}(A)$ 
 (see also (4.1) on page 985 of \cite{ERS}).
  Denote by $QT(A)$ the set of normalized 2-quasi-traces of $A$ ($\|\tau|_A\|=1$) and (for $r>0$) 
 $QT_{[0,r]}(A)=\{r\tau: \tau \in QT(A): r\in [0,r]\}$  and 
 $QT_{(0,r]}(A)=\{r\tau: \tau\in QT(A): r\in (0,r]\}.$ 
 A  convex subset $S\subset {\wtd{QT}}(A)\setminus \{0\}$ is a basis
 for ${\wtd{QT}}(A),$ if for any $t\in {\wtd{QT}}(A)\setminus \{0\},$ there exists a unique  pair $r\in \R_+$ and 
 $s\in S$ such that $r\cdot s=t.$  
Let $e\in {\rm Ped}(A)_+\setminus \{0\}$ be a full element of $A.$
 Then $S_e=\{\tau: \tau\in {\wtd{QT}}(A): \tau(e)=1\}$ is a Choquet simplex and is a basis 
 for the cone ${\wtd{QT}}(A)$ (see Proposition 3.4 of \cite{TT}).

}}

Note that, for each $a\in ({{A}}
\otimes {\cal K})_+$ and $\ep>0,$ $f_\ep(a)\in {\rm Ped}(A\otimes {\cal K})_+.$ 
Define 
\beq
\widehat{[a]}(\tau):=d_\tau(a)=\lim_{\ep\to 0}\tau(f_\ep(a))\rforal \tau\in {\widetilde{QT}}(A).
\eneq

Let $A$ be a simple \CA\,with $\wtd{QT}(A)\setminus \{0\}\not=\emptyset.$
Then
$A$
is said to have (Blackadar's) strict comparison, if, for any $a, b\in (A\otimes {\cal K})_+,$ 
condition 
\beq
d_\tau(a)<d_\tau(b)\rforal \tau\in {\widetilde{QT}}(A)\setminus \{0\}
\eneq
implies that 
$a\lesssim b.$ 
%
%
%
%
\end{df}

It is known that any trace defined on ${\rm Per}(A)$ is lower semicontinuous on ${\rm Ped}(A)$ (see, e.g., 
\cite[Proposition 5.6.7]{Pbook}). 
It is perhaps less known that this also holds for $2$-quasi-traces. We present here for clarification.

\begin{prop}\label{Plin}
Let $A$ be a \CA, ${\rm Ped}(A)$ be its Pedersen ideal  and 
$\tau$  a 2-quasi-trace defined on ${\rm Ped}(A).$ Then $\tau$ is lower semicontinuous on ${\rm Ped}(A).$

Moreover $\tau$ can be extended to a lower semicontinuous quasi-trace on $A_+,$ 
i.e., there is a lower semicontinous map $\widetilde \tau: A_+\to [0, \infty]$ such that\\  
(1) $\widetilde \tau(\af x)=\af \tilde \tau(x)$ for $\af\in \R_+,$\\ 
(2) $\widetilde \tau(x+y)=\widetilde \tau(x)+\widetilde \tau(y)$ if $x, y\in A_+$ and $xy=yx,$\\ 
(3) $\widetilde\tau(x^*x)=\widetilde\tau(xx^*)$
for all $x\in A,$ and\\
 (4) $\widetilde\tau(a)=\tau(a)$ for $a\in {\rm Ped}(A)_+.$
\end{prop}

\begin{proof}
For any $x\in {\rm Ped}(A)_+,$ then (by \cite[Proposition 5.6.2]{Pbook})  $\overline{xAx}\subset  {\rm Ped}(A).$
It follows from \cite[II.2.5]{BH} that  $\tau|_{\overline{xAx}}$ is  norm continuous.
In particular, 
\beq
\lim_{k\to\infty} \tau(f_{1/k}(x)xf_{1/k}(x))=\tau(x).
\eneq
Note that, for any $y\in {\rm Ped}(A)_{s.a.},$  $|y|\in {\rm Ped}(A).$ As ${\rm Ped}(A)$ is hereditary,
$y_+, y_-\in {\rm Ped}(A).$  Hence to show that $\tau$ is lower semicontinuous, it suffices to show the following: 
Suppose that $x_n\in {\rm Ped}(A)_+$ and 
$\lim_{n\to\infty}\|x_n-x\|=0,$ then 
\beq
\liminf_{n\to\infty} \tau(x_n)\ge \tau(x).
\eneq
  For each $k\in \N,$ $f_{1/k}(x)x_nf_{1/k}(x)\in {\overline{xAx}}.$
Since $\tau|_{\overline{xAx}}$ is norm continuous, we must have 
\beq
\lim_{n\to\infty} \tau(f_{1/k}(x)x_nf_{1/k}(x))=\tau(f_{1/k}(x)xf_{1/k}(x)).
\eneq
Note also that, for all $n, k\in \N,$ 
\beq
\tau(x_n)\ge \tau(x_n^{1/2}f_{1/k}(x)^2x_n^{1/2})=\tau(f_{1/k}(x)x_nf_{1/k}(x)).
\eneq
Hence, for all $k\in \N,$ 
\beq
\liminf_{n\to\infty} \tau(x_n)\ge \liminf_{n\to\infty}\tau((f_{1/k}(x)x_nf_{1/k}(x)))=\tau(f_{1/k}(x)xf_{1/k}(x)).
\eneq
Let $k\to\infty.$ We obtain
\beq
\liminf_{n\to\infty} \tau(x_n)\ge \tau(x).
\eneq
Thus $\tau$ is lower semicontinuous on ${\rm Ped}(A).$ 

To see the existence of $\widetilde\tau,$ 
let $\{e_\lambda\}=\{e\in {\rm Ped}(A)_+: \|e\|<1\}$ with the order inherited from  $A_+.$
Define  $\tau_\lambda: A\to \C$ by $\tau_\lambda(a)=\tau(e_\lambda a).$ Note 
that $e_\lambda a\in {\rm Ped}(A).$ So each $\tau_\lambda$ is norm continuous.
Moreover $\tau_\lambda$ satisfies (1). 
Define $\widetilde \tau(a)=\lim_\lambda \tau(e_\lambda a)$ for all $a\in A_+.$
Note that, for any $a\in A_+,$  if $\lambda\le \mu,$ 
\beq
\tau_\lambda(a)=\tau(a^{1/2} e_\lambda a^{1/2})\le \tau(a^{1/2}e_\mu a^{1/2})=\tau_\mu(a). 
\eneq
Thus $\widetilde \tau: A_+\to [0, \infty]$ is lower semicontinuous and satisfies (1).
Moreover, if $a\in {\rm Ped}(A)_+,$ 
$a^{1/2}e_\lambda a^{1/2}\in \overline{aAa}.$ 
Hence $\lim_\lambda \|a^{1/2}e_\lambda a^{1/2}-a\|=0.$
Therefore, since $\tau$ is norm continuous on $\overline{a Aa},$
\beq
\widetilde\tau(a)=\lim_\lambda \tau(e_\lambda a)=\lim_{\lambda} \tau(a^{1/2} e_\lambda a^{1/2})=
\tau(a).
\eneq
So $\widetilde\tau$ also satisfies (4).

For $a, b\in A_+$ with $ab=ba,$  let $C=C^*(a, b).$
There is an isomorphism $\phi: C\cong  C_0(X\setminus \{(0,0)\}),$
where $X\subset [0,\|a\|]\times [0,\|b\|]$ is a compact subset.
Choose $0<\sigma<1,$ let $F_\sigma(\zeta)\in C_0(X\setminus \{(0,0)\})$ be 
such that $0\le F_\sigma\le 1,$  $F_\sigma(\zeta)=0$ if $|\zeta|<\sigma/2$ and 
$F_\sigma(\zeta)=1$ if $|\zeta|\ge \sigma.$
Let $c_\sigma\in C$ be such that $c_\sigma=\phi^{-1}(F_\sigma).$ 
Note $c_\sigma\in \{e_\lambda\}.$
Then
\beq\label{930}
\tau(c_\sigma(a+b))=\tau(c_\sigma a)+\tau(c_\sigma  b )\le \tilde \tau(a)+\tilde\tau(b).
\eneq 
Since $\lim_{\sigma\to 0}\|c_\sigma(a+b)-(a+b)\|=0,$ by  \eqref{930}, (4) and lower-semicontinuity of $\tilde \tau,$ 
\beq
\tilde \tau(a+b)\le \liminf_{\sigma\to 0}\tilde  \tau(c_\sigma(a+b))= \liminf_{\sigma\to 0} \tau(c_\sigma(a+b))\le \tilde \tau(a)+\tilde\tau(b).
\eneq
On the other hand, by the first part of \eqref{930}, we also have, for all $0<\sigma <1,$ 
\beq
\tilde \tau(a+b)\ge \tau(c_\sigma a)+\tau(c_{\sigma} b ).
\eneq
Hence
\beq
\tilde \tau(a+b)\ge \liminf_{\sigma\to 0}( \tau(c_\sigma  a) +\tau(c_{\sigma} b))
\ge \liminf_{\sigma\to 0} \tau(c_{\sigma} a)+\lim\inf_{\sigma\to 0}\tau(c_\sigma b)\\
= \liminf_{\sigma\to 0} \tilde \tau(c_{\sigma} a)+\lim\inf_{\sigma\to 0}\tilde \tau(c_\sigma b)\ge \tilde \tau(a)+\tilde \tau(b).
\eneq
Then (2) follows.

Let $x\in A.$ Then 
\beq
\widetilde\tau(x^*x)&=&\lim_\lambda \tau(e_\lambda x^*x)=\lim_\lambda \tau(e_\lambda^{1/2} x^*xe_\lambda^{1/2})=
\lim_\lambda \tau(xe_\lambda x^*) \\
&=&\lim_\lambda (\lim_\mu\tau(e_\mu xe_\lambda x^*))=
\lim_\lambda (\lim_\mu\tau(e_\mu^{1/2} xe_\lambda x^* e_\mu^{1/2}))\\
&\le &\lim_\lambda (\lim_\mu \tau(e_\mu^{1/2} xx^* e_\mu^{1/2}))=
\lim_\mu \tau(e_\mu xx^*)=\widetilde \tau(xx^*).
\eneq
Applying the above to $x^*,$ we obtain
\beq
\widetilde \tau(xx^*)\le \widetilde \tau(x^*x).
\eneq
Hence $\widetilde \tau(x^*x)=\widetilde\tau(xx^*)$ and $\widetilde\tau$ satisfies (3).
\end{proof}

\begin{df}\label{DGamma} 
Let $A$ be a \CA\, with ${\wtd{QT}}(A)\setminus\{0\}\not=\emptyset.$
Denote by $CL(\wtd{QT}(A))$ the set of real continuous functions on $\wtd{QT}(A)$ 
such that $f(r\tau)=rf(\tau)$ and $f(\tau_1+\tau_2)=f(\tau_1)+f(\tau_2)$ 
for all $r\in \R$ and $\tau, \tau_1,\tau_2\in \wtd{QT}(A).$
Let $S\subset {\wtd{QT}}(A)$ be a convex subset. 
Set (if $0\not\in S,$ we ignore the condition $f(0)=0$)
\beq\nonumber
\Aff_+(S)&=&\{f\in CL(\wtd{QT}(A))_+: f \,\,
{{\rm affine}}, 
f(s)>0\,\,{\rm for}\,\,s\not=0,\,\, f(0)=0\}\cup \{0\},\\
%
{\rm LAff}_+(S)&=&\{f:S\to [0,\infty]: \exists \{f_n\}, f_n\nearrow f,\,\,
 f_n\in \Aff_+(S)\},
 \eneq
(where $f_n\nearrow f$ means that the sequence $\{f_n\}$ is increasing and 
it converges pointwise  to $f$). 

For a  simple \CA\, $A$ and   each $a\in (A\otimes {\cal K})_+,$ the function $\hat{a}(\tau)=\tau(a)$ ($\tau\in S$) 
is in general in ${\rm LAff}_+(S).$   If $a\in {\rm Ped}(A\otimes {\cal K})_+,$
then $\hat{a}\in \Aff_+(S).$
For 
$\widehat{[a]}(\tau)=d_\tau(a)$ defined above,   we have 
$\widehat{[a]}\in {\rm LAff}_+({\wtd{QT}}(A)).$

We write $\Gamma: \Cu(A)\to {\rm LAff}_+({\wtd{QT}}(A))$ for 
the canonical map defined by $\Gamma([a])(\tau)=\widehat{[a]}(\tau)=d_\tau(a)$ 
for all $\tau\in {\wtd{QT}}(A).$

In the case that $A$ is  algebraically simple (i.e., $A$ is a simple \CA\,  and $A={\rm Ped}(A)$), 
$\Gamma$ also induces a canonical map 
$\Gamma_1: \Cu(A)\to {\rm LAff}_+(\Qw),$
where $\Qw$ is the weak*-closure of $QT(A).$  Since, in this case, 
$\R_+\cdot \Qw={\wtd{QT}}(A),$ the map $\Gamma$ is surjective if and only if $\Gamma_1$
is surjective.  We would like to point out that, in this case, $0\not\in \Qw$ (see Lemma 4.5 of \cite{eglnp}).

In the case that $A$ is stably finite and simple,
denote by $\Cu(A)_+$ the set of purely non-compact elements (see Proposition 6.4 of \cite{ERS}).
Suppose that $\Gamma$ is surjective.  Then  $\Gamma|_{\Cu(A)_+}$ is surjective as well  (see Theorem 7.12 \cite{FLosc}, for example).

 It is also helpful to note that, if $A$ is algebraically simple and $T(A)$ (or $QT(A)$) 
 is compact, then $T(A)$ (or $QT(A)$) is a compact basis for the cone $\wtd T(A).$ 
 It follows that $T(A)$  (or $QT(A)$)
 it is a metrizable Choquet simplex (see \cite[Theorem 3.1]{PdMIII}). 

%
%
%
For most part of the paper,  we  will assume that all 2-quasi-traces of a separable \CA\, $A$ 
are in fact traces for convenience. This is the case if $A$ is exact (by \cite{Haagtrace}).
\end{df}

\begin{df}\label{DNcu}
Let $l^\infty(A)$ be the \CA\ of bounded sequences of $A.$
{{Recall}} that 
$c_0(A):=\{\{a_n\}\in l^\infty(A): \lim_{n\to\infty}\|a_n\|=0\}$ 
is an 
ideal of $l^\infty(A).$ 
We view $A$ as a subalgebra of $l^\infty(A)$ 
via the canonical map $\iota: a\mapsto\{a,a,,...\}$ for all $a\in A.$

Put $\pi_\infty^{-1}(A')=\{\{x_n\}\in l^\infty(A):  \lim_{n\to\infty}\|x_na-ax_n\|=0,\,\, \forall a\in A\}.$ 
It is called the central sequence algebra of $A.$
\end{df}

\begin{df}\label{Dt2norm}
Let $A$ be a \CA\, and $\tau$  a finite 2-quasi-trace on $A.$
For each $a\in A,$ define 
\beq
\|a\|_{_{2, \tau}}=\tau(a^*a)^{1/2}.
\eneq
If $F$ is a subset of  finite 2-quasi-traces on $A,$ define 
\beq
\|a\|_{_{2, F}}=\sup\{\|a\|_{2, \tau}: \tau\in F\}.
\eneq

Fix a free  ultrafilter $\varpi\in \bt(\N)\setminus \N.$ 
Let $A$ be a separable \CA\, and $\tau$ be a non-zero finite quasi-trace on $A.$ 
Define 
\beq
I_{_{\tau, \varpi}}&=&\{\{a_n\}\in l^\infty(A): \lim_{n\to\varpi}\tau(a_n^*a_n)=0\}\andeqn\\
I_{_{\tau, \N}}&=&\{\{a_n\}\in l^\infty(A): \lim_{n\to\infty}\tau(a_n^*a_n)=0\}.
\eneq
If $F\subset QT_{(0,r]}(A)$ is a  subset ($r>0$), define 
\beq
I_{_{F, \varpi}}&=&\{\{a_n\}\in l^\infty(A): \lim_{n\to\varpi}\sup_{\tau\in F}\tau(a_n^*a_n)=0\}\andeqn\\
I_{_{F, \N}}&=&\{\{a_n\}\in l^\infty(A): \lim_{n\to\infty}\sup_{\tau\in F}\tau(a_n^*a_n)=0\}.
\eneq
Let $r>0$ and $C_{[0, r]}(F)=\{ t \cdot s: s\in F: 0\le t\le r\}.$ 
It is clear that $I_{_{F, \varpi}}=I_{_{{\rm conv}(F)}, \varpi}=I_{_{C_{[0,r]}(F), \varpi}}.$
Since, for each $n\in \N,$  
$
\sup\{\tau(a_n^*a_n): \tau\in F\}=\sup \{\tau(a_n^*a_n): \tau \in \overline{F}\},
$
$I_{_{F, \varpi}}=I_{_{\overline{F}, \varpi}}.$

In the case $F=\overline{QT(A)}^w,$ or $F=\overline{T(A)}^w,$  we may write $I_{_{\varpi}}$ instead of 
$I_{_{\Qw, \varpi}}$ (or $I_{_{\Tw, \varpi}}$),  in particular, in the case that $QT(A)$ is a non-empty compact set.
Since $\|ab\|_{_{2, F}}\le \|a\|\|b\|_{_{2, F}}$ and 
$\|x^*\|_{_{2, F}}=\|x\|_{_{2, F}}$ (see Lemma 3.5 of \cite{Haagtrace} and Definition 2.16 of \cite{FLosc}),
both $I_{_{\tau, \varpi}}$ and $I_{_{F, \varpi}}$ are (closed two-sided) ideals of $l^\infty(A)$ (see also Proposition \ref{PIdeal}).
Denote by $\Pi_{_{\tau, \varpi}}: l^\infty(A)\to l^\infty(A)/I_{_{\tau, \varpi}}$ and 
$\Pi_{_{F, \varpi}}:  l^\infty(A)\to l^\infty(A)/I_{_{F, \varpi}}$ the quotient maps, respectively.
We also write $\Pi_{_{\varpi}}$ for $\Pi_{_{T(A), \varpi}}.$

Let $\tau\in T_{(0,1]}(A).$ 
Define, for any $a=\Pi_{\varpi}(\{a_n\}),$ 
\beq
\tau_\varpi(a)=\lim_{n\to\varpi}\tau(a_n).
\eneq
 Then, if $\tau\in F,$  $\tau_\varpi$ is a trace on $l^\infty(A)/I_{_{F, \varpi}}.$
%
For $a=\{a_n\}\in l^\infty(A)/I_{_{\tau, \varpi}},$ 
define 
\beq
\|a\|_{_{2, \tau_\varpi}}=\lim_{n\to\varpi}\tau(a_n^*a_n)^{1/2},
\eneq
where $a=\Pi_{\tau, \varpi}(\{a_n\}).$  Define 
\beq
\|a\|_{_{2, F_\varpi}}=\lim_{n\to\varpi}\sup\{\|a_n\|_{2,\tau}: \tau\in F\}.
\eneq
Then,  
 $\|a\|_{_{2, \tau_\varpi}}$ is a norm on $l^\infty(A)/I_{_{\tau, \varpi}}$ and 
$\|a\|_{_{2, F_\varpi}}$ is a norm on $l^\infty(A)/I_{_{F, \varpi}}.$

Let $\{\tau_n\}$ be a sequence of traces in $F.$ 
Define, for each $a=\{a_n\}\in l^\infty(A),$ 
\beq
\tau_\varpi(a)=\lim_{n\to\varpi}\tau_n(a_n).
\eneq
Denote by $F^\varpi$ the set of all these limit traces. 
\end{df}

\begin{df}[see  Definition A.1 of \cite{eglnkk0}]\label{DefOS1}
Let $A$ be a \CA\,  
Let $S\subset
 {\wtd{QT}}(A)
\setminus \{0\}$ be a compact subset. 
Define, for each $a\in {\rm Ped}(A\otimes {\cal K})_+,$ 
\beq
\omega(a)|_S&=&\inf\{\sup\{d_\tau(a)-\tau(c): \tau\in S\}: c\in \overline{a(A\otimes {\cal K})a}, \,0\le c\le 1\}.
\eneq

Recall (from Theorem 4.7 of \cite{eglnp}) that a $\sigma$-unital \CA\, $A$ is called compact, if ${\rm Ped}(A)=A.$
Then, by Lemma 4.5 of \cite{eglnp}, $0\not\in \Qw.$ 
If $A$ is compact,  we may choose $S=\Qw.$
In that case, we will omit $S$ in the notation.
Note that $\omega(a)=0$ if and only if $d_\tau(a)$ is continuous on $S.$

Let $A$ be a $\sigma$-unital simple \CA\, with $\wtd{QT}(A)\not=\{0\}.$
Let $e\in A$ be a strictly positive element.
If $A$ has continuous scale,  then $\omega(e)=0$ and 
$QT(A)$ is compact (see, for example, Proposition 5.4 of \cite{eglnp}).
  If $A$ also has strict comparison, then $A$ has continuous scale 
 if and only if $\omega(e)=0$ (see Proposition 5.4 and Theorem 5.3 of \cite{eglnp}). 
 %
 %
\end{df}

\begin{df}[Definition 4.7 of \cite{FLosc}]\label{defOs2}
Let $A$ be a  $\sigma$-unital compact \CA\, with $QT(A)\not=\emptyset.$
Let $a\in {\rm Ped}(A\otimes {\cal K})_+$ and let
$\Pi: l^\infty(A)\to l^\infty(A)/I_{_{\Qw,\N}}$ be the quotient map. Define
\beq\nonumber
\Omega^T(a)=\inf\{\|\Pi(\iota(a)-\{b_n\})\|: b_n\in \Her(a)_+, \|b_n\|\le \|a\|, \,\, \lim_{n\to\infty}\omega(b_n)|_{\Qw}=0\}.
\eneq
One may call $\Omega^T(a)$ the tracial approximate oscillation of $a.$
If $\Omega^T(a)=0,$ we say  that the element $a$ has  approximately tracial  oscillation zero.
Note that $\Omega^T(a)=0$ if and only if there exists $b_n\in \Her(a)_+$ with $\|b_n\|\le \|a\|$ such that  
\beq
\lim_{n\to\infty}\|a-b_n\|_{_{2, \Qw}}=0\andeqn\lim_{n\to\infty}\omega(b_n)=0
\eneq
(see   Proposition 4.8 
 of \cite{FLosc}). 

We say that $A$ has T-tracial approximate oscillation zero, if $\Omega^T(a)=0$ 
for all $a\in  {\rm Ped}(A\otimes {\cal K})_+$ (we still assume that $A$ is compact).
If we view $\|\cdot \|_{_{2, \Qw}}$ as an  $L^2$-norm, 
 the condition is analogous  to the fact that ``almost'' continuous functions are $L^2$-norm dense.
 It is shown in \cite[Theorem 1.1]{FLosc} that a separable  simple \CA\,  with $\wtd{QT}(A)\setminus \{0\}\not=\emptyset$
 which also has strict comparison has T-tracial approximate oscillation if and only if $A$ has stable rank one.
 \end{df}
 

 \begin{df}\label{Ppermproj}
Let $\varpi\in \bt(\N)\setminus \N$ be a free ultrafilter.
Let $p\in l^\infty(A)/I_{_{\Qw, \varpi}}$ be a projection. An element $\{e_n\}\in l^\infty(A)_+^{\bf 1}$ is called 
a permanent projection lifting of $p,$ if, for any sequence of integers $\{m(n)\},$
$\Pi_\varpi(\{e_n^{1/m(n)}\})=p.$
\end{df}

 \begin{df}[see Definition 7.1.1 of \cite{En}]\label{Dtransfinite}
 To every metrizable  space X  one assigns  the small  transfinite dimension of X, denoted by ${\rm trind}(X),$ which is the integer $-1,$  an ordinal number, or the symbol $\Omega.$ The value of ${\rm trind}(X)$ is uniquely determined by the following conditions.

(1) ${\rm trind}(X) = -1,$  if and only if $X = \emptyset,$

 (2)  ${\rm trind}(X) \le \af,$  where $\af$ is an ordinal number, if for every point 
 $x \in X$ and each neighborhood $V$ of $x,$ there exists an open set $U \subset X$ such that
$x\in U \subset V$ and ${\rm trind}(bd(U))<\af,$ where  $bd(U)=\overline{U}\setminus U,$ 

(3) ${\rm  trind}(X) = \af,$  if ${\rm trind}(X)\le  \af$  and ${\rm trind}(X) \le \bt$  for no ordinal $\bt  < \af,$

(4) ${\rm trind}(X) =\Omega,$  if there is no ordinal $\af$  such that ${\rm trind}(X) \le  \af.$
 
 
 If $\af$ is an ordinal and ${\rm trind}(X)\le \af,$ then we say that $X$ has transfinite dimension.
 If $X$ is separable and ${\rm trind}(X)=n\in \N,$ then ${\rm trind}(X)={\rm dim}(X)$ (the covering dimension) of $X.$
 \end{df}
 
 \begin{df}\label{Dcountabldim}

(i)  A separable metrizable space $X$ is said to have countable dimension if it is 
a countable union of  finite (covering) dimensional subsets (5.1.1 of \cite{En}). 
It is known that a separable metrizable space $X$ which has transfinite dimension must be countable-dimensional
(\cite[Theorem 7.1.8]{En}) and every completely metrizable space with countable dimension has transfinite dimension. 
 %
%
%

(ii) A compact metrizable space $X$ is countable-dimensional if and only if $X$ has transfinite 
dimension (\cite[Corollary 7.1.32]{En}), and in this case ${\rm trind}(X)$ is a countable ordinal (see
Corollary 7.1.27 and Corollary 7.1.32 of \cite{En}). 

(iii) Let $X$ be a separable metrizable space which is $\sigma$-compact and countable-dimensional.
Then  
$X=\cup_{n=1}^\infty X_n,$ where $X_n\subset X_{n+1}$ and each $X_n$ is compact and 
has  countable dimension.
Since a compact metrizable space is countable-dimensional if and only if it has transfinite dimension,
$X$ is $\sigma$-compact and countable-dimensional if and only if 
$X=\cup_{n=1}^\infty X_n,$ where $X_n\subset X_{n+1}$ and each $X_n$ is compact and 
has  transfinite dimension.

We refer to \cite{En} 
(also \cite{bW}) 
for examples and more discussions for countable and (small and large) transfinite 
dimensional spaces. 
 \end{df}

\section{Some basics and quotients}

One of the purposes of this section is to present  Theorem \ref{Ptight} which plays an important role 
in later sections.

Throughout the rest of this paper, $\varpi\in \bt(\N)\setminus \N$ is 
a fixed free ultrafilter. 
{\em Any Choquet simplex in this paper is a metrizable Choquet simplex.}
 If $T$ is a convex set, then $\eT$ is its extremal boundary, the subset of 
extremal points of $T.$

The following proposition is known. 
The  main point of it  is,  perhaps,  that we do not assume  that $F$ (or $K$) is closed
(note that the quotient space in (2) of the proposition may not be complete).

\begin{prop}\label{PIdeal}
Let $A$ be a separable \CA\, and $F$ be a non-empty subset of 
faithful quasi-traces in $QT_{(0,1]}(A).$ 
Then 

(1) $I_{_{F, \varpi}}$
is closed in $l^\infty(A).$ 

(2) If $K\subset F,$ then $(I_{_{K, \varpi}}/I_{_{F, \varpi}}, \|\cdot \|_{_{2, F_\varpi}})$  is closed in $(l^\infty(A)/I_{_{F, \varpi}},  \|\cdot \|_{_{2, F_\varpi}}).$
\end{prop}

\begin{proof}
Let us prove (2) only. An obvious modification of the proof  implies  that (1) also holds.

Let $x^{(k)}=\{x_n^{(k)}\}_{n\in \N}$ ($k\in \N$) be a sequence of elements in 
$I_{_{K, \varpi}}$ and $x=\{x_n\}\in l^\infty(A)$ such that
\beq
\lim_{k\to\infty}\|\Pi_{_{F, \varpi}}(\{x_n^{(k)}\}-x)\|_{_{2, F_\varpi}}=0.
\eneq
To see that $x\in I_{_{K, \varpi}},$
let $\ep>0.$  There is $k_0\in \N$ such that, for all $k\ge k_0,$ 
\beq
\lim_{n\to \varpi}\sup\{\|x_n^{(k)}-x_n\|_{_{2, \tau}}: \tau\in F\}<\ep/2.
\eneq
%
%
There is (for each $k\ge k_0$)  ${\cal P}_k\in  \varpi$ such that, for all $n\in {\cal P}_k,$
\beq
\sup\{\|x_n^{(k)}-x_n\|_{_{2, \tau}}: \tau\in F\}<\ep/2.
\eneq
Since $\{x_n^{(k)}\}_{n\in \N}\in I_{_{K, \varpi}},$ there is also ${\cal  Q}_k\in \varpi$ such that, for all $n\in {\cal Q}_k,$
\beq
\sup\{\|x_n^{(k)}\|_{_{2, \tau}}: \tau\in K\}<\ep/2.
\eneq
Let ${\cal R}={\cal P}_{k_0}\cap {\cal Q}_{k_0}\in \varpi.$ 
If $n\in {\cal R}$ (see \cite[Lemma 3.5]{Haagtrace} and also \cite[Definition 2.16]{FLosc}),
\beq
&&\hspace{-0.5in}\sup\{\|x_n\|^{2/3}_{_{2, \tau}}: \tau\in K\}\le \sup\{\|x_n^{(k_0)}\|^{2/3}_{_{2, \tau}}: \tau\in F\}+
\sup\{\|x_n^{(k_0)}-x_n\|^{2/3}_{_{2, \tau}}: \tau\in F\}\\
&&\hspace{0.95in}<(\ep/2)^{2/3}+(\ep/2)^{2/3}\le \ep^{2/3}.
\eneq
It follows that $x\in I_{_{K, \varpi}}.$
\end{proof}

\begin{prop}{\rm (cf. \cite[Proposition 1.1]{CETW})}\label{Pcompact}
Let $A$ be a separable simple \CA\, and $K\subset QT(A)$ be a non-empty compact subset. 
Then $l^\infty(A)/I_{_{F, \N}}$ and $l^\infty(A)/I_{_{F, \varpi}}$ are unital for any non-empty set $F
\subset K.$
\end{prop}

\begin{proof}
Let $\{e_n\}\subset A_+^{\bf 1}$ be an approximate identity for $A.$
Then $\tau(e_n)\nearrow 1$ on $QT(A).$ By the Dini theorem, 
\beq
\lim_{n\to\infty} \sup\{1-\tau(e_n): \tau\in K\}=0.
\eneq
Let $e=\{e_n\}.$ Then, for any $x=\{x_n\}\in l^\infty(A),$ working in $\wtd A,$  one obtains that
\beq
\|e_nx_n-x_n\|_{_{2, K}}\le \|x_n\|\sup\{1-\tau(e_n): \tau\in K\}^{1/2}\to 0.
\eneq
It follows that $ex-x\in I_{_{K,\N}}.$ In other words, $\Pi_{_{F, \N}}(e)$ is the unit of $l^\infty(A)/I_{_{K, \N}}.$
Hence its quotient $l^\infty(A)/I_{_{K, \varpi}}$ and $l^\infty(A)/I_{_{F, \varpi}}$ are all unital.
\end{proof}

The following is a non-unital version  of \cite[Proposition 9.3]{Linclr1}. Since lately it has been used for non-unital 
cases, for clarity let us state it in the case that $A$ is algebraically simple.  The proof is  exactly the same as 
that of \cite[Proposition 9.3]{Linclr1}, which is based on \cite{CP}.

\begin{prop}[{\rm \cite[Proposition 9.3]{Linclr1}}]\label{P93}
Let $A$ be a separable algebraically simple \CA\, with $\wtd T(A)\setminus \{0\}\not=\emptyset.$
Then, for any $f\in \Aff(\Tw)$ and any $\ep>0,$ there exists $a\in A_{s.a.}$ with $\|a\|\le \|f\|+\ep$ 
such that $\tau(a)=f(\tau)$ for all $\tau\in \Tw.$ If $f(\tau)>0$ for all $\tau\in \Tw,$ then  one may choose $a\in A_+.$ 
\end{prop}

\begin{proof}
Fix $f\in \Aff(\Tw).$ By 
Proposition 2.7 and 2.8 of \cite{CP} 
as in the proof 
of \cite[Proposition 9.3]{Linclr1}, $f$ gives an element in $(A^q)^{**}$
(see \cite{CP} for $A^q$).
Since $f$ is weak*-continuous, there is ${\bar b}\in A^q$ such that ${\bar b}(\tau)=f(\tau)$
for all $\tau\in \Tw$ and $\|{\bar b}\|\le \|f\|+\ep/2.$ 
 Thus, 
 one obtains $b\in A_{s.a.}$ such that $\tau(b)=f(\tau)$ for all $\tau\in \Tw$ and $\|b\|\le \|f\|+\ep.$
 
It is perhaps the second part that is more interesting.  Note first that,
since $A$ is algebraically simple,  by \cite[Lemma 4.5]{eglnp}, $0\not\in \Tw.$ 
Suppose that $f(\tau)>0$ for all $\tau\in \Tw.$ It follows from \cite[Corollary 6.4]{CP} that 
there exists $x\in A_+\setminus \{0\}$ and $z\in A_0$ (notation in \cite{CP}) such that $b=x+z$
(using the fact that  $A$ is algebraically simple again).
In other words, $\tau(b)=\tau(x)=f(\tau)$ for all $\tau\in \Tw.$ By \cite[Proposition 2.9]{CP}, 
there are $a, y\in A_+$ such that $x\sim a\le y$ ($\sim$ as in \cite{CP}) and $\|y\|\le \|f\|+\ep.$ 
So $\|a\|\le \|y\|\le \|f\|+\ep.$ But $\tau(a)=\tau(b)=f(\tau)$ for all $\tau\in \Tw.$  The proposition follows.
\end{proof}

\begin{lem}\label{Lincreasing}
Let $A$ be a separable algebraically simple \CA\, with non-empty compact $T(A).$ 
Suppose that $A$ has strict comparison and $\Gamma: \Cu(A)\to {\rm LAff}_+(T(A))$ is surjective 
(see Definition \ref{DGamma}) and 
$g_n
\in \Aff(T(A))$ \st\,    

(1) $g_n(t)>0$  for all $t\in T(A)$ and $n\in \N,$ 

(2) $\sum_{n=1}^\infty g_n(t)<1$ for all $t\in T(A),$ and 

(3) $\inf\{1-\sum_{n=1}^\infty g_n(t): t\in T(A)\}=\sigma>0.$

Then,  for any $\ep>0,$ there is a sequence of  mutually orthogonal elements $\{a_n\}\subset A_+^{\bf 1}$ \st\, $\tau(a_n)<g_n(\tau)$ 
for all $\tau\in T(A)$ ($n\in \N$), and 
\beq\label{Lincreasing-e1}
\sup\{|\tau(a_n)-g_{n}(\tau)|: \tau\in T(A)\}<\ep\cdot\min\{\|g_n\|, \sigma\}/2^n.
\eneq
%
\end{lem}

\begin{proof}
Since $\Gamma$ is surjective,
there is, for each $n\in \N,$ 
$e_n'\in (A\otimes {\cal K})_+^{\bf 1}$
\st\,  $d_\tau(e_n')=g_n(\tau)$ for all $\tau\in T(A).$ 
Since $d_\tau(e_1')$ is continuous on $T(A),$ 
one may choose $0<\dt_1'<\ep\|g_1\|/2^5$   such that
\beq\label{628-n1}
&&d_\tau(e_1')-\tau(f_{\dt_1'}(e_1'))<\ep\cdot \min\{\|g_1\|, \sigma\}/2^5\rforal \tau\in T(A).
\eneq
Put $e_1''=f_{\dt_1'/2}(e_1').$ 
Note that 
$d_\tau(e_1'')=d_\tau(f_{\dt_1'/2}(e_1'))\le d_\tau(e_1')$ for all $\tau\in T(A).$
Choose $0<\dt_1<\dt_1'/2$   \st\, 
\beq
\sup\{|\tau(f_{\dt_1}(e_1''))-d_\tau(e_1')|: \tau\in T(A)\}<\ep\{\|g_1\|, \sigma\}/2^4
\eneq
(see \eqref{628-n1}) and a strictly positive element $a_1^\perp$ of 
$B_1= \{a\in A: af_{\dt_1}(e_1'')=f_{\dt_1}(e_1'')a=0\}.$
Note that $(1-f_{\dt/2}(e_1''))\perp f_\dt(e_1'').$ 
Thus $1-f_{\dt/2}(e_1'')\lesssim a_1^\perp.$ 
Hence 
for any $0<\dt<\dt_1$ and  for any $\tau\in T(A),$ 
\beq\nonumber
d_\tau(f_{\dt/2}(e_1''))+d_\tau(a_1^\perp) &\ge& d_\tau(f_\dt(e_1''))+d_\tau(a_1^\perp)\ge 1-\tau(f_{\dt/2}(e_1''))+d_\tau(f_{\dt}(e_1''))\\\nonumber
&>&1-(d_\tau(e_1'')-d_\tau(f_{\dt}(e_1''))\ge  1-(d_\tau(e_1')-\tau(f_{\dt}(e_1''))\\
&>&1-\ep\cdot \min\{\|g_1\|, \sigma\}/2^4
\eneq
(see \cite[Proposition 2.24]{Linoz}).
Therefore, we may assume that (also by (3))
\beq\label{Lincreasing-e2}
d_\tau(a_1^\perp)>1-g_1(t)-\ep\cdot\min\{\|g_1\|, \sigma\}/2^3>g_2(t) \rforal \tau\in T(A).
\eneq
Define $a_1=f_{\dt_1}(e_1'').$  Then \eqref{Lincreasing-e1} holds for $n=1.$ 

Since $A$ has strict comparison, applying \eqref{Lincreasing-e2} (and, (2)  above and \eqref{Lincreasing-e2})
we 
choose $0<\dt_2'<\ep\cdot \|g_1\|/2^{2+5}$  and $x_2\in (A\otimes {\cal K})$ (see \cite[Proposition 2.4]{Rr2}) \st\, 
\beq\nonumber
d_\tau(e_2')-\tau(f_{\dt_2'}(e_2'))<\ep\cdot \min\{\|g_2\|, \sigma\}/2^{2+5},\,\,x_2^*x_2=f_{\dt_2'/2}(e_2'),\andeqn x_2x_2^*\in {B_1}_+^{\bf 1}.
\eneq
Put $e_2''=x_2x_2^*.$ Note that 
$$
\sup\{|\tau(e_2'')-d_\tau(e_2')|: \tau\in T(A)\}<\ep\{\|g_2\|, \sigma\}/2^{2+5}.
$$
Choose $0<\dt_2<\dt_2'$  \st\, 
\beq\label{7-2-n1}
\sup\{|\tau(f_{\dt_2}(e_2''))-d_\tau(e_2')|: \tau\in T(A)\}<\ep\{\|g_2\|, \sigma\}/2^{2+4}
\eneq
and a strictly positive element $a_2^\perp$ of 
$B_2= \{a\in B_1: af_{\dt_2}(e_2'')=f_{\dt_2}(e_2'')a=0\}.$
Applying \cite[Proposition 2.24]{Linoz} again, we may assume that,  for all $\tau\in T(A),$ by \eqref{7-2-n1},
\beq\label{Lincreasing-e3}
d_\tau(f_{\dt_2/2}(e_2''))+d_\tau(a_2^\perp)\ge d_\tau(a_1^\perp)-(d_\tau(f_{\dt_2/4}(e_2'')-d_\tau(f_{\dt_2/2}(e_2'')))\\
> d_\tau(a_1^\perp)-(d_\tau(e_2'))-d_\tau(f_{\dt_2/2}(e_2'')))>
d_\tau(a_1^\perp)-\ep\{\|g_2\|, \sigma\}/2^{2+4}.
\eneq
Thus  (recall that $g_2(t)=d_t(e_2')$)
\beq
d_\tau(a_2^\perp)>1-(g_1(t)+g_2(t))-\ep\cdot(\min\{\|g_1\|, \sigma\}/2^3+\min\{\|g_2\|, \sigma\}/2^{2+3})
\eneq
for all $\tau\in T(A).$
Define $a_2=f_{\dt_2}(e_2'').$  Note that $a_1\perp a_2.$ Then \eqref{Lincreasing-e1} holds for $n=2.$ Note also that 
 $a_1+a_2\perp a_2^\perp.$
The lemma then follows from the induction.

\end{proof}

\begin{rem}
If $A$ is of stable rank one, then the proof of Lemma \ref{Lincreasing} could be much simplified. 
On the other hand, for each $n\in \N,$ by Lemma \ref{P93}, there is $a_n\in A_+^{\bf 1}$ such that $\tau(a_n)=g_n(\tau)$ for all $\tau\in T(A).$
One could choose  $b_n=\sum_{j=1}^n a_j.$ 
Then $\tau(b_n)=\sum_{j=1}^n g_j(\tau)$ for all $\tau\in T(A).$   In particular, $b_n\le 
b_{n+1}.$ 
However, we would  not have the control of $\|b_n\|$ as 
in Lemma \ref{Lincreasing} (see the proof of  Lemma \ref{Lsperation2}). So it is important that $\{a_n\}$ are mutually orthogonal.
%
\end{rem}

The unital case of the next proposition was obtained in \cite{S-2} and \cite{TWW-2}.
We state a non-unital version of it for  convenience. The proof 
is a straightforward modification of
that of 
\cite[Proposition 4.1]{S-2} and 
a similar modification of that of \cite[Lemma 3.3]{TWW-2} also works.

\begin{prop}[{\rm cf.\,\,\cite[Proposition 4.1]{S-2} and \cite[Lemma 3.3]{TWW-2}}]\label{Tsato-1}
Let $A$ be a separable  amenable algebraically simple \CA\, with nonempty compact $T(A).$ 
For any  non-negative function $f\in \Aff(T(A)),$  there exists $\{a_n\}\in \pi_\infty^{-1}(A')_+$  such 
that $\|a_n\|\le \|f\|$ and 
\beq
\lim_{n\to\infty}\sup\{|\tau(a_n)-f(\tau)|:\tau\in T(A)\}=0.
\eneq
\end{prop}

\begin{proof}
We may assume that $f\not=0.$
Then $f+1/2^n$ is strictly positive for all $n\in \N.$
It follows from  Proposition \ref{P93}
 that, for each $n\in \N,$  there is $b_n\in A_+$ with 
$\|b_n\|\le \|f\|+1/2^n$ such that 
$\tau(b_n)=f(\tau)+1/2^{n+1}$ for all $\tau\in  T(A).$ 
Consider $B=\wtd A.$ Then  $B$ is a separable unital \CA\, with 
\beq
T(B)=\{\af \tau+(1-\af)\tau_B: \tau\in T(A), \, \af\in [0,1]\},
\eneq
where $\tau_B$ is the unique tracial state which vanishes on $A.$
It follows from \cite[Corollary 3.3]{S-2} that there exists $\{d_n^{(k)}\}\in \pi_\infty^{-1}(B')_{s.a.}$
such that $\|d_n^{(k)}\|\le \|b_k\|$ and 
\beq
\tau(d_n^{(k)})=\tau(b_k)\rforal \tau\in T(B).
\eneq
Since $b_k\in A,$ $\tau_B(b_k)=0$ for all $k\in \N.$ It follows that 
$\tau_B(d_n^{(k)})=0$ for all $n, k\in \N.$ In other words, $d_n^{(k)}\in A$ for 
all $n, k\in \N.$  
Since $A$ is separable, by taking an appropriate subsequence of $\{d_n^{(k)}\}_{n,k},$
we obtain $\{a_n\}\in \pi_\infty^{-1}(A')_+$ such that 
\beq
&&\lim_{n\to\infty}(\|f\|/\|a_n\|)=1\andeqn\\
&&\lim_{n\to \infty}\sup\{|\tau(a_n)-f(\tau)|:\tau\in T(A)\}=0.
\eneq
Since $f\not=0,$ we may assume that $\|a_n\|\not=0$ for all $n\in \N.$
Replacing $a_n$ by $(\|f\|/\|a_n\|) a_n,$ we may also assume that $\|a_n\|\le \|f\|.$
\end{proof}

%

The next two propositions are folklores.

\begin{prop}\label{Punion}
Let $A$ be a \CA\, and $S_1, S_2$ be subsets of $T_{[0,1]}(A).$
Then, for any $x=\{x_n\}\in l^\infty(A),$
\beq
\|\Pi_{_{S_1\cup S_2, \varpi}}(x)\|_{_{2, (S_1\cup S_2)_\varpi}}=\max\{\|\Pi_{_{S_1, \varpi}}(x)\|_{_{2, {S_1}_\varpi}},
 \|\Pi_{_{S_2, \varpi}}(x)\|_{_{2, {S_2}_\varpi}}\}.
\eneq
\end{prop}

\begin{proof}
Let $a_n=\|x_n\|_{_{2, S_1\cup S_2}},$ 
$b_n=\|x_n\|_{_{2, S_1}}$ and $c_n=\|x_n\|_{_{2, S_2}},$ $n\in \N.$
Put
$\af=\lim_{n\to\varpi}a_n,$ $\bt=\lim_{n\to\varpi}b_n$ and $\gamma=\lim_{n\to\varpi}c_n.$
Note that $a_n=\max\{b_n, c_n\}.$
For any $\ep>0,$ there are ${\cal P}_a, {\cal P}_b, {\cal P}_c\in\varpi$ such that
\beq
&&|a_n-\af|<\ep\rforal n\in {\cal P}_a,\,\, |b_n-\bt|<\ep\rforal n\in {\cal P}_b\\
 &&\andeqn |c_n-\gamma|<\ep \rforal n\in {\cal P}_c.
\eneq
Choose ${\cal Q}={\cal P}_a\cap {\cal P}_b\cap {\cal P}_c\in \varpi.$ 
Then, for any $n\in {\cal Q},$ 
$|\max\{b_n, c_n\}-\max\{\bt, \gamma\}|<2\ep.$
It follows that
\beq
|\af-\max\{\bt, \gamma\}|<3\ep.
\eneq
Since $\ep$ is arbitrary, this implies that $\af=\max\{\bt, \gamma\}.$ 
The proposition follows.
\end{proof}

\begin{prop}\label{Pnorm2}
Let $A$ be a \CA\, and $\{I_\lambda: \lambda\in \Lambda\}$ be a family 
of ideals. Put $I=\cap_{\lambda\in \lambda}I_\lambda.$
Denote by $\pi_\lambda: A\to A/I_\lambda$ and 
$\pi_I: A\to A/I$ the quotient maps, respectively.
Then, for any $a\in A,$
\beq\label{Pnorm2-e}
\|\pi_I(a)\|=\sup\{\|\pi_{I_{\lambda}}(a)\|: \lambda\in \Lambda\}.
\eneq
\end{prop}

\begin{proof}
Let $a\in A.$ Since $A$ is a \CA, it suffices to show that \eqref{Pnorm2-e} holds for $a\in A_{s.a}.$

Put $X={\rm sp}(a).$ Let $B=C^*(a)$ be the commutative \SCA\, generated by $a.$ 
Let $J_\lambda=I_\lambda\cap B$ and $J=I\cap B.$
Denote by $F$ and $F_\lambda$ the closed subsets of $X$ 
associated with $J$ and $J_\lambda$ in such a way that
${\rm sp}(\pi_I(a))=F$ and ${\rm sp}(\pi_\lambda(a))=F_\lambda.$ 
Then $F=\overline{\cup F_\lambda}.$ 

Note that
\beq
\|\pi_I(a)\|=\sup\{|t|: t\in F\}\andeqn \|\pi_\lambda(a)\|=\sup\{|t|: t\in F_\lambda\}.
\eneq
It follows that
\beq
\|\pi_I(a)\|=\sup\{\|\pi_\lambda(a)\|: \lambda\in \Lambda\}.
\eneq
\end{proof}

The next proposition may hold in greater generality. Recall, by Proposition \ref{PIdeal},  that
$(I_{_{F_1, \varpi}}/I_{_{F_2, \varpi}},\, \|\cdot \|_{_{2, {F_2}_\varpi}})$ 
is closed in $(l^\infty(A)/I_{_{F_2, \varpi}},\, \|\cdot \|_{_{2, {F_2}_\varpi}}).$

\begin{prop}\label{Pquotientnorm}
Let $A$ be a separable 
 simple \CA\, with non-empty compact $T(A).$
Suppose that $F_1\subset  
F_2\subset 
\partial_e(T(A))$   are subsets. 

(1) Then, for any $x\in l^\infty(A)/I_{_{F_2, \varpi}},$
\beq
&&\|\pi(x)\|_{_{2,{F_1}_\varpi}}\le \inf \{\|x+j\|_{_{2, {F_2}_{\varpi}}}:\, j\in I_{_{F_1, \varpi}}/I_{_{F_2, \varpi}}\},
\eneq
where 
$\pi: l^\infty(A)/I_{_{F_2, \varpi}}\to l^\infty(A)/I_{_{F_1, \varpi}}$ is the quotient map.
%

(2) If, in addition, $A$ is algebraically simple \CA\, 
and 
$F_1\subset 
F_2\subset \partial_e(T(A))$ are compact, 
then
$(l^\infty(A)/I_{_{F_1, \varpi}}, \|\cdot\|_{2, {F_1}_\varpi})$
is the quotient normed space of the normed space\\ $(l^\infty(A)/I_{_{F_2, \varpi}}, \|\cdot \|_{_{2, {F_2}_{\varpi}}}),$
i.e.
\beq
\|\pi(x)\|_{_{2,{F_1}_\varpi}}=\inf \{\|x+j\|_{_{2, {F_2}_{\varpi}}}:\, j\in I_{_{F_1, \varpi}}/I_{_{F_2, \varpi}}\}
\eneq
for all $x\in l^\infty(A)/I_{_{F_2, \varpi}}.$ 
\end{prop}

\begin{proof}
Fix $x\in l^\infty(A)/I_{_{F_2, \varpi}}.$ We may write 
$x=\Pi_{_{F_2, \varpi}}(\{a_k\}),$ where $\{a_k\}\in l^\infty(A).$ 
%

For (1), let us fix $j\in I_{_{F_1, \varpi}}/I_{_{F_2, \varpi}}$
and $\ep>0.$
Let $\{j_k\}\in I_{_{F_1, \varpi}}$ be such that $\Pi_{_{F_2, \varpi}}(\{j_k\})=j.$
There exists ${\cal P}\in \varpi$ such that, for all $k\in {\cal P},$
\beq
\tau(j_k^*j_k)<\ep/3, \,\, |\tau(j_k^*a_k)|<\ep/3\andeqn |\tau(a_k^*j_k)|<\ep/3
\eneq
for all $\tau\in F_1.$
It follows that, for $k\in {\cal P},$
\beq\nonumber
\sup\{\|a_k+j_k\|_{_{2,\tau}}^2:\tau\in F_1\}&=&\sup\{\tau((a_k+j_k)^*(a_k+j_k)):\tau\in F_1\}\\\nonumber
&=&
\sup\{\tau(a_k^*a_k+j_k^*j_k+j_k^*a_k+a_k^*j_k): \tau\in F_1\}\\\nonumber
&>&\sup\{\tau(a_k^*a_k): \tau\in F_1\}-\sup\{\tau(j_k^*j_k): \tau\in F_1\}\\\nonumber
&&\hspace{0.4in}-\sup\{\tau(j_k^*a_k): \tau\in F_1\}-\sup\{\tau(a_k^*j_k): \tau\in F_1\}\\\nonumber
&>&\sup\{\tau(a_k^*a_k):\tau\in F_1\}-\ep.
\eneq
This implies that
\beq
\|\pi(x)\|_{_{2, {F_1}_\varpi}}^2 &\le &
\lim_{k\to\varpi}\sup\{\|a_k+j_k\|_{_{2,\tau}}^2:\tau\in F_1\}+\ep\\
&\le &\lim_{k\to\varpi}\sup\{\|a_k+j_k\|_{_{2,\tau}}^2:\tau\in F_2\}+\ep\\
&=&\|x+j\|_{_{2, {F_2}_\varpi}}^2+\ep.
\eneq
Let $\ep\to 0.$ We obtain, for any fixed $j\in I_{_{F_1, \varpi}}/I_{_{F_2, \varpi}},$ 
\beq
\|\pi(x)\|_{_{2, {F_1}_\varpi}}\le \|x+j\|_{_{2, {F_2}_\varpi}}.
\eneq
Hence
\beq\label{Pqn-e-22}
\|\pi(x)\|_{_{2, {F_1}_\varpi}}\le \inf\{\|x+j\|_{_{2, {F_2}_\varpi}}: j\in  I_{_{F_1, \varpi}}/I_{_{F_2, \varpi}}\}.
\eneq
%
This proves (1).

For (2), we  now recall that $T(A)$ is assumed to be compact and 
$A$ is algebraically simple. So it is a Choquet simplex (see the second last paragraph  of \ref{DGamma}).

Define, for each $n\in \N,$  a relatively open  subset
$O_n=\{t\in \partial_e(T(A)): {\rm dist}(t, F_1)<\ep_n\}.$ 
where $\ep_n\searrow 0.$ 
Since $\widehat{a_k^*a_k}$ is continuous on $T(A),$ for each $k,$ 
there is $n(k)\in \N$ such that
\beq\label{Pqn-ee}
\sup\{\tau(a_k^*a_k): \tau\in O_{n(k)}\}\le \sup\{\tau(a_k^*a_k): \tau\in F_1\}+1/k^2.
\eneq
There is, for each $k\in \N,$ a function
$f_k'\in C(F_2)_+^{\bf 1}$
 such that ${f_k'}|_{F_1}=0$ and 
${f_k'}|_{F_2\setminus O_{n(k)}}=1.$
It follows from  \cite[Theorem II. 3.12]{Alf} that there is $f_k\in \Aff(T(A))$ \st\, 
$0\le f_k\le 1$ and ${f_k}|_{F_2}=f_k'.$

It follows from  Lemma \ref{P93}
that there is, for each $n\in \N,$ a sequence $\{b_{n,k}\}_{k\in \N}\in l^\infty(A)$
with $0\le b_{n,k}\le 1$  such that  (considering $f_k+1/2k^3$)
\beq\label{Pqn-e-10}
\sup\{|\tau(b_{n,k})-f_n(\tau)|: \tau\in T(A)\}<1/k^2.
\eneq
Define $b=\Pi_{_{F_2, \varpi}}(\{b_{n,m(n)}\}).$
Note that $0\le b\le 1.$   Moreover (recall that ${f_k}|_{F_1}=0$),
\beq\nonumber
\sup\{\tau(b_{k,m(k)}): \tau\in F_1\}\le 
\sup\{\tau(f_k): \tau\in F_1\}+\sup\{|\tau(b_{k,m(k)})-f_k(\tau)|: \tau\in F_1\}
<{1\over{m(k)^2}}.
\eneq
It follows that $b\in I_{_{F_1, \varpi}}/I_{_{F_2, \varpi}}.$

Put $c=(1-b)^{1/2}.$
Then $\Pi_{_{F_1, \varpi}}(xc)=\Pi_{_{F_1, \varpi}}(x).$
In other words, there is $j\in I_{_{F_1, \varpi}}/I_{_{F_2, \varpi}}$ such that 
$xc=x+j.$ 

Put $S_k=F_2\setminus O_{n(k)}.$  Note that ${f_k}|_{S_k}=1.$ Then, for each $k\in \N,$
\beq\nonumber
&&\hspace{-0.6in}\sup\{\|a_k(1-b_{k,m(k)})^{1/2}\|_{_{2, \tau}}^2: \tau\in S_k\}
\le \|x\|^2 \sup\{\|(1-b_{k,m(k)})^{1/2}\|_{_{2, \tau}}^2: \tau\in S_k\}\\\label{Pqn-e++}
&&\hspace{0.2in}<\|x\|^2(1/k^2+\sup\{(1-f_k)|_{S_k}:\tau\in S_k\})\le \|x\|^2/k^2.
\eneq
Also, for each $k\in \N,$ by \eqref{Pqn-ee},
\beq\nonumber
\hspace{-0.4in}\sup\{\|a_k(1-b_{k, m(k)})^{1/2}\|_{_{2, \tau}}^2: \tau\in O_{n(k)}\}&\le& 
\|c\|\sup\{\|a_k\|_{_{2, \tau}}^2:\tau\in O_{n(k)}\}\\
&\le & \sup\{\|a_k\|_{_{2, \tau}}^2: \tau\in F_1\}+1/k^2.
\eneq
It follows from Proposition \ref{Punion},  for each $k\in \N,$
\beq\label{Pqn-e+3}
&&\hspace{-0.4in}\sup\{\|a_k(1-b_{k, m(k)})^{1/2}\|_{_{2, \tau}}^2: \tau\in F_2\}\le \\
&&\max\{ \sup\{\|a_k\|_{_{2, \tau}}^2: \tau\in F_1\}+1/k^2,\,\|x\|^2/k^2\}.
\eneq
Hence, by 
\eqref{Pqn-e+3},
\beq\nonumber
\hspace{-0.2in}\|xc\|_{_{2, {F_2}_\varpi}}^2&=&
\lim_{k\to \varpi}\sup\{\|a_k(1-b_{k, m(k)})^{1/2}\|^2_{_{2,\tau}}:\tau\in F_2\}\\\nonumber
&\le & \lim_{k\to\varpi}\max\{ \sup\{\|a\|_{_{2, \tau}}^2: \tau\in F_1\}+1/k^2,\,\|x\|/k^2\}
=\|\pi(x)\|_{_{2, {F_1}_\varpi}}^2.
%
\eneq
It follows 
that
\beq\label{Pqn-e-21}
\inf\{\|x+j\|_{_{2, {T(A)}_\varpi}}: j\in \Pi_{_{T(A), \varpi}}(I_{_{F_1,\varpi}})\}\le \|\pi(x)\|_{_{2, {F_1}_\varpi}}.
\eneq
Hence part (2) follows by also applying part (1). 
%
%
%
%
%
\end{proof}

\begin{lem}\label{Lunifdini}
Let $\DT$ be a Choquet simplex, $F\subset \eD$ be a compact subset and $F\subset O$
be a relatively open subset of $\eD.$
Suppose that 
$0<\ep<1/4.$  
Then there exists a sequence of functions $f_n\in \Aff(\DT)_+^{\bf 1}$ 
such that 

{\rm (i)} $f_n(t)<f_{n+1}(t)$ for all $t\in \eD,$ 

{\rm (ii)} $0<f_n(t)\le \sum_{j=1}^{n} \ep^2/2^j$ for all 
$t\in F,$ 
 
{\rm (iii)}  $\lim_{n\to \infty}f_n(t)\ge 1-\ep/2$ for all $t\in \eD\setminus O,$ 
and 

{\rm (iv)} $1-f_n(t)\ge \ep/4$ for all $t\in \eD.$
\end{lem}

\begin{proof}
Let $G=\partial_e(\Delta)\setminus F.$   We first show that
there is an increasing  sequence $\{g_n\}\subset \Aff(\Delta)$ 
such that  ${g_n}|_F=0$ and $\lim_{n\to\infty} g_n(x)\ge 1-\ep/8$ for all $t\in G.$ 

Choose $x\in G.$ By \cite[Theorem 11.12]{kG}, there exists 
$g_x\in \Aff(\Delta)$ such that
\beq
0\le  g_x\le 1,\, g_x(x)=1\andeqn g_x|_F=0.
\eneq 
Fix $0<\ep<1/4.$ There exists $\dt_x>0$ such that, if ${\rm dist}(x, t)<\dt_x,$
$
g_x(t)>1-\ep/8.
$
Since $G$ is second countable, there are $\{x_n\}\subset G$ such that
$
\bigcup_{n=1}^\infty O(x_n, \dt_{x_n})\supset G,
$
where $O(x_n, \dt_{x_n})=\{x\in G: {\rm dist}(x, x_n)<\dt_{x_n}\},$ and 
there is, for each $n\in \N,$  $g_{x_n}\in \Aff(\Delta)$ such that $0\le g_{x_n}\le 1,$ $g_{x_n}(t)>1-\ep/8$ 
for $t\in O(x_n,\dt_{x_n})$ 
and ${g_{x_n}}|_{F}=0.$
Define 
\beq\label{Lfacen-f0}
g_0=\sum_{n=1}^\infty g_{x_n}/2^n.
\eneq
Then 
\beq
g_0\in \Aff(\Delta),\,\,\, 0\le g_0\le 1,\,\,\,g_0(t)>0 \rforal t\in G\andeqn g_0|_F=0. 
\eneq
Define $g_1=f_{x_1}.$ 
Then 
\beq
g_1\vee g_{x_2}\le 2^2 g_0\wedge 1.
\eneq
By \cite[Corollary II. 3.11]{Alf}, there exists $g_2\in \Aff(\Delta)$ such that
\beq
g_1\vee g_{x_2}\le g_2\le 4g_0\wedge 1.
\eneq
Then (since ${g_0}|_F=0$)
\beq
g_{x_1}\vee g_{x_2}\le g_2\andeqn {g_2}|_{F}=0.
\eneq

Suppose that we have constructed $g_1, g_2,...,g_n\in \Aff(\Delta)$ such that
\beq
0\le g_{j-1}\le g_j\le 1, \, g_1\vee g_2\vee  g_{j-1}\vee g_{x_j}\le g_j\le 2^j g_0\wedge 1, \,\,\, 1\le j\le n.
\eneq
Hence 
\beq
g_n|_F=0\andeqn g_n(t)\ge 1-\ep/8 \rforal t\in \cup_{i=1}^nO(x_i, \dt_x)
\eneq
Note that  (recall that $g_n\le 2^n g_0$)
\beq
g_n\vee g_{x_{n+1}}\le 2^{n+1} g_0\wedge 1.
\eneq
Applying \cite[Corollary II. 3.11]{Alf}, we obtain $g_{n+1}\in \Aff(\Delta)$ such that
\beq
g_n\vee g_{x_{n+1}}\le g_{n+1}\le 2^{n+1} g_0\wedge 1.
\eneq
Thus we constructed an increasing sequence $\{g_n\}$ in $\Aff(\Delta)$
such that
\beq
0\le g_{x_1}\vee g_{x_2}\vee \cdots \vee g_{x_n}\le g_n\le 1\andeqn {g_n}|_F=0.
\eneq
Note since $g_n(t)\ge 1-\ep/8$ for all $t\in \cup_{j=1}^n O(x_j,\dt_j),$ 
we conclude that $\lim_{n\to\infty} g_n(t)\ge 1-\ep/8$ for all $t\in G.$

Now define $f_n=(1-3\ep/8)g_n+\sum_{i=1}^n \ep^2/2^{j+1}.$  
Then (i), (ii),  (iii)  and (iv) follow. 
%
\end{proof}

\begin{lem}\label{Lsperation2}
Let $A$ be a separable algebraically simple \CA\, with non-empty and compact $T(A)$ and 
$\sigma$-compact $\partial_e(T(A)).$  
Suppose that $A$ has strict comparison and $\Gamma$ is surjective.
Suppose also  that  $F\subset \partial_e(T(A))$ is a compact subset and $a\in A^{\bf 1}.$
Then, for any $\ep>0,$ there exists $c\in A_+^{\bf 1}$ 
\st\, $\tau(c)<\ep$ for all $\tau\in F$ and 
\beq
\|a(1-c)\|_{_{2,T(A)}}<\|a\|_{_{2, F}}+\ep.
\eneq
\end{lem}

\begin{proof}
Recall that in this case, $T(A)$ is a metrizable Choquet simplex
(see the second last  
part of  Definition \ref{DefOS1}). 

Fix $1/2>\ep>0.$ 
Put $\af=\|a\|_{_{2, F}}+\ep/2.$  Since $\widehat{a^*a}(t)=t(a^*a)$ ($t\in T(A)$) is a continuous function on $T(A),$  there is a relatively open subset $O\subset \partial_e(T(A))$ 
with $F\subset O$ \st\, 
\beq
\|a\|_{_{2, \bar O}}\le \af.
\eneq
Applying Lemma \ref{Lunifdini}, 
we obtain a sequence $\{f_n\}\subset \Aff(T(A))_+$ \st\,  $0<f_n(t)<f_{n+1}(t)<1$ for all $t\in T(A)$ and $n\in \N,$ and 
\beq\label{6-20-n1}
&&0<f_n(t)<\ep^2/4 \rforal t\in F,\,\, \inf\{1-f_n(\tau): \tau\in T(A)\}>\ep/4
\andeqn\\\label{Lsperation2-3}
 &&\lim_{n\to\infty} f_n(t)\ge 1-\ep/2\rforal t\in \partial_e(T(A))\setminus O.
\eneq
By Lemma \ref{Lincreasing}, there are mutually orthogonal  elements $c_n\in A_+^{\bf 1}$ \st\, \\
$\tau(c_n)\le f_n(\tau)-f_{n-1}(\tau)$ ($f_0=0$),  and 
\beq
\sup\{|\tau(c_n)-(f_n(\tau)-f_{n-1}(\tau))|: \tau\in T(A)\}<\ep\min\{\|f_n\|,\ep/2\}/16\cdot 2^n,\,\,\, n\in \N.
\eneq 
Put $b_n=\sum_{j=1}^n c_j,$ $n\in \N.$ Since $c_i\cdot c_j=0,$ if $i\not=j,$  we have
$0\le b_n\le 1.$ 
Then
$b_n\le b_{n+1}\le 1,$ and 
\beq
\sup\{|\tau(b_n)-f_n(\tau)|: \tau\in T(A)\}<\ep\min\{\|f_n\|,\ep/2\}/16,\,\,\, n\in \N.
\eneq 
It follows from \eqref{6-20-n1} that, for $\tau\in F,$ 
\beq\label{6-20-n2}
\tau(b_n)<\ep^2/4+\ep^2/32.
\eneq
Moreover (see also \eqref{Lsperation2-3}), for $t\in \partial_e(T(A))\setminus O,$  
\beq
\liminf_{n\to\infty}
t(1-b_n)\le \ep/4+\ep/16.
\eneq
Define $g_n\in \Aff(T(A))$ by 
\beq
g_n(\tau)=\tau(a^*a(1-b_n))-\af^2=\tau(|a|(1-b_n)|a|)-\af^2 \rforal \tau\in T(A).
\eneq
Then, 
for $\tau\in O,$  
\beq
g_n(\tau)
&\le& \tau(a^*a)-\af^2\le \|a\|_{_{2,O}}^2-\af^2\le 0.
\eneq
For $\tau\in \partial_e(T(A))\setminus O,$ 
\beq
\liminf_{n\to\infty}g_n(\tau)&=&\liminf_{n\to\infty}\tau((1-b_n)^{1/2}a^*a(1-b_n)^{1/2})-\af^2\\
&\le& \liminf_{n\to \infty} \|a\|^2\tau(1-b_n)-\af^2\le \ep/4+\ep/16-\af^2<0.
\eneq
Therefore $\inf_n g_n\le 0.$
On the other hand, since $(1-b_n)\ge (1-b_{n+1})$ for all $n\in \N$ (working in $\wtd A,$ if necessary),
for all $\tau\in \partial_e(T(A)),$
\beq\label{364}
g_n(\tau)=\tau(|a|(1-b_n)|a|)-\af^2\ge \tau(|a|(1-b_{n+1})|a|)-\af^2=g_{n+1}(\tau).
\eneq
In other words, $\{g_n\}$ is decreasing. 
By \cite[Proposition 2.1]{And}, ${\Aff(T(A))_+}|_{\partial_e(T(A))}$ separates points and closed sets. It follows from 
\cite[Theorem 2.6]{And}  that 
$\Aff(T(A))|_{\partial_e(T(A))}$  has  strong Dini property (see \cite[Definition 2.4]{And}).  Therefore, there is $N$ \st\, 
\beq\label{Lsperation2-10}
g_N(\tau)<\ep^2/4\rforal \tau\in \partial_e(T(A)).
\eneq
Put $c=1-(1-b_N)^{1/2}.$ Note that $c\in A_+^{\bf 1}.$
We estimate that  (see also \eqref{6-20-n2}), for all $\tau\in F,$
\beq
\tau(c)=1-\tau((1-b_N)^{1/2})\le 1-\tau(1-b_N)=\tau(b_N)<\ep^2/4+\ep^2/32.
\eneq
Moreover (see also  \eqref{364} and \eqref{Lsperation2-10} ),
\beq
\|a(1-c)\|_{_{2, T(A)}}^2&=&\sup\{\tau((1-c)a^*a(1-c)): \tau\in T(A)\}\\
&=&\sup\{\tau(a^*a(1-c)^2): \tau\in T(A)\}\\
&=&\sup\{\tau(a^*a(1-b_N)): \tau\in T(A)\}\\
&\le & \sup\{g_N(\tau): \tau\in T(A)\}+\af^2<\ep^2/4 +\af^2.
\eneq
So
\beq
\|a(1-c)\|_{_{2, T(A)}}\le (\af^2+\ep^2/4)^{1/2}\le (\af+\ep/2)= \|a\|_{_{2, F}}+\ep.
\eneq
\end{proof}

At this point let us introduce the following 
definition:

\begin{df}\label{DT5}
Let $A$ be a separable algebraically simple \CA\, with non-empty  compact 
$T(A).$ We say that $T(A)$ has property (TE), if, 
for any compact subset $F\subset \partial_e(T(A)),$ 
$\|\cdot\|_{_{2, {F}_\varpi}}$ is a quotient norm of $\|\cdot\|_{_{2, {T(A)}_{\varpi}}},$ i.e.,
\beq
\|\pi_{F}(x)\|_{_{2,{F}_\varpi}}=\inf \{\|x+j\|_{_{2, {T(A)}_{\varpi}}}:\, j\in I_{_{F, \varpi}}/I_{_{T(A), \varpi}}\}
\eneq
for all $x\in l^\infty(A)/I_{_{T(A), \varpi}},$  where $\pi_{F}: l^\infty(A)/I_{_{T(A), \varpi}}\to l^\infty(A)/I_{_{F, \varpi}}$
is the quotient map.

By (2) of Proposition \ref{Pquotientnorm}, for any separable amenable algebraically  simple \CA\,  $A$ with non-empty Bauer simplex  $T(A),$
$T(A)$  has property (TE). 
%
%
\end{df}

\begin{lem}\label{PQKK}
Let $A$ be a separable algebraically simple \CA\, with non-empty  compact $T(A).$
Suppose that $F_1\subset F_2\subset \partial_e(T(A))$ are compact subsets \st\, for all $x\in l^\infty(A)/I_{_{T(A), \varpi}},$
\beq\label{PQKK-1}
\|\pi_{F_1}(x)\|_{_{2, {F_1}_\varpi}}&=&\inf\{\|\pi_{F_2}(x)+j\|_{_{2, {F_2}_{\varpi}}}: j\in I_{_{F_1, \varpi}}/I_{_{F_2, \varpi}}\}\andeqn\\\label{PQKK-2}
\|\pi_{F_2}(x)\|_{_{2, {F_2}_\varpi}}&=&\inf\{\|x+k\|_{_{2, T(A)_{\varpi}}}: k\in I_{_{F_2, \varpi}}/I_{_{T(A),\varpi}}\},
\eneq
where $\pi_{F_i}: l^\infty(A)/I_{_{T(A), \varpi}}\to l^\infty(A)/I_{_{F_i, \varpi}}$ is the quotient map, $i=1,2.$
Then
\beq\nonumber
\|\pi_{F_1}(x)\|_{_{2, {F_1}_\varpi}}=\inf\{\|x+b\|_{_{2, T(A)_{\varpi}}}: b\in I_{_{F_1, \varpi}}/I_{_{T(A),\varpi}}\}
\rforal x\in l^\infty(A)/I_{_{T(A), \varpi}}.
\eneq
\end{lem}

\begin{proof}
Fix $x\in l^\infty(A)/I_{_{T(A), \varpi}}$ and $\ep>0.$
By \eqref{PQKK-1}, there is $j\in I_{_{F_1, \varpi}}$ \st\,
\beq
\|\pi_{F_1}(x)\|_{_{2, {F_1}_\varpi}}&>&\|\pi_{F_2}(x+j)\|_{_{2, {F_2}_{\varpi}}}-\ep/2.
\eneq
Applying \eqref{PQKK-2}, we obtain $k\in I_{_{F_2, \varpi}}$ \st\,
\beq
\|\pi_{F_2}(x+j)\|_{_{2, {F_2}_\varpi}}>\|(x+j)+k\|_{_{2, T(A)_\varpi}}-\ep/2.
\eneq
It follows that 
\beq
\|\pi_{F_1}(x)\|_{_{2, {F_1}_\varpi}}>\|x+(j+k)\|_{_{2, T(A)_\varpi}}-\ep.
\eneq
Recall that $I_{_{F_2, \varpi}}\subset I_{_{F_1, \varpi}}.$ Hence $(j+k)\in I_{_{F_1, \varpi}}.$
We conclude that
\beq
\|\pi_{F_1}(x)\|_{_{2, {F_1}_\varpi}}\ge \inf\{\|x+b\|:b\in I_{_{F_1, \varpi}}/I_{_{T(A), \varpi}}\}.
\eneq
The lemma then follows (by applying (1) of Proposition \ref{Pquotientnorm}).
\end{proof}

\begin{thm}\label{Ptight}
Let $A$ be a separable algebraically simple \CA\, with non-empty compact $T(A).$
Suppose that $A$ has strict comparison and $\Gamma$ is surjective. 
Then $T(A)$ has property (TE),
%

%
\end{thm}

\begin{proof}
Fix $x\in (l^\infty(A)/I_{_{T(A), \varpi}})^{\bf 1}.$ We may write 
$x=\Pi_{_{\varpi}}(\{a_k\}),$ where $\{a_k\}\in l^\infty(A)^{\bf 1}.$

Let us fix a compact subset $F\subset \partial_e(T(A)).$ 
For each $k\in \N,$ by Lemma \ref{Lsperation2}, there is $d_k\in A_+^{\bf 1}$ such 
that  $\tau(d_k)<1/k$ for all $\tau\in  F,$  and 
\beq
\|a_k(1-d_k)\|_{_{2, T(A)}}\le \|a_k\|_{_{2,F}}+1/k.
\eneq  
Hence $d=\{d_k\}\in I_{_{F, \varpi}}.$ Moreover, 
\beq
\|x(1-d)\|_{_{2, T(A)_\varpi}}=\lim_{n\to\varpi}\sup\{\|a_k(1-d_k)\|_{_{2, \tau}}: \tau\in T(A)\}\le \|\pi_F(x)\|_{_{2, F_\varpi}}.
\eneq
%
Applying (1) of Proposition \ref{Pquotientnorm}, we obtain that
\beq\nonumber
\|\pi_F(x)\|_{_{2, F_\varpi}}=\inf\{\|x-j\|_{_{2, T(A)_\varpi}}: j\in I_{_{F, \varpi}}/I_{_{\varpi}}\}.
\eneq
%
\end{proof}

%
%


The following is well-known. 

\begin{lem}\label{Lstrictp}
Let $A$ be a unital \CA\, and $I\subset A$ be a $\sigma$-unital ideal of $A.$ 
Suppose that $p_1, p_2,...,p_k$ are mutually orthogonal projections with $\sum_{i=1}^k p_k=1$ 
and $\{e_{i,n}\}$ is an approximate identity for $e_iIe_i,$ $1\le i\le k.$ 

(1) Then $\{\sum_{i=1}^k e_{i,n}\}$ is an approximate identity for $I.$

 (2) If $s\in I$ is a strictly positive element of $I,$ 
then $\sum_{i=1}^k p_isp_i$ is a strictly positive element of $I,$ and 

(3)  $p_isp_i$ is a strictly positive element of $I_i$ ($1\le i\le k$).
\end{lem}

\begin{proof}
Let $s\in I$ be a strictly positive element of $I.$ 
If $k=2,$ then   (see \cite[Lemma I.1.11]{BH})
\beq
s&=&s^{1/2}s^{1/2}=((1-p_1)s^{1/2}+p_1s^{1/2})(s^{1/2}(1-p_1)+s^{1/2}p_1)\\
&\le& 2((1-p_1)s(1-p_1)+p_1sp_1).
\eneq
Let $f$ be a positive linear functional of $I.$ 
Then $f((1-p_1)s(1-p_1)+p_1sp_1)=0$ implies that $f(s)=0.$ It follows that 
$(1-p_1)s(1-p_1)+p_1sp_1$ is a strictly positive element of $I.$ 

Let $g$ be a positive linear functional of $p_1Ip_1.$ 
Define $\wtd g(a)=g(p_1ap_1)$ for all $a\in A.$ Then $\wtd g$ is a positive linear functional 
of $A.$ If $g(p_1sp_1)=0,$  then $\wtd g((1-p_1)s(1-p_1)+p_1sp_1)=0.$
Since $(1-p_1)s(1-p_1)+p_1sp_1$ is a strictly positive element of $I,$ this implies that 
$\wtd g=0.$  Thus $g=0.$ In other words, $p_1sp_1$ is a strictly positive element of $p_1Ip_1.$ 
Hence (2) and (3) follows for $k=2.$ The general case for (2) and (3) follows from the induction.
Note that (1) follows from (2).

\end{proof}

\begin{lem}\label{Lprojlift}
Let $A$ be a unital \CA,\, $I$ an ideal of $A,$  $\pi: A\to A/I$ the quotient map and  $p\in A$ a projection.
Suppose that $I$ has real rank zero and ${\bar q}\in A/I$ is another projection such that
\beq
\|{\bar q}-\pi(p)\|<1/4.
\eneq
Then there is a projection $q\in A$ such that $\pi(q)={\bar q}.$
\end{lem}

\begin{proof}
Let $a\in A_+^{\bf 1}$ be such that $\pi(a)={\bar q}.$ 
Then there is $j\in I$ such that
\beq
\|p-a+j\|<1/4.
\eneq
Let $B$ be the \SCA\, generated by $1_A,\, p, a$ and $j.$ It is separable. Set $J_B=B\cap I.$ Then $J_B$
is also separable. Let $b$ be a strictly positive element for $J_B.$ Define $J=\overline{bIb}.$
Then $J$ is a $\sigma$-unital hereditary \SCA\, of $I.$ Hence $J$ has real rank zero. 
It follows that $pJp$ and $(1-p)J(1-p)$ all have real rank zero.
Then, by Lemma \ref{Lstrictp},  both $pJp$ and $(1-p)J(1-p)$ are $\sigma$-unital. Let $\{e_{1,n}\}$  and $\{e_{2,n}\}$ be approximate identities consisting of projections for 
$pJp$ and $(1-p)J(1-p),$ respectively (see \cite[Proposition 2.9]{BP}).  By 
Lemma \ref{Lstrictp}, $\{e_n\}=\{e_{1, n}+e_{2, n}\}$ forms an approximate identity for $J$ 
consisting of projections.  Note  that $\{e_n\}$ commutes with $p.$
 Choose a large $n\in\N$ such that
\beq
\|p(1-e_n)-(1-e_n)a(1-e_n)\|<1/2. 
\eneq
Since $p(1-e_n)$ is a projection, there is a function $f\in C_0((0, 1])_+^{\bf 1}$ such that
$f(1)=1$ and $q=f((1-e_n)a(1-e_n))$ is a projection. Note that $\pi(q)=f(\pi(a))=f(\bar q)={\bar q}.$
The lemma follows. 
\end{proof}

The following is taken from  Corollary 3.3 of  \cite{Ellaut}  and its proof is just a modification of 
that of \cite[Corollary 3.3]{Ellaut}  (see also \cite[Lemma 4.8]{Lincrell}).

\begin{lem}[The Elliott Lifting Lemma]\label{Lelliott}
Let $A$ be a \CA\, of real rank zero, $I\subset A$ be an ideal 
of $A$ and $D$ a finite dimensional \CA\, and $\phi: D\to A/I$  a \hm.
Then there exists a \hm\, $\wtd \phi: D\to A$ such that 
$\pi\circ \wtd\phi=\phi.$ 
\end{lem}

\begin{proof}
To simplify notation, \wilog, we may assume that $\phi$ is injective.
Since $A$ has real rank zero,  by \cite[Theorem 3.14]{BP}, there is a projection $E\in A$ such that
$\pi(E)=1_D.$ Then $EAE$ also has real rank zero. 
Let $G$ be a finite subset of $D$ which generates $D.$ 
Let $S$ be a countable set of projections in $D$  containing  $1_D$  and all projections in $G$  which is dense in the set of projections of $D.$
Since $EAE$ has real rank zero, every projection in $S$ lifts to a projection in $EAE$ (see \cite[Theorem  3.14]{BP}).
Let $P$ be the set of projections in $EAE$ such that $\pi(P)=S$ containing $E$ and $\wtd G$ be a finite subset 
of $EAE$ such that $\pi(\wtd G)=G.$
  Let $C$ be the \CA\, generated by $P$ and $\wtd G.$  It is unital as it contains $E.$
  Let $J_C=C\cap I.$ Then $J_C$ is a separable \CA. Choose a strictly positive element 
  $c\in J_C.$ Define $J=\overline{cIc}.$ Then $J$ is a $\sigma$-unital hereditary \SCA\, of $I$ which has real rank zero.
   
   Let $A_1=C+J.$ Then we have the short exact sequence 
   \beq
   0\to J\to A_1\to D\to 0.
   \eneq
   Note that every projection in $S\subset D$ lifts to a projection in $A_1$ since $C\subset A_1.$
   By Lemma \ref{Lprojlift}, every projection in $D$ lifts to a projection in $A_1.$ It follows from 
   \cite[Lemma 4.8]{Lincrell} that there is a \hm\, $\wtd \phi: D\to A$ such that $\pi\circ \wtd \phi=\phi.$
%
%
\end{proof}


\section{Examples of Choquet simplices and separation}

We begin with the following examples.

\begin{exm}\label{Exsimplex}

(1) Let $K$ be a compact metrizable space with countable dimension.
Let $C=C(K).$ Then $T(C)=M(K)_+^{\bf 1}$ (the set of all probability Borel measures)  is a metrizable Choquet simplex with $\partial_e(T(C))=K.$

(2) By \cite{rH}, any separable completely metrizable countable-dimensional space $T$ is the extremal boundary of some metrizable Choquet simplex. 
In particular, any separable, locally compact metrizable space $T$ of countable dimension 
is the extremal boundary of some metrizable Choquet simplex $\DT$ (note also that, in this case, $\eD=T$ is $\sigma$-compact). 
\end{exm}

We would like to give more specific examples of non-Bauer simplexes.

\begin{df}\label{Dchoqet}
Let $\Delta$ be a (metrizable) Choquet simplex.   Then $\eD$ is a $G_\dt$-set (see, for example, 
\cite[Cor. I.4.4]{Alf}).
By the Choquet theorem, for each $t\in \Delta,$ 
there is a  unique boundary measure $\mu_t$ on $\partial_e(\Delta),$
a Borel probability measure 
on $\Delta$ concentrated on 
$\partial_e(\Delta)$  such that $t$ is the barycenter of   $\mu_t,$ i.e., 
\beq\label{Dc1-0}
f(t)=\int_{\eD} f(x) d\mu_t \rforal f\in \Aff(T).
\eneq

Let $S\subset \partial_e(T)$ be a Borel subset. Denote by
$M_S=\{\mu_t: t\in \Delta,\, \mu_t(\eD\setminus S)=0\}.$ In other words,
$M_S$ is the set of those Borel probability measures on $\Delta$ concentrated on $S.$
In what follows we may identify $t$ with $\mu_t$ and $M_S$ with 
the subset of those points in  $\Delta$ whose associated extremal boundary measures concentrated on $S.$
Denote by ${\rm conv}(S)$ the convex hull of $S.$ Note that ${\rm conv}(S)\subset M_S\subset \overline{{\rm conv}(S)}.$
\end{df}

Suppose that  $\DT$ is a Bauer  simplex, i.e., $\eD$ is closed. 
Let $K$ and $F$ be two disjoint closed subsets 
of $\eD.$ Then there exists  $f\in \Aff(\DT)$ with $0\le f\le 1$ such that $f|_K=0$ and 
$f|_F=1.$  From 
(2) of Proposition \ref{Pquotientnorm}, one sees that a separable  algebraically simple  
\CA\, $A$ with $T(A)=\DT$ has property (TE). Note that this holds  without assuming 
strict comparison and surjectivity of $\Gamma.$ 
 It is clear that property (TE) is closely related to some  separation property of $T(A).$
 It turns out that it also closely related to the tightness property introduced by W. Zhang (\cite{Zw}).
 We would like to look at this phenomenon  in non-Bauer simplexes more closely. While these lines of discussion 
 are not fully used in the proof of Theorem \ref{TM-2}, we believe the clarification and examples might be helpful 
 for future study (in particular, when one attempts to remove the assumption that $\Gamma$ is surjective in 
 the study of Toms-Winter conjecture).

\begin{df}\label{Dtight}
Fix a  Choquet simplex $\Delta.$
Denote by $bd(\eD)=\overline{\eD}\setminus \eD.$

(1) If there exists a compact subset $K\subset \eD$ such that 
$bd(\eD)\subset \overline{{\rm conv}(K)},$ we say that $\eD$ has property (T1).

(2)  (W. Zhang \cite[Definition 2.1]{Zw}).
Let us recall W. Zhang's notion of 
tight extremal boundaries.
We say   that 
$\eD$  is tight
if, for any $\ep>0,$ there exists a compact subset $K\subset \eD$ such 
that 
\beq
\mu_t(K)>1-\ep\rforal t\in bd(\eD)=\overline{\eD}\setminus \eD.
\eneq
To be consistent with other names of properties, if $\eD$ is tight, we also say that $\eD$ has 
 property (T2). Note that, if $\eD$ is compact, then $\eD$ is tight.   By (4) of Proposition \ref{Pfinitebde} below, 
 if $\eD$ has property (T1), then it has property (T2).
 The tightness condition on $\eD$  should not be viewed as some measure theoretical 
 condition but an affine condition.

 (3) We say that $\eD$ has property (T3') if  there exists a compact subset $K_c\subset \eD$ satisfying the following:
 for any 
compact subset $F\supset K_c$ and any
open subset $O\subset \Delta$ with $O\supset \overline{{\rm conv}(F)},$ there is an open subset 
$G\subset O$ such that $G\supset \overline{{\rm conv}(F)}$ and 
$\overline{{\rm conv}(\eD\setminus G)}$ is a face.

Recall that ${\rm conv}(\eD\setminus G)\subset M_{\eD\setminus G}\subset \overline{{\rm conv}(\eD\setminus G)}$
and $M_{\eD\setminus O_F}$ is a face. So, if  $M_{\eD\setminus G}$ is closed for any $G$ as above, 
then $\eD$ has property (T3').
Of course if $\eD$ is compact, $\eD$ has property (T3').

Suppose that $\eD$ has property (T1), i.e.,
there exists a compact 
subset $K_c\subset \eD$ such that $bd(\eD)\subset {\overline{{\rm conv}(K_c)}}.$ 
Then $\eD$ has property (T3'). To see this, let $O\supset \overline{{\rm conv}(K_c)}$
be an open subset. Then $\eD\setminus O=\overline{\eD}\setminus O$ is compact.
Put $S=\eD\setminus O.$ Since $S$ is compact, 
by (2) of Proposition \ref{Pfinitebde} below, $M_S=\overline{{\rm conv}(S)}$ is a face.
Hence $\eD$ has property (T3'). 
%
%

(4) We say that $\eD$ has property (T3) (affinely  $T_3$),
 if  there exists a compact subset 
$K_c\subset \eD$ satisfying the following: for any compact $F\subset \eD$ with $K_c\subset F$
 and any 
open subset $O\subset \Delta$ with $O\supset \overline{{\rm conv}(F)},$ there exists an open subset 
$G\supset \overline{{\rm conv}(F)}$ and $G\subset O$ such that there is $f\in \Aff(\Delta)$ with $0\le f\le 1,$ 
$f|_F=0$ and $f|_{\eD \setminus G}=1.$

Suppose that $\eD$ has the following property: for any compact subset $K\subset \eD$ and any relatively open 
subset $V\subset \eD$ with $K\subset V,$ there is $f\in \Aff(\Delta)$ with $0\le f\le 1$ 
such that $f|_K=0$ and $f|_{\eD\setminus V}=1.$  Then $\eD$ has property (T3).

To see this, let $d$ be the metric on $\Delta$ and $O\subset \Delta$ be an open subset 
such that $F:=\overline{{\rm conv}(K)}\subset O.$ Since $F$ is compact, there exists 
$\ep>0$ such that 
$G:=O_\ep(K)=\{t\in \Delta: d(t, K)<\ep/2\}\subset O.$ Let $V=\eD\cap G.$ 
Then $V$ is relatively open and $K\subset V.$
 There is $f\in \Aff(\Delta)$ with $0\le f\le 1$ such that $f|_K=0$ and $f|_{\eD\setminus V}=1.$
 But $\eD\setminus V=\eD\setminus G.$ So $f_{\eD\setminus G}=1.$  In other words, 
 $\eD$ has property (T3).

One can also  show that (T3') and (T3) are equivalent.

\end{df}

 
 Part (2) of the next proposition is known. 
 At least a part of (3) of the next proposition  is known to W. Zhang (see the end of Definition 2.1 of 
 \cite{Zw}).
 
 \begin{prop}\label{Pfinitebde}
 Let $\Delta$ be a Choquet simplex.
 
 (1) If $\eD$ is tight, then $\eD$ is $\sigma$-compact.
 
 (2) If $K\subset \eD$ is compact,  then $\overline{{\rm conv}(K)}=M_K.$ 
 
 
 (3) If there exists a compact subset $K\subset \eD$ \st\, $bd(\eD)\setminus \overline{{\rm conv}(K)}$
 is finite, then $\eD$ is tight.  In particular, if $bd(\eD)$ is a finite subset, then $\eD$ has property (T2).
 
 (4) If $\eD$ has property (T1), then $\eD$ has property (T2).
 \end{prop}
 
 \begin{proof}
For (1), let $K\subset \eD$ be a compact subset such that 
$\mu_t(K)>1/2$ for all $t\in bd(\eD).$
Let $\tau\in \eD\setminus K$ and $K_1=K\sqcup \{\tau\}.$ 

We claim that there is an open subset $O\subset \Delta$ such that $\tau\in O$ and 
$O\cap \overline{bd(\eD)}=\emptyset.$

Otherwise there is a sequence $\{\xi_n\}\subset bd(\eD)$ such that $\xi_n\to \tau.$ 
Note that $K_1$ is compact. By \cite[Theorem II. 3.12]{Alf}, there exists a function $f\in \Aff(\Delta)$ 
with $0\le f\le 1$ such that $f|_K=1$ and  $f(\tau)=0.$  
Then $f(\xi_n)\to f(\tau)=0.$
However, for all $n\in \N,$
\beq
f(\xi_n)=\int_{\eD} f d\mu_{\xi_n}\ge \int_{K} d\mu_{\xi_n}>1/2.
\eneq
This is a contradiction. This proves the claim. 
 
Since $\Delta$ is a compact metrizable space,  there is a neighborhood $O_1$ of $\tau$ in $\Delta$ 
\st\, ${\bar O_1}$ is compact and ${\bar O_1}\subset O.$ 
Let $G=O_1\cap \eD$ be the relative open subset of $\eD.$
Since $O\cap bd(\eD)=\emptyset,$ 
${\bar O}_1\cap \overline{\eD} \subset O\cap \overline{\eD}\subset \eD.$  Hence  ${\bar O}_1\cap \overline{\eD}$ is compact
subset of $\eD.$ 
It follows that ${\bar G}={\bar O}\cap \eD$
is compact.

Since $\eD\setminus K$ is second countable, we obtain countably many 
compact sets $\{{\bar G}_n: n\in \N\}$ \st \,
$\cup_{n=1}^\infty{\bar G}_n\supset \eD\setminus K.$ 
Then $\eD=K\cup \cup_{n=1}^\infty{\bar G}_n.$ This implies that $\eD$ is $\sigma$-compact. 

For (2), we note that  $\overline{{\rm conv}(K)}$ is a closed face and $\partial_e(\overline{{\rm conv}(K)})=K$  (see \cite[Cor. 11.19]{kG}).  
In particular, 
$\overline{{\rm conv}(K)}$ is a compact convex space. By \cite[Theorem II. 3.12]{Alf},
for every $g\in C(K)_{s.a.},$ there is $f\in \Aff(\Delta)$ \st\, $f|_K=g.$ 
Note that $f|_{\overline{{\rm conv}(K)}}\in \Aff(\overline{{\rm conv}(K)}).$ 
By \cite[Theorem II. 4.3]{Alf}, $\overline{{\rm conv}(K)}$ is a Bauer simplex. 
By the Choquet Theorem,  for each $x\in \overline{{\rm conv}(K)},$ there is a unique boundary measure $m_x$
on $K$ with the barycenter $x.$ 
Define a measure $\wtd m_x$ on $\Delta$ by 
\beq
\wtd m_x(S)=m_x(K\cap S) \rforal {\rm Borel\,\, sets}\,\, S\subset \DT.
\eneq
So $\wtd m_x$ is a boundary measure. 
For any affine function $f\in \Aff(\Delta),$ $f|_{\overline{{\rm conv}(K)}}\in \Aff(\overline{{\rm conv}(K)}).$
It follows that
\beq
f(x)=\int_{\eD} fd \mu_x=\int_K f dm_x \rforal f\in \Aff(\Delta).
\eneq
By the uniqueness of such boundary measure $\mu_x,$ one obtains that $\wtd m_x=\mu_x$ and 
\beq
\mu_x(K)=1\rforal x\in \overline{{\rm conv}(K)}.
\eneq
It follows that $\overline{{\rm conv}(K)}=M_K.$

For (3), let $K$ be in the statement of (3). 
By (2), one also has $\overline{{\rm conv}(K)}\cap \eD=K.$

Recall  that $bd(\eD)\setminus \overline{{\rm conv}(K)}$ 
is a finite subset, whence $F=bd(\eD)\cup \overline{{\rm conv}(K)}$ is compact.

Note that $F\cap \eD=\overline{{\rm conv}(K)}\cap \eD=K.$ 
Hence $(\eD\setminus K)\cap F=\emptyset.$  
 Therefore, for each  $\tau\in \eD\setminus K,$ there is a neighborhood 
$O_\tau$ in $\DT$ \st\, $O_\tau\cap F=\emptyset.$  Since $\Delta$ is compact, 
$\tau$ has a compact neighborhood $K_{\tau, 1}\subset O_\tau.$ 
Note that  $\overline{\eD}\cap (K_{\tau,1}\setminus F)\subset \eD.$
But $K_{\tau,1}\cap F=\emptyset.$ Therefore $\overline{\eD}\cap K_{\tau,1}=\overline{\eD}\cap (K_{\tau,1}\setminus F)\subset \eD$ is a compact neighborhood of $\tau$ in $\eD.$ 
By the fact that $\eD\setminus K$ is second countable, we conclude that $\eD\setminus K$ is $\sigma$-compact.
 It follows that $\eD$ is $\sigma$-compact. 
 
 To show that $\eD$ has property (T2), let $\ep>0.$ 
 Write $S=bd(\eD)\setminus \overline{{\rm conv}(K)}.$ Since $S$ is finite and $\eD$ is $\sigma$-compact,
 one obtains a compact subset $K_e\subset \eD$ such that, for any $s\in S,$
$ \mu_s(K_e)>1-\ep.$
 It follows from (2), if $x\in \overline{{\rm conv}(K)},$ $\mu_x(K)=1.$ 
 Put $K_c=K\cup K_e.$ 
 Then, for any $x\in bd(\eD),$ 
 \beq
 \mu_x(K_c)>1-\ep.
 \eneq
 
 Finally we note that (4) follows (3) (but also follows from (2)). 
 \end{proof}

\begin{lem}\label{Ltight-1}
Let $A$ be a separable algebraically simple \CA\, with non-empty compact 
$T(A)$ 
with tight $\partial_e(T(A)).$ 
Then, 
for any $\ep>0$ and any 
compact subset $S\subset
\partial_e(T(A)),$
there exist a compact subset $F\subset \partial_e(T(A))$ with $F\supset S$   satisfying the following:
for any $a\in A^{\bf 1},$ there exists 
$d\in A_+^{\bf 1}$ 
such that $\tau(d)<\ep$ for all $\tau\in F$ and 
\beq
\sup\{\|a(1-d)\|_{{2, \tau}}: \tau\in T(A)\}<\sup\{\|a\|_{_{2, \tau}}: \tau\in F\}+\ep.
\eneq 
\end{lem}

\begin{proof}
Recall that in this case, $T(A)$ is a metrizable Choquet simplex
(see the second last  
part of  Definition \ref{DefOS1}). 
Since $\partial_e(T)$ is tight, there is a compact subset $K$ such that
\beq\label{Ltigt-1-e-1+}
\mu_t(K)>1-\ep^2/16\rforal t\in \overline{\partial_e(T)}\setminus \partial_e(T).
\eneq
Put $F=S\cup K.$ Fix $a\in A^{\bf 1}.$ Set $\af=\sup\{\|a\|_{_{2, \tau}}: \tau\in F\}.$ 

For any  $t\in \overline{\partial_e(T(A))}\setminus \partial_e(T(A))$ (see also \eqref{Ltigt-1-e-1+}),
\beq
\hspace{-0.3in}t(a^*a)=\int_{\partial_e(T(A))} \tau(a^*a) d\mu_t
=\int_{F} \tau(a^*a) d\mu_t +\int_{\partial_e(T(A))\setminus F}d\mu_t<\af^2+\ep^2/16 ,
\eneq
where $\mu_t$ is the unique boundary measure associated with $t.$

Put $G=\{\tau\in T(A): \|a\|_{_{2, \tau}}<\af+\ep/4\}.$ Then $G$ is open 
and we just have shown that $\overline{\partial_e(T(A))}\setminus \partial_e(T(A))\subset G.$ 
Hence $K_1=\overline{\partial_e(T(A))} \setminus G=\partial_e(T(A))\setminus G$ is a compact subset 
of $\partial_e(T(A)).$  Put $F_1=K_1\cup F.$

By the continuity of $\widehat{a^*a},$ we obtain relative open subsets $O_1, O_2$ of $F_1$ such that 
$F\subset O_1\subset \bar O_1\subset O_2\subset {\bar O}_2\subset F_1$
such that
\beq\label{Ltight-1-e-1}
\sup\{\|a\|_{_{2, \tau}}: \tau\in O_2\}<\af
+\ep^2/16.
\eneq
There is $g_0\in C(F_1)_+^{\bf 1}$ such that ${g_0}|_{F}=0$ and 
${g_0}|_{F_1\setminus O_2}=1.$ By \cite[Theorem II.3.12, (ii)]{Alf}, there exists $h\in \Aff(T(A))$ with $0\le h\le 1$
such that $h|_{F_1}=g_0.$  
For each $k\in \N,$ choose $h_k\in \Aff(T(A))_+$ such that $h_k=h+\ep^2/3k.$

It follows from  Proposition \ref{P93}
that, for each $k\in \N,$ there 
exists $b_k\in A_+^{\bf 1}$ such that
\beq
&&\sup\{|\tau(b_k)-h_k(\tau)|: \tau\in F_1\}<\ep^2/2k.
 %
\eneq
Put $a_k=a(1-b_k)^{1/2}.$ Note that, for $k\in \N,$ 
\beq\label{Ltight-1-e-10}
&&\hspace{-0.8in}\sup\{\tau(b_k): \tau \in S\}<\ep^2/k,\\\nonumber
\hspace{-0.4in}\|a_k\|_{_{2, O_2}}&=&\sup\{\tau(a^*a(1-b_k))^{1/2}:\tau\in O_2\}
=\sup\{\tau(|a|(1-b_k)|a|)^{1/2}:\tau\in O_2\}\\
&\le& \sup\{\tau(a^*a)^{1/2}:\tau\in O_2\}\le \af+\ep^2/16.
\eneq
Moreover
\beq
\|a_k\|_{_{2, F_1\setminus O_2}} &=&
\sup\{\tau(a^*a(1-b_k))^{1/2}:\tau\in F_1\setminus O_2\}\\
&=&\sup\{\tau((1-b_k)^{1/2}a^*a(1-b_k)^{1/2})^{1/2}:\tau\in F_1\setminus O_2\}\\
&\le &\|a^*a\|\sup\{\tau((1-b_k)^2)^{1/2}: \tau\in F_1\setminus O_2\}\\
&\le & \sup\{\tau((1-b_k))^{1/2}: \tau\in F_1\setminus O_2\}<\ep/\sqrt{k}.
\eneq 
We also have that
\beq
\|a_k\|_{_{2, G}}<\af+\ep/4.
\eneq
By Proposition \ref{Punion}, 
\beq
\|a(1-b_k)\|_{_{2, T(A)}}<\af+\ep/4.
\eneq
The lemma then follows by choosing a large $k.$

\end{proof}

From the proof of Theorem \ref{PQKK} and that of (2) of Proposition \ref{Pquotientnorm},  as well  as
(1) of  Proposition \ref{Pquotientnorm} and  Lemma \ref{Ltight-1},
one obtains the 
next proposition.  We omit the proof as we do not use them in later sections. However, we believe  that 
property (TE) is closely related to the tightness and the separation property such as (T3).  It is also possible
that property (TE) always holds for any separable algebraically simple \CA\,  with compact tracial state space.

\begin{prop}\label{PtightTT}
Let $A$ be a separable algebraically simple  \CA\, with non-empty compact $T(A).$
Then $T(A)$ has property (TE),
%
if one of the following holds.

(1)  $\partial_e(T(A))$ is tight; 

(2) $\partial_e(T(A))$  has property (T3).
%
\end{prop}

\begin{exm}\label{Exsimplex2}

(1)  (\cite[p.868]{Brs} and \cite[Example 3.3]{CETW}). Let $Y$ be any $\sigma$-compact  and locally compact metrizable 
space with countable dimension.
Let us present a non-Bauer  $Y+2$-simplex. 

Write $Y=\cup_{n=1}^\infty K_n,$ where $K_n\subset K_{n+1}$ and each $K_n$ is a compact 
subspace with countable dimension.   Let $Z$ be the 
one point compactification of $Y$ and write $Z=Y\cup \{\xi_\infty\}.$
Define  
\beq
E=\{f\in C(Z, M_2):  f(\xi_\infty)=\begin{pmatrix} a & 0\\
                0& b\end{pmatrix},\, a,b\in \C\}.
\eneq
Then $E$ is a unital \CA,  $T(E)$ is a Choquet simplex, and $\partial_e(T(E))=
Y\sqcup \{\tau^+, \tau^-\}$ which is not compact (in weak*-topology),
where $\tau^+(f)=a$ and $\tau^-(f)=b$ if 
$f(\xi_\infty)=\diag(a,b)$ as above.
Put $X_1=\{\tau^+, \tau^-\}$ and $X_n=X_1\cup K_n,$ $n\in \N.$  Note that $X_n$ is compact for each $n$ and 
has countable dimension. Then $\partial_e(T(E))=\cup_{n=1}^\infty X_n.$ It is $\sigma$-compact and countable-dimensional (see \cite[Example 2.13]{Linoz}). 
Put $\Delta=T(E).$ Then $bd(\overline{\eD})\subset {\rm conv}(X_1).$ 
So $\eD$  has property (T1) and is therefore  tight.
One can also similarly present a $Y+n$-simplex.
 
(3) Let $X$ be a (countable-dimensional) compact metrizable space. 
Define $B=C(X)\otimes E.$ Then $T(B)=T(C(X))\times T(B).$ One sees that 
$\partial_e(T(B))=X\times Y$ and $bd(\overline{\partial_e(T(B))})=X\times \{1/2(\tau^++\tau^-)\}$ which is homeomorphic to $X.$
So $bd(\overline{\partial_e(T(B))})$ has infinitely many points. 
 The extremal boundary $\partial_e(T(B))$ is a $\sigma$-compact 
metrizable (countable-dimensional) space (see  \cite[Theorem 5.2.20]{En}).

Put $K_c=\{X\times \{\tau^+, \tau^-\}\}.$ Then $K_c$ is compact. Moreover, $bd(\overline{\partial_e(T(B))})\subset \overline{{\rm conv}}(K_c).$ So $T(B)$ has property (T1), whence (T2) and (T3). 
In fact, for any unital 1-step RSH-algebra $B,$ 
$\partial_e(T(B))$ has  property (T1) and hence tight
(see Example \ref{ERSH} below). 
\end{exm}

\begin{df}\label{Dbounded}
Let $A$ be a \CA.
Denote by $\widehat{A}$ the primitive ideal space of $A,$ and, for each $n\in \N,$
denote by $_n \widehat{A}$ the subset of $\widehat{A}$ consisting of those kernels of irreducible representations
whose ranks are no more than $n.$  Put $\widehat{A}_n={_n\widehat{A}}\setminus {_{n-1}\widehat{A}}.$ 
We say that a \CA\, $A$ has bounded rank of irreducible representations if  $\widehat{A}={_n \widehat{A}}$ for some 
$n\in \N.$  An $FD$-algebra is a \CA\, whose irreducible representations are  
finite dimensional. In other words, $\widehat{A}=\cup_{n=1}^\infty \widehat{A}_n.$  

Recall that ${_n\widehat{A}}$ is always closed and $\widehat{A}_n$ is locally compact  and Hausdorff (\cite[Proposition 4.4.4 and 4.4.10]{Pbook}).  An  $FD$-algebra $A$ has countable-dimensional spectrum, if each $\widehat{A}_n$ 
has countable dimension. Let $A$ be an FD-algebra. Suppose that $\tau\in \partial_e(T(A)).$ 
 By \cite[Lemma 2.16]{Lincrell},  
 $\tau={\rm tr}_j \circ \pi_\tau$ for some $j\in \N,$ where $\pi_\tau\in \widehat{A}_j$ and 
 ${\rm tr}_j$ is the tracial state of $M_j.$ 
 So one  may write that  $\partial_e(T(A))=\cup_{n=1}^\infty \widehat{A}_n.$ 
It then follows that, for separable FD-algebras $A,$ $\partial_e(T(A))$ is $\sigma$-compact. 
\end{df}
In a preliminary report of this research,  we include the following two statements which are not used towards the proof of Theorem \ref{TM-2}.
With a standard induction, it is straightforward to prove the following  lemma (using a 4-Lemma), 
and the proof of the next proposition is a consequence of the next lemma (but also uses a Stone-Weierstrass theorem of Kaplansky  and Sakai (\cite{Sk})).
\begin{lem}\label{LT3}
Let $D$ be a unital separable \CA\, with bounded rank of irreducible representations.
Then $\partial_e(T(D))$ is $\sigma$-compact and  has property (T3).
\end{lem}

\begin{prop}\label{PT3}
Let $C$ be a separable unital FD-algebra.
Then $\partial_e(T(C))$ is $\sigma$-compact and has property (T3).
\end{prop}

\begin{exm}\label{ERSH}
Let $A$ be a unital 
 recursive sub-homogeneous   \CA\,(RSH-algebra). Then $A$  has bounded rank of irreducible representations.
 Therefore, by Proposition \ref{PT3}, $\eD$ has property (T3). 
%
%

  A unital 0-step RSH-algebra has the form $C(X_0)\otimes M_{r(0)},$ where $X_0$ is a compact metric space 
  (of countable dimension) and $r(0)\in \N.$ 

(i) W. Zhang showed that, if $C$ is a 1-step RSH-algebra, then 
$\partial_e(T(C))$ is tight (see the proof of \cite[Proposition  6 9]{Zw}). In fact, he showed 
that $bd(\overline{\partial_e(T(C))})\subset \overline{{\rm conv}(X_0)},$ where $X_0$ is the compact 
subset of $\partial_e(T(C))$ from the 0-step of homogeneous \CA. In other words, 
$\partial_e(T(C))$ always has property (T1) and hence has property (T2).

(ii) However, in general, a 2-step RSH-algebra $C$ 
may have $\partial_e(T(C))$ 
which is locally compact and has property (T3)  but does not have property (T2).

Let $E$ be as  of Example \ref{Exsimplex}.  Choose a compact metric  space $\Omega$ 
(with countable dimension) such that 
$Z=Y\cap \{\xi_\infty\}\subset \Omega$ and  each point of $Z$ is a cluster point of $\Omega\setminus Z.$ 
Let $C^{(1)}=E^{(1)}\oplus E^{(2)}$ where $E^{(i)}=E,$ $i=1,2.$ Define $\psi: C(Z, M_2)\oplus C(Z, M_2)\to C(Z, M_4)$
by $\psi((a,b))=\diag(a, b)$ for any pair $a, b\in C(Z, M_2).$
Define $\phi_1: C^{(1)}\to C(Z, M_4)$ by $\phi_1=\psi|_{E^{(1)}\oplus E^{(2)}}.$
Note that  we may write 
\beq
\partial_e(T(C^{(1)}))=(Y^{(1)}\sqcup \{\tau_1^+, \tau^-_1\})\sqcup (Y^{(2)}\sqcup \{\tau_2^+, \tau^-_2\}),
\eneq
where we identify $\partial_e(T(E^{(i)}))$ with $\partial_e(T(E))$ ($i=1,2$), 
and with an obvious meaning of the notation. 
Define 
\beq
C^{(2)}=\{(c, f)\in C^{(1)}\oplus C(\Omega, M_4): \phi_1(c)=R(f)\},
\eneq
where $R: C(\Omega, M_4)\to C(Z, M_4)$ is the quotient map defined by $R(f)=f|_Z$
for all $f\in C(\Omega, M_4).$
For each $z\in Z,$ define $\tau^{(z)}(d)={\rm tr}_4((\phi_1(\Psi_2(d))(z))$
for all $d\in C^{(2)},$  where ${\rm tr}_4$ is the normalized trace 
on $M_4.$ 
By identifying $\phi_1(C^{(1)})$ with $E^{(1)}\oplus E^{(2)},$ we may 
write that
\beq\label{exm-e-0}
\tau^{(z)}=(1/2)({\rm tr}_2(P_1(\phi_1(\Psi_2(d))(z))))+(1/2)({\rm tr}_2(P_2(\phi_1(\Psi_2(d))(z))))
\eneq
for all $d\in C^{(2)},$
where $P_i: C^{(1)}\to E_i$ is the quotient map ($i=1,2$) and 
${\rm tr}_2$ is the tracial state on $M_2.$  So $\tau^{(z)}\not\in \partial_e(T(C^{(2)})).$
It is easy to verify that  $\partial_e(T(C^{(2)}))$ is locally compact. 
But since each $z\in Z$ is a cluster point of $\Omega\setminus Z,$ one checks easily 
that $\tau^{(z)}\in \overline{\partial_e(T(C^{(2)}))}.$ 
%
Note that, if $z=y\in Y,$   then, by \eqref{exm-e-0},  we may write that (as $y^{(1)}=y\in Y^{(1)}$ and 
$y^{(2)}=y\in Y^{(2)},$ abusing notation)
\beq\label{exm-e-1}
\tau^{(z)}=(1/2)\tau_{y^{(1)}}+(1/2)\tau_{y^{(2)}},\,\,  {\rm where}\,\, y^{(i)}\in Y^{(i)},\,\, i=1,2.
\eneq
To see that $\partial_e(C^{(2)})$ does not have property (T1),
fix $\ep>0.$ Let $K\subset \partial_e(T(C^{(2)}))$ be any compact subset.
We will show that one can always find $\tau\in  \overline{\partial_e(T(C^{(2)}))}\setminus \partial_e(T(C^{(2)}))$
such that  $\mu_\tau(K)<1-\ep$ (in fact $\mu_\tau(K)=0$). 

As before, we may write $\partial_e(C^{(2)})=\partial_e(T(C^{(1)}))\sqcup (\Omega\setminus Z),$ 
where $\partial_e(T(C^{(1)}))$ is relatively closed. 
We then write $K=K_1\cup K_2,$ where $K_1\subset \partial_e(T(C^{(1)}))$ is compact
and 
$K_2\subset \Omega\setminus Z.$ 
 Moreover $K_2$ 
is also compact (as none of $\tau^{(z)}$ are extremal points).
There is a relatively  open subset $O\supset K$  (of $\partial_e(T(C^{(2)}))$) such that ${\bar O}$ is compact.
Since $Y\cup\{\tau^+, \tau^-\}$ is not compact, 
there is $y\in Y$ such that both $y^{(1)}, y^{(2)}\in 
\partial_e(T(C^{(2)})) \setminus {\bar O}$ (abusing notation as above).  As $\partial_e(T(C^{(2)}))$ has property (T3) (see Proposition \ref{PT3}), 
choose $f\in \Aff(T(A))$ such that $f|_K=0$ and $f|_{\partial_e(T(C^{(2)}))\setminus O}=1.$
By \eqref{exm-e-1}, $\tau^{(y)}\in {\rm conv}(\partial_e(T(C^{(2)}))\setminus O).$ 

Now let $g=1-f\in \Aff(T(C^{(2)})).$ Then $g(\tau^{(y)})=0$ and $g|_K=1.$ Thus
\beq
\mu_{\tau^{(y)}}(K)=\int_K d\mu_{\tau^{(y)}}\le  \int_{\partial_e(T(C^{(2)}))} g d\mu_{\tau^{(y)}}=g(\tau^{(y)})=0.
\eneq

\end{exm}

\begin{rem}\label{Rsimplexex}
 \CA s $C,$  $E,$  $D,$   $C^{(n)},$ and $C^{(2)}$ are not simple. 
From \cite{Btrace}, any metrizable Choquet  simplex  $S$
can be realized as a tracial state space of a unital simple AF-algebra.
In other words, simplexes in  examples of  \ref{Exsimplex} can all be realized as 
the tracial state space of some 
separable simple \CA s.
%
In fact 
$S$  can also be realized as a tracial state space of a unital (or stably projectionless) 
separable  simple \CA\, $A$ which has  arbitrary  $K_1(A)$-group (and any compatible $K_0(A)$)
 (see, for example, \cite{GLNI}, \cite{eglnkk0} and \cite{GLIII}).  

%
\end{rem}

\begin{df}\label{Dscdim}

 Let $A$ be a separable simple \CA\, with $\wtd T(A)\setminus \{0\}\not=\emptyset.$


(1) We say that $\wtd T(A)$ has a {\em $\sigma$-compact countable-dimensional extremal boundary,} 
if $\wtd T(A)$ has a basis $S$ such that $\partial_e(S)$ is $\sigma$-compact and countable-dimensional.




(2) We say that $\wtd T(A)$ has countable extremal boundary, 
if $\wtd T(A)$ has a basis $S$  such that $\partial_e(S)$ has  countably many points.

It follows from part (1) of \cite[Proposition 2.17]{Linoz} 
that the above definitions (1) and (2)
 do not depend on the choice of the basis $S.$
\end{df}

\section{Tracial approximate commutativity}

The main purpose of this section is to present Proposition \ref{Ldvi-3} which plays a complementary role 
to the  central surjectivity of Sato.  One may notice that both the $C^*$-norm and a trace 2-norm are used 
for $l^\infty(A)/I_{_{F, \varpi}}$ in this section (and later sections).
%
%

%
%

\begin{lem}\label{Ldvi-1}
Let $A$ be a   non-elementary separable algebraically simple  \CA\, 
with non-empty $QT(A)$ 
and 
$\{a_n\}\in l^\infty(A)_+^{\bf 1}$  be
such that  $p=\Pi_{_{\N}}(\{a_n\})$ is a 
projection.  Suppose that 
$A$ has strict comparison and  the canonical map $\Gamma: \Cu(A)\to {\rm LAff}_+(\Qw)$ is surjective.
Then, for any $n\in \N,$ there exists a 
unital \hm\, $\psi: M_n\to p(l^\infty(A)/I_{_{\varpi}})p.$
\end{lem}

\begin{proof}
Since $p$ is a projection, by \cite[Proposition 2.22]{Linoz}, we may assume that $\{a_n\}$ is a permanent projection lifting 
of $p.$ It follows from \cite[Proposition 2.22 (iii)]{Linoz}(see also the claim after (e.2.44) in the proof of \cite[Proposition 2.22]{Linoz})
that $\lim_{k\to\infty}\omega(a_k)=0$ (see Definition \ref{DefOS1}).
By \cite[Lemma 8.4]{FLosc}, for each $n\in \N,$ there is, for each $k\in \N,$ 
an order zero \cpc\, $\phi_k: M_n\to \Her(a_k)$ such that
\beq\label{Ldvi-1-e-1}
\|a_k\phi_k(1_n)-a_k\|_{_{2, \Qw}}<\sqrt{\omega(a_k)+1/k^2}.
\eneq
%
%
%
Define $\psi: M_n\to p(l^\infty(A)/I_{_{\Qw, \N}})p$ by $\psi(x)=\Pi_{_{\N}}\circ \{\psi_k(x)\}$
for $x\in M_n.$   By \eqref{Ldvi-1-e-1}, 
\beq
\psi(1_n)p=p.
\eneq
It follows that $\psi(1_n)=p.$ This implies that $\psi$ is a \hm.
\end{proof}

\begin{lem}\label{Ldvi-2}
Let $A$ be a  non-elementary separable algebraically simple   \CA\, 
with non-empty $QT(A)$ 
and $a\in A_+^{\bf 1}.$   Suppose that $A$ has strict comparison and T-tracial oscillation zero.
Then, for any $n\in \N,$ any $\ep>0,$  there exists 
a sequence of order zero \cpc\,  $\psi_k: M_n\to   \Her(a)$  ($k\in \N$)
such that 
\beq
&&\|[\Pi_{_{\Qw, \N}}(\{\psi_k(y)\}),\, \Pi_{_{\Qw, \N}}(\iota(a))]\|<\ep\tforal y\in M_n^{\bf 1}\tand\\
&&\|\Pi_{_{\Qw, \N}}(\{\psi_k(1_n)\}-\iota(a))\|<\ep.
\eneq
\end{lem}

\begin{proof}
Let ${\bar a}=\Pi_{_{\N}}(\iota(a)).$ By Theorem 6.4 of \cite{FLosc},  $l^\infty(A)/I_{_{\Qw,\N}}$ has real rank zero.
Therefore there are mutually orthogonal  projections 
$p_1, p_2,...,p_m\in \overline{{\bar a}( l^\infty(A)/I_{_{\Qw,\N}}){\bar a}}$ and $\lambda_1,\lambda_2,...,\lambda_m\in (0,1]$ such that
\beq
\|{\bar a}-\sum_{i=1}^m \lambda_i p_i\|<\ep.
\eneq
Since $A$ has strict comparison and T-tracial approximate oscillation zero, by \cite[Theorem 1.1]{FLosc},
the canonical map $\Gamma$ is surjective. 
By Lemma \ref{Ldvi-1}, for each $n\in \N,$ there is a unital \hm\, $\phi_i: M_n\to p_i(l^\infty(A)/I_{_{\Qw, \N}})p_i$
($1\le i\le m$).
Define  an order zero \cpc\, $\phi: M_n\to \overline{{\bar a}( l^\infty(A)/I_{_{\Qw,\N}}){\bar a}}$ 
by $\phi(y)=\sum_{i=1}^m \lambda_i \phi_i(y)$ for $y\in M_n.$
Note that $\phi(1_n)=\sum_{i=1}^m \lambda_i p_i.$ 
Then,   for all $y\in M_n,$
\beq
\phi(y)\sum_{i=1}^m\lambda_i p_i=\sum_{i=1}^m \lambda_i^2 \phi_i(y)p_i=
\sum_{i=1}^m \lambda_i^2 p_i\phi_i(y)=(\sum_{i=1}^m\lambda_i p_i)\phi(y).
\eneq
It follows  that (for all $y\in M_n^{\bf 1}$)
\beq
\|[{\bar a},\, \phi(y)]\|<\ep\andeqn \|\bar a-\phi(1_n)\|<\ep.
\eneq
By \cite[Proposition 1.2.4]{WcovII},  there is a sequence of order zero \cpc s $\psi_k: A\to \Her(a)$ such that 
$\phi=\Pi_{\Qw, \N}\circ \{\psi_k\}.$ 
The lemma then follows.
\end{proof}

\begin{lem}\label{Ldvi-3}
Let $A$ be a  separable algebraically simple  \CA\, 
which has strict comparison, T-tracial approximate oscillation zero and a  nonempty 
$QT(A),$ 
and 
let $D$ be a finite dimensional \CA.
 Suppose 
that $\phi: D\to A$ is an order zero \cpc.
Then, for any 
$\ep>0$ and   any $n\in \N,$  
 there is a sequence of  order zero \cpc s
 $\psi_k: M_n\to \overline{\phi(1_D)A\phi(1_D)}$  such that
 \beq\nonumber
&& \|[\Pi_{_{\Qw, \N}}(\iota(\phi(x))),\,\,\Pi_{_{\Qw, \N}}(\{\psi_k(y)\})]\|<\ep\tforal x\in D^{\bf 1}\tand y\in M_n^{\bf 1},\\\nonumber
 &&\hspace{-0.4in}\tand \|\Pi_{_{\Qw, \N}}(\iota(\phi(1_D)))-\Pi_{_{\Qw, \N}}(\{\psi_k(1_n)\})\|<\ep.
 \eneq
\end{lem}

\begin{proof}
Let  $D=M_{r_1}\oplus M_{r_2}\oplus \cdots \oplus M_{r_l}$ and 
$\{e_{s, i,j}\}_{1\le i,j\le l}$  be a system of matrix units for $M_{r_s},$ $1\le s\le l.$
Let $a_{s,1,1}=\phi(e_{s,1,1}).$ 
Put $R=l^2 r_1^2\cdots r_l^2.$ 
It follows from Lemma \ref{Ldvi-2} that, for any $\ep>0,$ 
there exists a sequence of  order zero \cpc s $\phi_{s,k}: M_n\to \Her(a_{s,1,1})$
such that (for $1\le s\le l$)
\beq\label{Ldvi-3-e-5}
&&\hspace{-0.4in}\|[\Pi_{_{\Qw, \N}}(\iota(a_{s, 1,1})),\, \Pi_{_{\Qw, \N}}(\{\phi_{s,k}(y)\})]\|<\ep/2R\rforal y\in M_n^{\bf 1}\andeqn\\\label{Ldvi-3-e-6}
&&\hspace{-0.4in}\|\Pi_{_{\Qw, \N}}(\iota(a_{s, 1,1})-\{\phi_{s,k}(1_n)\})\|<\ep/2R.
\eneq
Note that $\overline{\phi(1_{M_{r(s)}})C\phi(1_{M_{r(s)}})}\cong M_{r(s)}(\overline{a_{s,1,1}Ca_{s,1,1}})$
(see, for example, Proposition 8.3 of \cite{FLosc}), $1\le s\le l.$
Define $\psi_{s,k}: M_n\to \overline{\phi(1_{M_{r_s}})A\phi(1_{M_{r_s}})}\cong M_{r(s)}(\overline{a_{s,1,1}Ca_{s,1,1}})$ by 
\beq
\psi_{s,k}(y)=\diag(\overbrace{\phi_{s,k}(y), \phi_{s,k}(y), ...,\phi_{s,k}(y)}^{r_s})\rforal y\in M_n\,\, (1\le s\le l),
\eneq
and defined $\psi_k: M_n \to \overline{\phi(1_D)A\phi(1_D)}$ by 
$\psi_k(y)=\sum_{i=1}^l \psi_{s,k}(y)$ for all $y\in M_n$ and $k\in \N.$
 We check, by \eqref{Ldvi-3-e-5} and \eqref{Ldvi-3-e-6}  that
\beq\nonumber
&&\hspace{-0.3in}\|[\Pi_{_{\Qw, \N}}(\iota(\phi(x))),\, \Pi_{_{\Qw, \N}}(\{\psi_k(y)\})]\|<\ep\rforal x\in D^{\bf 1}\andeqn y\in M_n^{\bf 1},\andeqn\\\nonumber
&& \|\Pi_{_{\Qw, \N}}(\iota(\phi(1_D))-\{\psi_k(1_n)\})\|<\ep.
\eneq
\end{proof}

\begin{lem}\label{Lcomm-1}
Let $A$ be a  separable algebraically simple \CA\,  which has strict comparison, T-tracial approximate 
oscillation zero and 
a non-empty 
subset $F,$ and a compact subset $K$ of $QT(A)$ such that 
$F\subset K\subset
QT(A).$ 
Let 
$D$
be a finite dimensional \CA\, and 
$\phi: D\to l^\infty(A)/I_{_{\Qw, \varpi}}$  be a \hm.
Then, for any $\ep>0,$ any finite subset $S\subset I_{_{F, \varpi}}/I_{_{\Qw, \varpi}},$ 
and any integer $n\ge 1,$ there exists a \hm\, $\psi: M_n\to  l^\infty(A)/I_{_{\Qw, \varpi}}$
such  that
\beq
&&\hspace{-0.2in}(1)\,\,
\|[\phi(x),\,  \psi(y)]\|=0
\tforal x\in D\tand y\in M_n,\\
&&\hspace{-0.2in}(2)\,\, \|[z,\, \psi(y)]\|<\ep\tforal z\in {\cal S}\tand y\in M_n^{\bf 1}\tand\\
&&\hspace{-0.2in}(3) \,\, \pi_F\circ \psi(1_{M_n})=1,
\eneq
where $\pi_F: l^\infty(A)/I_{_{\Qw, \varpi}}\to l^\infty(A)/I_{_{F, \varpi}}$ is the quotient map.
\end{lem}

\begin{proof}
Write $D=M_{r_1}\oplus M_{r_2}\oplus \cdots \oplus M_{r_l}$ and 
$\{e_{s, i,j}\}_{1\le i,j\le l}$ a system of matrix units for $M_{r_s},$ $1\le s\le l.$
Put $C=l^\infty(A)/I_{_{\Qw, \varpi}}.$
Let $p_{s,i,i}=\phi(e_{s,i,i})$ ($1\le i\le r_s$ and $1\le s\le l$) and $p=\phi(1_D).$
By Lemma \ref{Ldvi-1},  there exists 
a unital \hm\, $\phi_s': M_n\to \Her(p_{s,1,1}).$
Define $\psi_{s}': M_n\to \phi(1_{M_{r_s}})C \phi(1_{M_{r_s}})$ by  ($1\le s\le l$)
\beq
\psi_{s}'(y)=\diag(\overbrace{\phi_{s}'(y), \phi_{s}'(y), ...,\phi_{s}'(y)}^{r_s})\rforal y\in M_n,
\eneq
and defined $\psi': M_n \to pCp$ by 
$\psi'(y)=\sum_{s=1}^l \psi_{s}'(y)$ for all $y\in M_n.$ 
By Proposition \ref{Pcompact}, $l^\infty(A)/I_{_{K, \varpi}}$ is unital.
Hence its quotient $l^\infty(A)/I_{_{F, \varpi}}$ is also unital.
Let $B=(1-p)C(1-p)$ and $\bar B=\Pi_{_{F, \varpi}}(B).$ Then $\bar B$ is unital. Denote by $e_{\bar B}$ the unit.
Then $e_{\bar B}+\pi_F(p)$ is the unit for $l^\infty(A)/I_{_{F, \varpi}}.$ 
Recall that $A$ has T-tracial approximate oscillation zero. By Theorem 6.4 of \cite{FLosc}, $l^\infty(A)/I_{_{\Qw, \varpi}}$
has real rank zero. So does its quotient $C.$ Hence 
$B$ has real rank zero. There is a projection $e_B\in B$ such that $\pi_F(e_B)=e_{\bar B}$ (\cite[Theorem  3.14]{BP}). 
Put $e=e_B+p$ which is a projection. Note that $\pi_F(e)=e_{\bar B}+\pi_F(p)$ is the unit  for $l^\infty(A)/I_{_{F, \varpi}}.$

By Lemma \ref{Ldvi-1} again, there exists a unital \hm\, $\psi'': M_n\to e_BCe_B.$
Let $q_i=\psi''(e_{i,i}),$ $1\le i\le n.$ 

Denote $J_F=I_{_{F, \varpi}}/I_{_{\Qw, \varpi}}.$  
%
Let $C_0$ be the \SCA\, of $C$ generated by $\phi(D),$ $\phi'_s(M_n),$ $\psi'_s(M_n)$ ($1\le s\le l$), $\psi''(M_n)$ and 
$S.$
Then $S\subset C_0\cap J_F.$
  Let $s_0\in C_0\cap J_F$ be a strictly positive element of $C_0\cap J_F.$
Let $J_1=\overline{s_0J_Fs_0}$ and  $C_1$ be the \CA\, generated by $C_0$ and $J_1.$
Then $J_1$ is a $\sigma$-unital ideal of  \CA\, 
$C_1.$   Moreover $J_1$ is a $\sigma$-unital hereditary \SCA\, of $J_F$ which has real rank zero. 

By (3) Lemma \ref{Lstrictp},  $p_{s,i,i}s_0p_{s,i,i},$ $q_js_0q_j$  
and $(1-e)s_0(1-e)$ are  strictly positive elements of $p_{s, i,i}J_1p_{s, i,i},$ 
$q_jJ_1q_j$ and $(1-e)J_1(1-e),$  respectively ($1\le i\le r(s),$ $1\le s\le l,$ and $1\le j\le n$).
Let $\{d_{s,i, k}: k\in \N\}\subset p_{s, i,i}Cp_{s, i,i},$  
$\{g_{i,k}: k\in \N\}\subset q_iCq_i$  and $\{f_k\}\subset (1-e)J_1(1-e)$ be approximate identities consisting 
of projections for $p_{s, i,i}J_1 p_{s, i,i},$ $q_iJ_1q_i$ and $(1-e)J_1(1-e),$ 
respectively ($1\le s\le l$ and $1\le i\le n$).
Put 
\beq
h_k'=\sum_{i=1}^n (\sum_{s=1}^l d_{s, i, k}+g_{i,k})\andeqn h_k=h_k'+f_k,\,\, k\in \N.
\eneq
Then, by Lemma \ref{Lstrictp}, $\{h_k\}$ forms an approximate identity for $J_F.$ 
Choose $k_0\in \N$ such that
\beq\label{Lcomm-1-5}
\|s(1-h_{k})\|<\ep/4(\max_{s\in S}\|s\|+1)\rforal s\in S\andeqn k\ge k_0.
\eneq
Choose $k>k_0.$ Then $f_k-f_{k_0}$ is a non-zero projection in $J_F$
(If $(1-e)J_F(1-e)$ were unital, we let $\psi_J=0$).
By Lemma \ref{Ldvi-1}, there is a unital \hm\, $\psi_J: M_n\to (f_k-f_{k_0})J_F(f_k-f_{k_0}).$ 
From the construction, it is clear that 
\beq
h_k\phi(d)&=&\phi(d)h_k\rforal  d\in D\andeqn k\in \N\andeqn\\
h_k(\psi'(b)+\psi''(b))&=&(\psi'(b)+\psi''(b))h_k\rforal b\in M_n\andeqn k\in \N.
\eneq
Now define $\psi: M_n\to eCe$ by 
\beq
\psi(b)=(\psi'(b)+\psi''(b))(e-h_{k_0}) +\psi_J(b)\rforal b\in M_n.
\eneq
It follows that $\psi(1_n)=e-h_{k_0}+(f_k-f_{k_0}).$
Then $\psi$ is a \hm\, and
\beq
\phi(x)\psi(y)=\psi(y)\phi(x)\rforal x\in D\andeqn y\in M_n.
\eneq
It follows from \eqref{Lcomm-1-5} that, for any $y\in M_n^{\bf 1},$  
\beq\nonumber
&&\hspace{-0.2in}\|[s,\, \psi(y)]\|^2=\|s((e-h_{k_0})+(f_k-f_{k_0}))\psi(y)-\psi(y)((e-h_{k_0})+(f_k-f_{k_0}))s\|^2\\\nonumber
&&\hspace{0.9in}\le 4\|s((e-h_{k_0})+(f_k-f_{k_0}))\|^2=4\|s((e-h_{k_0})+(f_k-f_{k_0}))^2s^*\|\\\nonumber
&&\hspace{0.9in}\le 4\|s(1-h_{k_0})s^*\|=4\|s(1-h_{k_0})\|^2
<\ep^2\hspace{0.2in} \rforal s\in S.
\eneq
So far we have shown that (1) and (2) hold. 
Since  $\psi(1_n)=e-h_{k_0}+(f_k-f_{k_0})$ and $h_{k_0}, f_k, f_{k_0}\in J_F,$
$\pi_F\circ \psi$ is unital.  So (3) holds and the lemma follows.
\end{proof}

\begin{prop}\label{Ldvi-3}
%
Let $A$ be a  separable 
algebraically simple \CA\,  with 
strict comparison, T-tracial approximate oscillation zero, 
and a non-empty  compact $T(A),$ 
and  let
$F\subset \partial_e(T(A))$ be a compact subset.
Suppose that $\ep\in (0,1/2),$ ${\cal F}\subset A^{\bf 1}$ is a finite subset, 
$D$ is a finite dimensional \CA\, and 
$\phi: D\to l^\infty(A)/I_{_{F, \varpi}}$ is a \hm\, 
such that
\beq\label{Ldvi-3-e-1}
\|\Pi_{_{F, \varpi}}(\iota(x))-\phi(y_x)\|_{_{2, F_\varpi}}< \ep/2\tforal x\in {\cal F}\,\, \,{and\,\, some}\,\, y_x\in D^{\bf 1}.
\eneq
Then, for any integer $n\ge 1,$ there exists a \hm\, $\psi: M_n\to  l^\infty(A)/I_{_{\Tw, \varpi}}$
such  that
\beq
&&\hspace{-0.2in}(1)\,\,
\|[\Pi_{_{\varpi}}(\iota(x)),\,  \psi(y)]\|_{_{2, T(A)_\varpi}}<\ep\tforal x\in {\cal F}\tand y\in M_n^{\bf 1},\tand\\
&&\hspace{-0.2in}(2) \,\, \pi_F\circ \psi(1_{M_n})=1,
\eneq
where $\pi_F: l^\infty(A)/I_{_{\Tw, \varpi}}\to  l^\infty(A)/I_{_{F, \varpi}}$ is the quotient map.
%
\end{prop}

\begin{proof}
For each $x\in {\cal F},$ fix a pair $(x, y_x)$ such that \eqref{Ldvi-3-e-1} holds.
Fix $n\in \N.$  
Put $C=l^\infty(A)/I_{_{\Tw, \varpi}}$ and $C_F=l^\infty(A)/I_{_{F, \varpi}}.$
Let $\pi_F: C\to C_F$ be the quotient map. 
Since $A$ has  T-tracial approximate oscillation zero, by Theorem 6.4 of \cite{FLosc}, 
$C$ has real rank zero.   By Elliott's lifting lemma (see Lemma \ref{Lelliott}),
there is a \hm\, $\phi_D: D\to C$ such that $\pi_F\circ \phi_D=\phi.$ 

It follows from \cite[Theorem 7.11]{FLosc}, $\Gamma$ is surjective. 
Then, by  Theorem \ref{Ptight}, $T(A)$ has has property (TE).
Therefore, 
for each $x\in {\cal F},$ there is $j_x\in J_F=I_{_{F, \varpi}}/I_{_{T(A), \varpi}}$ such that
\beq
\|\Pi_\varpi(\iota(x))-\phi_D(y_x)+j_x\|_{_{2, T(A)_\varpi}}<\ep/2.
\eneq
Put  
\beq\label{Cdvi-4-4}
\eta=\ep/2-\max\{\|\Pi_\varpi(\iota(x))-\phi_D(y_x)+j_x\|_{_{2, T(A)_\varpi}}:x\in {\cal F}\}>0.
\eneq

By Lemma \ref{Lcomm-1}, there is a \hm\, $\psi: M_n\to C$ such that
\beq\label{Cdvi-4-5-}
&&\|[\phi_D(b),\, \psi(z)]\|=0\rforal b\in D\andeqn z\in M_n,\\\label{Cdvi-4-5}
&&\|[j_x, \, \psi(z)]\|<\eta/2\rforal x\in {\cal F}\andeqn z\in M_n^{\bf 1}
\eneq
and $\pi_F\circ \psi$ is unital (so (2) holds). 
It follows that, by \eqref{Cdvi-4-5-}, \eqref{Cdvi-4-5}  and by  \eqref{Cdvi-4-4}, for all $x\in {\cal F}$ and $z\in M_n^{\bf 1},$ 
\beq\nonumber
&&\hspace{-0.4in}\|\Pi_\varpi(\iota(x))\psi(z)-\psi(z)\Pi_\varpi(\iota(x))\|_{_{2, T(A)_\varpi}}
\le \|(\phi_D(y_x)-j_x)\psi(z)-\psi(z)(\phi_D(y_x)-j_x)\|_{_{2, T(A)_\varpi}}\\\nonumber
&&\hspace{0.7in}+\|
(\Pi_\varpi(\iota(x))-\phi_D(y_x)+j_x)\psi(z)-\psi(z)(\Pi_\varpi(\iota(x))-\phi_D(y_x)+j_x)\|_{_{2, T(A)_\varpi}}\\\nonumber
&&\hspace{0.7in}<\eta/2+\|(\Pi_\varpi(\iota(x))-\phi_D(y_x)+j_x)\psi(z)\|_{_{2, T(A)_\varpi}}\\\nonumber
&&\hspace{1.8in}
+\|\psi(z)(\Pi_\varpi(\iota(x))-\phi_D(y_x)+j_x)\|_{_{2, T(A)_\varpi}}\\\nonumber
&&\hspace{0.7in}<\eta/2+2(\ep/2-\eta)<\ep.
\eneq
Thus (1) holds. The lemma follows.
%
%
%
%
\end{proof}

\section{Semi-projectivity in 2-norm}

In this section we consider the stability of the projective \CA s in trace 2-norm (Proposition \ref{PPproj}) and present Lemma  \ref{Lsemip}.

\begin{lem}\label{Loz-hom}
Let $A$ be a separable algebraically simple \CA\, with T-tracial approximate oscillation  zero, 
$S\subset QT_{(0,1]}(A)$ be a subset  and
$D$ be a finite dimensional \CA.  Suppose that $\phi: D\to l^\infty(A)/I_{_{S, \varpi}}$
is an order zero \cpc. 
Then, for any $\ep>0,$ there exists a   finite dimensional \CA\, $D_1$
and a \hm\, $h: D_1\to \overline{\phi(1_D)( l^\infty(A)/I_{_{S, \varpi}})\phi(1_D)}$ such that, for any $x\in D^{\bf 1},$
there is $y_x\in D_1^{\bf 1}$ satisfying
\beq
\|\phi(x)-h(y_x)\|<\ep.
\eneq
If, in addition, $\phi(1_D)$ is a projection, one may require  that
$h(1_{D_1})=\phi(1_D).$
\end{lem}

\begin{proof}
Write $D=M_{r(1)}\oplus M_{r(2)}\oplus \cdots \oplus M_{r(l)}.$
Let $\{e_{s, i,j}: 1\le i,j\le l\}$ be a system of matrix units for $M_{r(s)},$ $1\le s\le l.$
Put $C=l^\infty(A)/I_{_{S, \varpi}},$  $R=l^2r(1)^2r(2)^2\cdots r(l)^2,$  and 
$a_{s, i,j}=\phi(e_{s, i,j}),$ $1\le i,j\le r(s)$ and $1\le s\le l.$
It follows from Theorem 6.4 of \cite{FLosc} that $l^\infty(A)/I_{_{\Qw, \N}}$
has real rank zero. Therefore its quotient $l^\infty(A)/I_{_{S, \varpi}}$ also has real rank zero.
As 
in the proof of Lemma \ref{Ldvi-2}, there is a commutative finite dimensional \CA\, $D_{s,0}$
and an injective \hm\, $\psi_{s, 1,1}: D_{s,0}\to \overline{a_{s,1,1}Ca_{s,1,1}}$
and  $d_{s,1,1}\in D_{s,0}^{\bf 1}$ such that
\beq\label{Loz-hom-2}
\|a_{s,1,1}-\psi_{s, 1,1}(d_{s,1,1})\|<\ep/2R,\,\, 1\le s\le l
\eneq
(this does not require that $A$ has strict comparison).
Note that $\overline{\phi(1_{M_{r(s)}})C\phi(1_{M_{r(s)}})}\cong M_{r(s)}(\overline{a_{s,1,1}Ca_{s,1,1}})$
(see, for example, Proposition 8.3 of \cite{FLosc}), $1\le s\le l.$
Moreover\\ $M_{r(s)}(\psi_{s,1,1}(D_{s,0}))\cong M_{r(s)}(D_{s,0})$ and this 
isomorphism gives an injective \hm\, 
$h_s: D_{s,0}\otimes M_{r(s)}\to \overline{\phi(1_{M_{r(s)}})C\phi(1_{M_{r(s)}})}$
such that $h_s(D_{s,0}\otimes e_{1,1})=\psi_{s,1,1}.$
Define $D_1=\bigoplus_{s=1}^l D_{s,0}\otimes M_{r(s)}$ and 
define a \hm\, $h: D_1\to \overline{\phi(1_D)C\phi(1_D)}$ by
$h|_{D_{s,0}\otimes M_{r(s)}}=h_s,$ $1\le s\le l.$
By \eqref{Loz-hom-2}, one checks that this \hm\, $h$ meets the requirements of the lemma.

If, $\phi(1_D)$ is a projection, one may choose $D_{s,0}$ above such that 
$1_{D_{s,0}}=\phi(e_{s,1,1})=a_{s, 1,1}$  which is also a projection.
Then the proof above implies that $h(1_{D_1})=\phi(1_D).$
\end{proof}

\begin{lem}\label{Lsemi-0}
Let $D$ be a finite dimensional \CA, $A$ be a separable simple \CA\, 
and $F\subset QT_{(0,1]}(A)$ a subset.
Then, for any $\ep>0,$ there exists $\dt>0$ satisfying the following:

Suppose that $\phi: D\to A$ is a \cpc\, 
such that 
\beq
\|\phi(a)\phi(b)\|_{_{2, F}}<\dt\tforal a, b\in D^{\bf 1} \,\,{\rm with}\,\, ab=0.
\eneq
Then there exists 
 an order zero \cpc\,  $\psi: D\to \Her(\phi(1_D))$ such that
\beq
\|\phi(x)-\psi(x)\|_{_{2, F}}<\ep\rforal x\in D^{\bf 1}.
\eneq
\end{lem}

\begin{proof}
Suppose the lemma is false.
Then, there exist $\ep_0>0,$ 
a sequence of \cpc s $\phi_k: D\to A$ 
and a sequence $\dt_k\in (0,1/2)$ such that $\sum_{k=1}^\infty \dt_k<\infty$
and, for any orthogonal pairs $a, b\in D_+^{\bf 1},$ 
\beq
&&\|\phi_k(a)\phi_k(b)\|_{_{2, F_\varpi}}<\dt_k,\,\, k\in \N,\andeqn\\\label{Lsemi1-1}
&&\sup\{\|\phi_k(x)-\psi(x)\|_{_{2, F}}: x\in D^{\bf 1}\}\ge\ep_0
\eneq
for any order zero \cpc\,  $\psi: D\to \Her(\phi_k(1_D)).$

Then $\{\phi_k(a)\phi_k(b)\}_{k\in \N}\in I_{_{F, \N}}$ for any orthogonal pairs $a, b\in D^{\bf 1}_+.$ Let $e=\Pi_{_{F, \N}}(\{\phi_k(1_D)\})\in l^\infty(A)/I_{_{F, \N}}.$
Define $\wtd \phi: D\to  \overline{e(l^\infty(A)/I_{_{F, \N}})e}$ by $\wtd\phi(x)=\Pi_{_{F, \N}}(\{\phi_k(x)\})$ for all 
$x\in D.$ 
Then $\wtd\phi$ is an order zero \cpc. 
By \cite[Proposition 1.2.4]{WcovII},
there exists 
an order zero \cpc\, $\Psi=\{\psi_n\}: D\to l^\infty(A)$ 
such that $\Pi_{_{F, \N}}\circ \Psi=\psi.$ Therefore 
\beq
\lim_{n\to\infty}\|\phi_n(x)-\psi_n(x)\|_{_{2, F}}=0\rforal x\in D.
\eneq
Since $D^{\bf 1}$ is compact, this implies that
\beq
\lim_{n\to\infty}\sup\{\|\phi_k(x)-\psi_n(x)\|_{_{2, F}}: x\in D^{\bf 1}\}=0.
\eneq
This contradicts \eqref{Lsemi1-1}. The lemma follows.
\end{proof}

One may want to compare the following with \cite[Lemma 1.2]{WcovII}.

\begin{cor}\label{Lsemip-L}
Let $D$ be a finite dimensional \CA, $A$  a separable simple \CA\, 
and $F\subset QT_{(0,1]}(A)$ a subset.
Then, for any $\ep>0,$ there exists $\dt>0$ satisfying the following:

Suppose that $\phi: D\to A$ is an order zero \cpc\, and 
$e\in A_+^{\bf 1}$  such that 
\beq
\|[e,\, \phi(x)]\|_{_{2, F}}<\dt\tforal x\in D^{\bf 1}.
\eneq
Then there exists 
 an order zero \cpc\,  $\psi: D\to \Her(e)$ such that
\beq
\|e\phi(x)e-\psi(x)\|_{_{2, F}}<\ep\rforal x\in D^{\bf 1}.
\eneq
\end{cor}

\begin{proof}
The corollary follows from Lemma \ref{Lsemi-0} by considering the  \cpc\, $\phi': D\to A$ defined by
$\phi'(x)=e\phi(x)e$ for all $x\in D.$ 
%
\end{proof}

The following  proposition is not used in this paper. We would like to observe  that the proof 
of it is contained in that of Proposition \ref{Lsemi-0}.

\begin{prop}\label{PPproj}
Let $A$  be a separable simple \CA,
$F\subset QT_{(0,1]}(A)$ a subset
and $C$ be a separable projective \CA.
Then, for any $\ep>0$ and any finite subset ${\cal F}\subset C,$ there exist $\dt>0$ and a finite 
subset ${\cal G}\subset C$ such that, if $\phi: C\to A$ is a \cpc\, satisfying 
\beq
\|\phi(a)\phi(b)-\phi(ab)\|_{_{2, F}}<\dt\rforal a, b\in {\cal G},
\eneq
then there exists a \hm\, $h: C\to A$ such that
\beq
\|h(x)-\phi(x)\|_{_{2, F}}<\ep\tforal x\in {\cal F}.
\eneq
\end{prop}

\begin{lem}\label{Lsemip}
Let $D$ be a finite dimensional \CA, $A$ a separable  algebraically simple \CA\, 
with T-tracial approximate oscillation zero 
and $F\subset QT_{(0,1]}(A)$  a subset.

Then, for any $\ep>0,$ there exists $\dt>0$ satisfying the following:

Suppose that $\phi: D\to l^\infty(A)/I_{_{F, \varpi}}$ is an order zero \cpc\,
and 
$\{e_n\}\in l^\infty(A)_+^{\bf 1}$  
such that 
\beq
\|[e,\, \phi(x)]\|_{_{2, F_\varpi}}<\dt\tforal x\in D^{\bf 1},
\eneq
where $e=\Pi_{_{F, \varpi}}(\{e_n\}).$
Then there exists   a finite dimensional \CA\, $D_1$ and 
 a  \hm\, $\psi: D_1\to e(l^\infty(A)/I_{_{F, \varpi}})e$ such that, for any $x\in D^{\bf 1},$ there exists $y_x\in D_1^{\bf 1}$ 
 such that
\beq
\|e\phi(x)e-\psi(y_x)\|_{_{2, F_\varpi}}<\ep.
\eneq
If, in addition,   $e$ is a projection, 
then one may require that $\psi(1_{D_1})=e.$
\end{lem}

\begin{proof}
Fix $\ep>0.$ Let $\dt$ be the constant given by Lemma \ref{Lsemip-L}  associated 
with $D$ and $\ep/3.$   Suppose that $\Phi=\{\phi_n\}: D\to l^\infty(A)$ is an 
order zero map such that $\Pi_{_{F, \varpi}}\circ \Phi=\phi.$
 Suppose that $e$ is as described in the lemma.
Then, there is ${\cal P}\in \varpi$ such that, for any $n\in {\cal P},$
\beq
\sup\{\|[e_n,\, \phi_n(x)]\|_{_{2, F}}: x\in D^{\bf 1}\}<\dt.
\eneq
Applying Lemma \ref{Lsemip-L}, we obtain, for each $n\in {\cal P},$ an order zero \cpc\, 
$\psi_n: D\to \Her(e_n)$ such that 
\beq\label{Lsemip-2-5}
\sup\{\|e_n\phi(x)e_n-\psi_n(x)\|_{_{2, F}}: x\in {\cal D}^{\bf 1}\}<\ep/3.
\eneq
We then obtain a sequence of order zero \cpc s $\{\psi_k\}: D\to  \overline{\{e_k\}(l^\infty(A))\{e_k\}}$
such that $\psi_k$ is as so defined above when $k\in {\cal P}.$ 
Let $\psi'=\Pi_{_{F, \varpi}}(\{\psi_k\}).$ 
Then, by \eqref{Lsemip-2-5}, 
\beq
\|e\phi(x)e-\psi'(x)\|_{_{2, F_\varpi}}\le \ep/3\tforal x\in D^{\bf 1}.
\eneq
Applying Lemma \ref{Loz-hom}, one obtains a finite dimensional 
\CA\, $D_1$ and a  \hm\, $\psi: D_1\to e(l^\infty(A)/I_{_{F, \varpi}})e$ 
such that, for any $x\in D^{\bf 1},$ there exists $y_x\in D_1^{\bf 1}$ such that
\beq
\|\psi'(x)-\psi(y_x)\|_{_{2, F_\varpi}}<\ep/2.
\eneq
Therefore
\beq
\|e\phi(x)e-\psi(y_x)\|_{_{2, F_\varpi}}<\ep\rforal x\in D^{\bf 1}.
\eneq

If, in addition, $e$ is a projection,  then $e-\psi(1_{D_1})$ is a projection. 
Define $D_2=\C\oplus D_1$ and 
define $\wtd\psi: D_2\to l^\infty(A)/I_{_{F, \varpi}}$ by 
$\wtd \psi(\lambda, d)=\lambda(e-\psi(1_{D_1}))\oplus \psi(d)$
for $\lambda\in \C$ and $d\in D_1.$ Then 
$\wtd \psi(1_{D_2})=e$  and 
\beq
\|e\phi(x)e-\wtd\psi((0, y_x))\|_{_{2, F_\varpi}}<\ep\rforal x\in D^{\bf 1}.
\eneq
\end{proof}

\section{Finite dimensional approximation in 2-norm}

The purpose of this section is to present Proposition \ref{Pst-1}.

\begin{prop}\label{Ptesrk1}
Let $A$ be a separable \CA\, and $\tau\in T(A).$
Suppose that $\overline{\pi_\tau(A)}^{sot}$ is a hyper-finite type ${\rm II}_1$-factor. 
Then, for any $\ep>0,\, \sigma>0$ and  any finite subset ${\cal F}\subset A^{\bf 1},$ 
there exist
an integer $n>1,$ 
an open neighborhood $O$ of $\tau$ of $T(A)$ and 
an order zero \cpc\, $\phi: M_n\to A$ 
such
that


(1) $\|x-\phi(y_x)\|_{2, \bar O}<\ep$ for all $x\in {\cal F}$ and some $y_x\in M_n^{\bf 1},$ and 

(2) $\sup\{1-\tau(\phi(1_n)): \tau\in {\bar O}\}<\sigma.$
%
\end{prop}

\begin{proof}
Fix $\ep> 0, \, \sigma>0$ and a finite subset ${\cal F}\subset A.$ 
To simplify notation, we may assume that ${\cal F}\subset A^{\bf 1}.$ 
Put $N=\overline{\pi_\tau(A)}^{sot}.$
Then, since $N$ is a hyper-finite ${\rm II}_1$-factor,
we may also write $\overline{B}^{sot}=N,$ where $B$ is a UHF-algebra.
By Kaplansky's density theorem,
 there exists an integer $n>1$ and \SCA\, $M_n\subset N$  
with $1_{M_n}=1_N$ such 
that, for each $x\in {\cal F},$ there exists $b_x\in M_n^{\bf 1}$ such that
\beq
\|x-b_x\|_{_{2, \tau}}<\ep/8.
\eneq
Let $\{e_{i,j}\}\subset M_n$ be a system of matrix units for $M_n.$
By Kaplansky's density theorem again, there is, for each $k,$ 
an element $a_{k,i,j}\in A^{\bf 1}$ 
such that $\|a_{k,i,j}-e_{i,j}\|_{2, \tau}<1/k.$
It follows from Theorem 3.3 of \cite{KR}  that the canonical \hm\, 
$\Phi: l^\infty(A)/I_{\tau, \varpi}\to l^\infty(N)/I_{\tau, \varpi}(N)$ is an isomorphism.
Therefore  there 
is a unital \hm\, $\bar \phi: M_n\to l^\infty(A)/I_{\tau, \varpi}$ 
satisfying the following: for any $x\in {\cal F},$ there exists $y_x\in M_n^{\bf 1}$ such that 
\beq
\|\Pi_{\tau,\varpi}(\iota(x))-\bar \phi(y_x)\|_{_{2, \tau_\varpi}}<\ep/4.
\eneq
By \cite[Proposition 1.2.4]{WcovII},
there is an order zero \cpc\,  $\psi=\{\psi_k\}: M_n\to l^\infty(A)$
such that $\Pi_{_{\tau, \varpi}}\circ \{\psi_k\}=\bar \phi,$ 
\beq
&&\lim_{k\to\varpi}\|x-\psi_k(y_x)\|_{_{2, \tau}}<\ep/4 \rforal x\in {\cal F},\andeqn\\
&&\lim_{k\to\varpi} 1-\tau(\psi_k(1_n))=0.
\eneq
Therefore, 
 by choosing  $\phi=\psi_k$ 
 for some $k\in 
{\cal P}$ and some ${\cal P}\in \varpi,$
we have 
\beq
\|x-\phi(y_x)\|_{2, \tau}<\ep/2\andeqn 1-\tau(\phi(1_n))<\sigma/2.
\eneq

Hence, there is an open subset $O\subset T(A)$ and $\tau\in O,$ 
such that
\beq
\|x-\phi(y_x)\|_{_{2, \bar O}}<\ep\andeqn \sup\{1-
\tau(\phi(1_n)):\tau\in {\bar O}\}<\sigma.
\eneq
%
%
%
\end{proof}

\begin{lem}\label{LnFO}
Let $A$ be a separable 
algebraically simple \CA\, with T-tracial approximate oscillation zero
and let $K\subset  \partial_e(T(A))$ be a subset.
Let $\ep>0$ and ${\cal F}\subset A^{\bf 1}$ be a finite subset.
Suppose that there exist a finite dimensional \CA\, $D_0$ and  
a \hm\, $\psi: C_0((0,1])\otimes D_0\to  l^\infty(A)/I_{_{K, \varpi}}$
(or an order zero \cpc\, $\psi: D_0\to l^\infty(A)/I_{_{K, \varpi}}$)
such that
\beq\nonumber
\|\Pi_{_{K, \varpi}}(\iota(x))-\psi(y_x)\|_{_{2, K_\varpi}}<\ep/2\tforal x\in {\cal F}\,\,and\,\,some\,\, y_x\in (C_0((0,1])\otimes D_0)^{\bf 1}
\eneq
(or $y_x\in D_0^{\bf 1}$).
Then, 
there  are relatively
open subset $O\subset \partial_e(T(A))$ with $K\subset O,$ 
a finite dimensional \CA\, $D$ and a \hm\, 
$\phi: D\to l^\infty(A)/I_{_{\Tw, \varpi}}$ 
such that
\beq
\|\Pi_{_{F, \varpi}}(\iota(x))-\pi_F\circ \phi(y_x)\|_{_{2, F_\varpi}}<\ep\tforal x\in {\cal F}\,\,and\,\,some\,\, y_x\in D^{\bf 1},
\eneq
where $F=\bar O.$
If\, $l^\infty(A)/I_{_{F, \varpi}}$ is unital, one may also require that 
$\pi_F\circ \phi$ is a unital \hm, where $\pi_F: l^\infty(A)/I_{_{\Tw, \varpi}}\to l^\infty(A)/I_{_{F, \varpi}}$
is the quotient map.
\end{lem}

\begin{proof}
Fix $\ep>0$ and a finite subset ${\cal F}\subset A^{\bf 1}.$
We may write $D_0=M_{r(1)}\oplus M_{r(2)}\oplus \cdots \oplus M_{r(d)}.$
Set $R=d^2r(1)^2\cdots r(d)^2$ and $\eta=\ep/4R.$
Denote by $\{e_{s, i,j}: 1\le i, j\le r(s)\}$ a system of matrix units for $M_{r(s)},$ $1\le s\le d.$
Put  $C=l^\infty(A)/I_{_{\Tw, \varpi}}.$ 
 Since $C_0((0,1])\otimes D_0$ is projective, 
 choose
 $\wtd\psi=\{\psi_n\}: C_0((0,1])\otimes D_0\to l^\infty(A),$ a \hm\, 
  lifting of 
 $\psi$ (if we begin with an order zero \cpc, then we choose an order zero \cpc\, lifting 
 $\wtd \psi;$  see \cite[Proposition 1.2.1]{WcovII}).
 Choose an integer $k_0\in \N$ such that
 \beq
 \|x-\psi_{k_0}(b_x)\|_{_{2, K}}<\ep/2\rforal x\in {\cal F}\,\,\andeqn {\rm some}\,\, b_x\in (C_0((0,1])\otimes D_0)^{\bf 1}
\eneq
(or  $b_x\in D_0^{\bf 1}$).
 There are relatively open subsets $O\subset {\bar O}\subset O_1\subset \partial_e(T(A))$
 with $K\subset O$ such that 
 \beq\label{LnK-5}
 \|x-\psi_{k_0}(b_x)\|_{_{2, O_1}}<\ep/2\rforal x\in {\cal F}\,\,\andeqn {\rm some}\,\, b_x\in 
 C_0((0,1])\otimes D_0^{\bf 1}
 \eneq
 (or $b_x\in D_0^{\bf 1}$).  Put ${\cal G}_{\cal F}=\{b_x: x\in {\cal F}\}.$
 To simplify notation, \wilog, we may assume that  ${\cal G}_{\cal F}=\{g\otimes d: g\in {\cal G}_c, d\in D_0^{\bf 1}\},$
 where ${\cal G}_c\subset C((0,1])$ is a finite subset.
 We may further assume, \wilog, that $g f_\eta=g$ for all $g\in {\cal G}_c$ (see \ref{Nfg}  for $f_\eta$). 

Put $F=\bar O$  and $C_F=l^\infty(A)/I_{_{F, \varpi}}.$
Define $\psi': C_0((0,1])\otimes D_0\to C_F$ by $\psi'(b)=\Pi_{_{F, \varpi}}\circ \{\iota(\psi_{k_0}(b))\}$ for all $b\in C_0((0,1])\otimes D_0.$
Then $\psi'$ is a \hm\, (or an order zero \cpc\, from $D_0$ to $C_F$ if $\wtd \psi: D_0\to l^\infty(A)$ 
is an order zero \cpc
).
Let $a_{s, i,i}=\psi'(\jmath\otimes e_{s,i,i})$ 
(or $a_{s, i,i}=\psi'(e_{s,i,i,})$) ($1\le i\le r(s)$ and $1\le s\le d$). 
Since $A$ has T-tracial approximate oscillation zero, 
by Theorem 6.4 of \cite{FLosc}, $\overline{a_{s,1,1}C_Fa_{s,1,1}}$ has real rank zero.
Choose 
$b_{s,1,1}=f_{\eta/2}(a_{s,1,1})$ and $c_{s,1,1}=f_{\eta/4}(a_{s,1,1}),$  $1\le s\le d.$
There are mutually orthogonal projections $p_{s,1,1}, p_{s,1,2},...,p_{s,1,N_s}\in \overline{ b_{s,1,1}C_Fb_{s,1,1}}$
and $\lambda_{s,1}, \lambda_{s,2},...,\lambda_{s,N}\in (0,1]$ such that
\beq\label{LnK-6}
\|g(a_{s,1,1})-\sum_{k=1}^{N_s}g(\lambda_i)p_{s,1, i}\|<\ep/4R\andeqn g\in {\cal G}_c
\eneq
(in case that $\wtd \psi: D_0\to l^\infty(A)$ is an order zero \cpc, we choose $g=\jmath$ ---see \ref{Defj} for the function $\jmath$).
We note that 
\beq
p_{s,1,i} c_{s, 1,1}=p_{s,1,i}=c_{s, 1,1}p_{s,1,i}\andeqn
p_{s,1,i} c_{s, 1,1}^2=p_{s,1,i}=c_{s, 1,1}^2p_{s,1,i}
\eneq
$1\le i\le N_s,\,\,1\le s\le d.$  It follows that 
$\{\psi(f_{\eta/4}(\jmath)\otimes e_{s, i,1}) p_{s, 1,i} \psi(f_{\eta/4}(\jmath)\otimes e_{s, 1,j})\}_{1\le i,j\le r(s)}$
forms a system of  matrix unit for $M_{r(s)},$ $1\le s\le d.$

Let $D_{1,s}$ be the commutative  finite dimensional \SCA\, generated by $p_{s,1,1},...,p_{s,1, N_s}.$ 
Let $D_1=\bigoplus_{s=1}^d \C^{N_s} \otimes M_{r(s)}\cong \bigoplus_{s=1}^d D_{1,s}\otimes M_{r(s)}$ and
$q_{s,k}=p_{s,1,k}\otimes 1_{r(s)},$ $1\le s\le d.$
Define  a \hm\,
$\phi_0: D_1\to C_F$ by 
\beq\nonumber
&&\hspace{-0.4in}\phi_0((\af_{s,1}, \af_{s,2},...,\af_{s,N_s})\otimes e_{s,i,j})=
\psi(f_{\ep/4}(\jmath)\otimes e_{s, i,1})\cdot (\sum_{k=1}^{N_s} \af_{s,k}p_{s,1,k})\cdot \psi(f_{\ep/4}(\jmath\otimes e_{s,1,j}))
\eneq
for any $(\af_{s,1}, \af_{s,2},...,\af_{s,N_s})\in  \C^{N_s}$ and $1\le i,j\le n.$
By \eqref{LnK-5} and \eqref{LnK-6}, for any $x\in {\cal F},$ there exists $y_x\in D_1^{\bf 1}$ such that
\beq
\|\Pi_{_{F,\varpi}}(\iota(x))-\phi_0(y_x)\|_{_{2, F_\varpi}}<\ep.
\eneq
If $l^\infty(A)/I_{_{F, \varpi}}$ is unital, 
put $D=\C\oplus D_1$ and define $\bar \phi: D\to C_F$
by 
\beq
\bar \phi(\lambda, d)=\lambda\cdot  (1-\phi_0(1_D))\oplus \phi_0(d)
\eneq
for all $(\lambda, d)\in \C\oplus D_1=D.$  Then $\bar \phi$ is unital and 
\beq
\|\Pi_{_{F, \varpi}}(\iota(x))-\bar \phi(0, y_x)\|_{_{2, F_\varpi}}<\ep.
\eneq
If $l^\infty(A)/I_{_{F, \varpi}}$ is not unital, put $\bar \phi=\phi_0.$
Since $C$ has real rank zero, by Theorem 6.4 of \cite{FLosc}, 
applying Elliott's lifting lemma (see Lemma \ref{Lelliott}), we obtain a \hm\, $\phi: D\to C$ such that 
$\pi_F\circ \phi=\bar \phi.$ The lemma follows.
\end{proof}

Note that, if $T(A)$ is compact, then $l^\infty(A)/I_{_{F, \varpi}}$ is unital (see, for example, 
Proposition \ref{Pcompact}).

\begin{cor}\label{LnK}
Let $A$ be a separable amenable algebraically simple \CA\, with T-tracial approximate oscillation zero
and let $\tau\in \partial_e(T(A)).$ 
Then, for any $\ep>0$ and any finite subset ${\cal F}\subset A^{\bf 1},$ there  are
a relatively open subset $O\subset \partial_e(T(A))$ with $\tau\in O,$ 
a finite dimensional \CA\, $D$ and a \hm\, 
$\phi: D\to l^\infty(A)/I_{_{\Tw, \varpi}}$  such that (with $F=\bar O$)
\beq
\|\Pi_{_{F, \varpi}}(\iota(x))-\pi_F\circ \phi(y_x)\|_{_{2, F_\varpi}}<\ep\tforal x\in {\cal F}\,\,and\,\,some\,\, y_x\in D^{\bf 1},
\eneq
where $\pi_F: l^\infty(A)/I_{_{\Tw, \varpi}}\to l^\infty(A)/I_{_{F, \varpi}}$
is the quotient map.
If $l^\infty(A)/I_{_{F, \varpi}}$ is unital, then one may require that $\pi_F\circ \phi$ is a unital \hm.
\end{cor}

\begin{proof}
By Lemma \ref{Ptesrk1}, there is an integer $n\ge 2,$ and an 
open neighborhood $O$ of $\tau$ of $T(A)$ and 
an order zero \cpc\, $\psi': M_n\to A$ 
such
that
\beq\label{CnK-5}
\|x-\psi'(y_x)\|_{_{2, \bar O}}<\ep/4\rforal x\in {\cal F}\,\,{\text{and\,\, some}}\,\, y_x\in M_n^{\bf 1}.
\eneq
Put  $C=l^\infty(A)/I_{_{\Tw, \varpi}},$ 
 $F=\bar O$  and $C_F=l^\infty(A)/I_{_{F, \varpi}}.$
Define $\psi: M_n\to C_F$ by $\psi(b)=\Pi_{_{F, \varpi}}\circ \{\psi'(b)\}$ for all $b\in M_n.$
Then $\psi$ is an order zero \cpc. 

The lemma then immediately follows from Lemma \ref{LnFO}.
\end{proof}

\begin{lem}\label{Lbb}
Let $A$ be a separable  amenable algebraically simple \CA\, with nonempty  compact $T(A)$ and 
$F\subset \partial_e(T(A))$ be a compact subset.
Suppose that 
$a_1, a_2,...,a_l\in l^\infty(A)/I_{_{F, \varpi}}$  and
%
%
$f_1, f_2,...,f_m\in C(F)_+^{\bf 1}$ are mutually 
orthogonal functions.
Then there are $b^{(1)}=\{b^{(1)}_n\}, b^{(2)}=\{b^{(2)}_n\},..., b^{(m)}=\{b_n^{(m)}\}\in 
\pi_\infty^{-1}(A')_+^{\bf 1}$ 
such that  

(1)  $\Pi_{_{F, \varpi}}(b^{(1)}), \Pi_{_{F, \varpi}}(b^{(2)}),..., \Pi_{_{F, \varpi}}(b^{(m)})$ 
are mutually orthogonal, 

(2) $\lim_{n\to \infty}\sup\{|\tau(b_n^{(j)})-\tau(f_n)|: \tau\in F\}=0,\,\,\, j=1,2,...,m,$

(3) 
$\Pi_{_{F, \varpi}}(b^{(j)})a_i=a_i\Pi_{_{F, \varpi}}(b^{(j)}),$ $1\le i\le l,$ $1\le j\le m,$  and 

(4) $\lim_{n\to\varpi}\sup\{|\tau(b_n^{(j)}a_{i,n})-\tau(f_j)\tau(a_{i,n})|: \tau\in F\}=0,$
$1\le j\le m,$
where  $\Pi_{_{F, \varpi}}(\{a_{i,n}\})=a_i$ and $a_{i,n}\in A,$  $1\le i\le l.$
\end{lem}

\begin{proof}
First one observes that 
Lemma 3.4   in \cite{TWW-2} holds for  the case that $A$ is not necessarily unital 
but $T(A)$ is compact (see also 
Lemma 4.1 of \cite{S-2}---replacing $\partial_e(T(A))$ by $F$).

We will prove the lemma by induction on $m.$ 
By  \cite[Theorem II.3.12, (ii)]{Alf},
there is an affine function $g_1\in \Aff(T(A))_+^{\bf 1}$ such that
${g_1}|_F=f_1.$
It follows from Lemma \ref{Tsato-1} that there exists $\{b_n^{(1)}\}\in \pi_\infty^{-1}(A')_+^{\bf 1}$ such that
\beq\label{Lext-e-2}
\lim_{n\to\infty}\sup\{|\tau(b_n^{(1)})-\tau(f_1)|:\tau\in F\}=0.
\eneq
%
Note that, for each fixed $k,$ 
\beq
\lim_{n\to\infty}\|b_n^{(1)} a_{i,k}-a_{i,k} b_n^{(1)}\|=0.
\eneq
By applying Proposition 3.1 of \cite{CETW} (see also \cite[Lemma 4.2]{S-2}), 
passing to a subsequence of $\{b_n\},$ we may assume 
that, for all $y\in D^{\bf 1},$ 
\beq
&&
{\bar b}a_i=a_i {\bar b},\,\,1\le i\le l,\\\label{Lext-e-4}
&&\lim_{n\to\infty} \sup\{|\tau(b_n^{(1)}a_{i,n})-\tau(b_n^{(1)})\tau(a_{i,n})|: \tau\in F\}=0,
\eneq
where  ${\bar b^{(1)}}=\Pi_{_{F, \varpi}}(\{b_n^{(1)}\}).$
This proves the case that $m=1.$

Suppose that the lemma holds for $m.$ We assume that 
$f_1, f_2,...,f_m, f_{m+1}$ are mutually orthogonal. 
By the inductive assumption,  there are $b^{(1)'}=\{b^{(1)'}_n\}, b^{(2)'}=\{b^{(2)'}_n\},..., b^{(m)'}=\{b_n^{(m)'}\}\in 
\pi_\infty^{-1}(A')$ 
such that  (1), (2), (3) and (4) hold. 
It follows from Lemma \ref{Tsato-1} that there exists $\{b_n^{(m+1)'}\}\in \pi_\infty^{-1}(A')_+^{\bf 1}$ such that
\beq\label{Lbb-10}
\lim_{n\to\infty}\sup\{|\tau(b_n^{(m+1)'})-f_{m+1}(\tau)|: \tau\in T(A)\}=0.
\eneq
Since $\{b_n^{(m+1)'}\}\in \pi_\infty^{-1}(A')_+^{\bf 1},$ by passing to a subsequence, 
we may assume that 
\beq
\|b_n^{(j)'}b_n^{(m+1)'}-b_n^{(m+1)'}b_n^{(j)'}\|<1/n^2,\,\,\, 1\le j\le m.
\eneq
By the second part of \cite[Proposition 3.1]{CETW}, we may further assume (by passing to another subsequence 
if necessary)  that
\beq
\lim_{n\to\infty} \sup\{|\tau(b_n^{(j)'}b_n^{(m+1)'})-\tau(b_n^{(j)})\tau(b_n^{(m+1)'})|: \tau\in F\}=0.
\eneq
By the inductive assumption and by \eqref{Lbb-10},
we have 
\beq
\lim_{n\to \varpi}\sup\{|\tau(b_n^{(j)'}b_n^{(m+1)'})-f_j(\tau)f_{m+1}(\tau)|: \tau\in F\}=0.
\eneq
In other words, since $f_1, f_2,...,f_{m+1}$ are mutually orthogonal,
\beq\label{Lbb-12}
\lim_{n\to \varpi}\sup\{|\tau(b_n^{(j)'}b_n^{(m+1)'})|: \tau\in F\}=0.
\eneq
By \eqref{Lbb-12}, the inductive assumption,  
and (the non-unital version of) Lemma 3.4   in \cite{TWW-2},
we obtain  $b^{(1)}=\{b_n^{(1)}\},$ $b^{(2)}=\{b_n^{(2)}\},...,b^{(m+1)}=\{b_n^{(m+1)}\}$ in $\pi_\infty^{-1}(A')_+^{\bf 1}$ 
such that (1) holds for $m+1$ and 
\beq
\lim_{n\to\infty} \sup\{|\tau(b_n^{(i)})-\tau(f_i)|: \tau\in F\}=0,\,\, 1\le i\le m+1.
\eneq
Thus (2) holds for $m+1.$ 
Since $b^{(m+1)}\in \pi_\infty^{-1}(A')_+^{\bf 1},$ 
we may also assume 
that 
\beq\label{Lbb-11}
a\Pi_{_{F, \varpi}}(b^{(m+1)})=\Pi_{_{F, \varpi}}(b^{(m+1)})a
\eneq
for $a\in \{a_1,a_2,...,a_l\}.$
Hence (3) also holds.  Moreover, 
applying the second part of \cite[Proposition 3.1]{CETW}, we may further assume that (4) holds. 
%
%
%
This completes the induction and the lemma follows.
%
%
\end{proof}

The following is well known. We retain here for convenience.

\begin{lem}\label{LLtrb}
Let $A$ be a \CA\, with a (finite) trace $\tau.$ Suppose that $b\in A_+^{\bf 1}$ and 
$\tau(b)=0.$ 
Then, for any $c\in \Her(b)_+,$ 
$\tau(c)=0.$
\end{lem}

\begin{proof}
Let $d=bab$ for some $a\in A_+.$ 
Then
\beq
\tau(d)^2=\tau(b^2a)^2\le \tau(b^4)\tau(a^*a)\le \tau(b)\tau(a^*a)=0.
\eneq
It follows that $\tau(d)=0.$ Hence $\tau(c)=0$ for all $c\in \Her(b)_+.$
\end{proof}

\begin{lem}\label{Lext}
Let $A$ be a separable  amenable algebraically simple \CA\, which has T-tracial approximate oscillation zero 
and has a  nonempty  compact $T(A).$
Let $X_1\subset  O_1\subset X_2\subset
X \subset Y\subset  \partial_e(T(A))$  be   subsets
such that  $X_1, X_2,  X$ and $Y$ are compact subsets and $O_1$  is
relatively open in $X_2.$ 
Let $\ep>0$ and ${\cal F}\subset A^{\bf 1}$  be a finite subset.
Suppose that $D_0$ is a finite-dimensional \CA\, 
and  $\phi: D_0\to l^\infty(A)/I_{_{X,\varpi}}$
is a unital \hm\,
such that, for each $x\in {\cal F},$ 
 \beq\label{Lext-e-1}
 \inf\{\|\Pi_{_{X, \varpi}}(\iota(x))-\phi(y)\|_{_{2, X_\varpi}}: y\in D_0^{\bf 1}\}<\ep/4.
 \eneq

Then there is a \ finite dimensional \CA\, $D_1$ 
and a \hm\, $\psi: D_1\to l^\infty(A)/I_{_{T(A), \varpi}}$ 
such that (where $\Pi_i:  l^\infty(A)/I_{_{T(A), \varpi}}\to l^\infty(A)/I_{_{X_i, \varpi}}$ is the quotient map)

(1) $\inf\{\|\Pi_2\circ \psi(1_{D_1})\Pi_{_{X_2, \varpi}}(\iota(x))-\Pi_2\circ \psi(y)\|_{_{2,{X_2}_\varpi}}: y\in D_1^{\bf 1}\}<\ep$
for all $x\in {\cal F},$

(2) $\psi(1_{D_1})\in \Pi_\varpi(I_{_{Y\setminus X_2, \varpi}}),$

(3) $\Pi_1\circ \psi: D_1\to l^\infty(A)/I_{_{X_1, \varpi}}$ is unital,  

(4) $\|[\Pi_{_{Y,\varpi}}(\iota(x)),\, \pi_Y\circ \psi(1_{D_1})]\|_{_{Y, \varpi}}<\ep$ for all $x\in {\cal F},$
(where $\pi_Y: l^\infty(A)/I_{_{\varpi}}\to l^\infty(A)/I_{_{Y, \varpi}}$ is the quotient map.
%
%
%
\end{lem}

\begin{proof}
Let $\{\phi_n\}: D_0\to l^\infty(A)$ be an order zero \cpc\, lifting of $\phi$ and 
$\phi_Y=\Pi_{_{Y, \varpi}}\circ \{\phi_n\}: D_0\to l^\infty(A)/I_{_{Y, \varpi}}.$ 
Suppose that $f\in C(Y)_+^{\bf 1}$ such that 
$f|_{X_1}=1$ and $f|_{Y\setminus X_2}=0.$ 
%
It follows from Lemma \ref{Lbb}
  that there is a sequence 
$\{b_n\}\in (\pi_\infty^{-1}(A'))_+^{\bf 1}$ such that (considering a system of matrix units of $D_0$)
\beq\label{Lext-e-2}
&&\lim_{n\to\infty}\sup\{|\tau(b_n)-\tau(f)|:\tau\in Y\}=0,\\
&&b_Y\phi_Y(y)=\phi_Y(y)b_Y\andeqn\\\label{Lext-e-4}
&&\lim_{n\to\infty} \sup\{|\tau(b_n\phi_n(y))-\tau(b_n)\tau(\phi_n(y))|: \tau\in Y\}=0\rforal y\in D_0,
\eneq
where  $b_Y=\Pi_{_{Y, \varpi}}(\{b_n\}).$
Set $b=\Pi_{_{\varpi}}(\{b_n\})$ and $b_X=\Pi_{_{X, \varpi}}(\{b_n\}).$ 
Define $\phi_Y^b: D_0\to l^\infty(A)/I_{_{Y, \varpi}}$ by
$\Pi_{_{Y, \varpi}}(\{b_n^{1/2}\phi_n(y)b_n^{1/2}\})$ for all $y\in D_0.$
Put $B_0=\overline{\{b_n^{1/2}\}l^\infty(A) \{ b_n^{1/2}\}},$
$J_{_{Z, \varpi}}=I_{_{Z, \varpi}}\cap \Pi_{_{Z, \varpi}}(B_0),$ 
where $Z=X_1, X_2, X, Y,$ or $T(A).$  
Then $\phi_Y^b$ is an order zero \cpc\, from $D_0$ to $B_0/J_{_{Y, \varpi}}$ 
and $\phi_Y^b(1_{D_0})=b_Y^{1/2} \phi_Y(1_{D_0})b_Y^{1/2}.$
Also define $\phi^b: D_0\to B_0/I_{_{X, \varpi}}$ 
by $\phi^b(x)=b_X^{1/2} \phi(x)b_X^{1/2}$ for all $x\in D_0.$ Then
$\phi^b$ is also an order zero \cpc.
Note that $B_0/J_{_{Y, \varpi}}=\overline{b_Y^{1/2}(l^\infty(A)/I_{_{Y, \varpi}}) b_Y^{1/2}}.$

Write $D_0=M_{r_1}\oplus M_{r_2}\oplus \cdots M_{r_l}.$ 
Choose $R=l^2 r_1^2r_2^2\cdots r_l^2$
and write $\{e_{s, i,j}\}_{1\le i, j\le r_s}$ for a fixed system of matrix units for $M_{r_s},$ $1\le s\le l.$

Put $Q_X=B_0/J_{_{X, \varpi}}.$ Consider $B_{s,i}=\overline{\phi^b(e_{s,i,i})Q_X\phi^b(e_{s,i,i})}$
and ${\bar B_{s,i}}=\phi(e_{s,i,i})Q_X\phi(e_{s,i,i}),$ 
$1\le i\le r_s$ and $1\le s\le l.$ Put $d=\phi^b(1_{D_0}).$
Then, by Proposition 8.3 of \cite{FLosc}, for example, $\overline{dQ_Xd}\cong \bigoplus_{s=1}^l M_{r_s}(B_{s,1}).$
Also $\phi(1_{D_0})Q_X\phi(1_{D_0})\cong \bigoplus_{s=1}^l M_{r_s}(\phi(e_{s,1,1})Q_X\phi(e_{s,1,1}))$
(recall that $\phi$ is a \hm).

Let $\dt\in (0,\ep/4).$ 
Fix $\dt_0\in (0,\dt/2)$ for now. 
By Theorem 6.4 of \cite{FLosc}, $l^\infty(A)/I_{_{T(A), \N}}$ has real rank zero. Hence its quotient 
$l^\infty(A_1)/I_{_{T(A),\varpi}}$ also  has real rank zero.
Since 
$Q_X$  is a hereditary \SCA\, of a quotient of $l^\infty(A_1)/I_{_{T(A),\varpi}},$ 
it also has real rank zero. So does $B_{s,1}.$ 
Hence there is a commutative finite dimensional \SCA\, $C_s\subset B_{s,1}$ with 
$p_{s,1}=1_{C_s}$ 
such that, for some $y_s\in C_s^{\bf 1},$ 
\beq\label{Tst-2-5}
&&\|\phi^b(e_{s,1,1})-y_s\|<\dt_0/4R \andeqn\\\label{Tst-2-5+}
&&\|p_{s, 1}\phi^b(e_{s,1,1})-\phi^b(e_{s,1,1})\|<\dt_0/4R,\,\, 1\le i\le r_s\andeqn 1\le s\le l.
\eneq
Note that we may assume that 
\beq\label{Lext-e-10}
p_{s,1}\le f_\eta(\phi^b(e_{s,1,1}))
\eneq
for some $\eta\in (0, \dt_0/16R).$ 
We also note that $\phi^b(e_{s, i,i})=b_X\phi(e_{s,i,i})=\phi(e_{s,i,i})b_X\le \phi(e_{s,i,i}).$ 
Suppose that $C_s$ is generated by mutually orthogonal projections 
$p_{s,1,1}, p_{s,1,2},...,p_{s,1,m(s)}$ with $p_{s,1}=\sum_{k=1}^{m(s)} p_{s, i,k}.$

Put $v_{s,i,j,k}=\phi(e_{s,i,1})p_{s,1,k} \phi(e_{s,1,j}),$ $1\le k\le m(s),$  $1\le i, j\le r_s,$ $1\le s\le l.$
Note that $p_{s,1,k}$ is a sub-projection of $\phi(e_{s,1,1}).$  Then $\{v_{s,i,j,k}\}_{1\le i,j\le r_s}$ forms a system of matrix units for $M_{r_s}.$
Put 
\beq
C=\bigoplus_{s=1}^l M_{r_s}(\C\cdot p_{s,1,1}+\C \cdot p_{s,1, 2}+\cdots \C p_{s, 1, m(s)}).
\eneq
 Then $C\cong \bigoplus_{s=1}^l M_{r_s}(C_s).$ Put $D_1= \bigoplus_{s=1}^l M_{r_s}(C_s).$
Let $h: D_1\to C$ be the isomorphism. Let $p=h(1_{D_1})\in B_0/I_{_{X, \varpi}}$ (see 
\eqref{Lext-e-10}). 
Then $p$ is a projection. 
By  \eqref{Tst-2-5+}, we have that
\beq\label{Lext-6}
\|p \phi^b(1_{D_0})-\phi^b(1_{D_0})\|<\dt_0.
\eneq
%
It follows from \eqref{Lext-e-1} that, for each $x\in {\cal F},$
\beq\label{Lext-6+n10}
\inf\{\|p(\Pi_{_{X, \varpi}}(\iota(x))-\phi(y))\|_{_{2, X_\varpi}}: y\in D_0^{\bf 1}\}<\ep/4.
\eneq
Also
\beq
p_{s,i}:=\sum_{k=1}^{m(s)} v_{s, i,i,k} =\sum_{k=1}^{m(s)}\phi(e_{s,i,1})p_{s,1,k} \phi(e_{s,1,i})
\le \phi(e_{s,i.i})
\eneq
which is a projection and 
$p=\sum_{s=1}^l (\sum_{i=1}^{r_s} p_{s,i}).$ 
One then computes that
\beq
\phi(e_{s, i,j})p&=&\sum_{s=1}^l \phi(e_{s,i,j})p_{s,j}=\sum_{s=1}^l\sum_{k=1}^{m(s)} \phi(e_{s,i,j}e_{s,j,1})p_{s,1,k}\phi(e_{s,1,j})\\\nonumber
&=&\sum_{s=1}^l\sum_{k=1}^{m(s)} \phi(e_{s, i,1})p_{s,1,k}\phi(e_{s,1,j})=\sum_{s=1}^l\sum_{k=1}^{m(s)} \phi(e_{s, i,1})p_{s,1,k}\phi(e_{s,1,i}e_{s,i,j})\\
&=&(\sum_{s=1}^lp_{s,i} )\phi(e_{s, i,j})=\sum_{s=1}^l (\sum_{i=1}^{r_s} p_{s,i})\phi(e_{s, i,j})=p\phi(e_{s, i,j}).
\eneq
In other words $p$ 
commutes with 
$\phi(e_{s, i,j})$
for all $1\le i,j\le r_s,$ $1\le s\le l.$  Moreover $p\phi(e_{s, i,j})=\sum_{s=1}^l \sum_{k=1}^{m(s)}v_{s, i,j,k}\in C.$
Therefore (see also \eqref{Lext-6+n10})
\beq\label{Lext-e-14}
\inf\{\|p\Pi_{_{X, \varpi}}(\iota(x))-h(y)\|_{_{2, X_\varpi}}: y\in D_1\}<\ep/4.
\eneq
We also have  (by \eqref{Lext-e-1} and the fact  $p\phi(y)=\phi(y)p$ for all $y\in D_0$)
\beq\label{Lext-e-15}
\|[\Pi_{_{X, \varpi}}(\iota(x)), p]\|_{_{2, X_\varpi}}<\ep/4.
\eneq

Since $J_{_{X,\varpi}}/J_{_{T(A), \varpi}}$ has real rank zero,   by the Elliott  lifting Lemma (see  Lemma \ref{Lelliott}),
we obtain a \hm\, $\psi: D_1\to B_0/J_{_{T(A), \varpi}}\subset l^\infty(A)/I_{_{T(A), \varpi}}$ such 
that $\pi_X\circ \psi=h,$  where $\pi_X: l^\infty(A)/I_{_{T(A), \varpi}}\to l^\infty(A)/I_{_{X, \varpi}}$
is the quotient map. 

We will verify that $\psi$ meets the requirements of the lemma. 

First we note that, by \eqref{Lext-e-14},   (1) holds.

Recall that $B_0/J_{_{Y, \varpi}}=\overline{b_Y^{1/2}(l^\infty(A)/I_{_{Y, \varpi}})b_Y^{1/2}}.$
Note that, for any $\tau\in Y\setminus X_2,$ $\tau(f)=0.$
It follows from \eqref{Lext-e-2} that
\beq
&&\hspace{-0.8in}\lim_{n\to \varpi}\sup\{\tau(b_n):\tau\in Y\setminus X_2\}
\le  \lim_{n\to \varpi}\sup\{|\tau(b_n)-f(\tau)|: \tau\in Y\setminus X_2\}\\
&&\hspace{1.2in}+ \lim_{n\to \varpi}\sup\{|f(\tau)|: \tau\in Y\setminus X_2\}=0.
\eneq
Since $b_n^*b_n\le b_n,$ this implies that $\{b_n\}\in I_{_{Y\setminus X_2, \varpi}}$
and $b\in \Pi_{_{\varpi}}(I_{_{Y\setminus X_2, \varpi}}).$
Recall that\\ $B_0/I_{_{T(A),\varpi}}=\overline{b(l^\infty(A)/I_{_{T(A),\varpi}})b}.$
Therefore  $B_0/I_{_{T(A),\varpi}}\subset \Pi_{_{\varpi}}(I_{_{Y\setminus X_2, \varpi}}).$ Hence\\ $\psi(1_{D_1})\in  \Pi_{_{\varpi}}(I_{_{Y\setminus X_2, \varpi}})$ and (2) holds.

It also follows that $\psi(1_{D_1})\Pi_{_{\varpi}}(\iota(x))\in \Pi_\varpi(I_{_{Y\setminus X_2, \varpi}}).$
In particular,
\beq\label{Lext-e-17}
\pi_{Y\setminus X_2}([\Pi_{_{\varpi}}(\iota(x)),\, \psi(1_{D_1})])=0,
\eneq
where $\pi_{Y\setminus X_2}: l^\infty(A)/I_{_{\varpi}}\to l^\infty(A)/I_{_{Y\setminus X_2}}$ is the quotient map.
Since $Y=(Y\setminus X_2)\cup X,$  
 by \eqref{Lext-e-15}, 
\eqref{Lext-e-17}  and Proposition \ref{Punion}, we have that
\beq
\|[\Pi_{_{Y, \varpi}}(\iota(x)),\, \pi_Y\circ \psi(1_{D_1})]\|_{_{2, Y_
\varpi}}<\ep\rforal x\in {\cal F}.
\eneq
So (4) holds.

Recall that $f(\tau)=1$ for all $\tau\in X_1.$ By \eqref{Lext-e-2}, \eqref{Lext-e-4},
\beq
\lim_{n\to\varpi}\sup\{|\tau(b_n^{1/2}\phi_n(1_{D_0})b_n^{1/2})-\tau(\phi_n(1_{D_0}))|: \tau\in X_1\}=0.
\eneq
Since $\phi$ is unital, this implies that
\beq
\lim_{n\to\varpi}\sup\{|\tau(b_n^{1/2}\phi_n(1_{D_0})b_n^{1/2})-1|: \tau\in X_1\}=0.
\eneq
In other words, $\wtd\pi_{X_1}\circ \phi^b(1_{D_0})=1,$ where $\wtd\pi_{X_1}: 
l^\infty(A)/I_{_{X, \varpi}}\to l^\infty(A)/I_{_{X_1, \varpi}}$ is the quotient map. 
It then follows from \eqref{Lext-6} that $\wtd\pi_{X_1}(p)$ is invertible.
But $p=h(1_{D_1})$ is a projection. Therefore $\wtd\pi_{X_1}(p)=1.$
In other words, $\Pi_1(\psi(1_{D_1}))=1.$ This shows that (3) holds. 

\end{proof}

\begin{prop}
\label{Pst-1}
Let $A$  be a separable algebraically simple amenable \CA\, with 
T-tracial approximate oscillation zero and non-empty compact $T(A).$
Let $F\subset \partial_e(T(A))$ be a compact subset with
transfinite dimension ${\rm trind}(F)=c<\Omega$ (see Definition \ref{Dtransfinite}).
Then, for any $\ep\in (0, 1/2),$ and any finite subset ${\cal F}\subset A^{\bf 1},$ 
there exists a finite dimensional \CA\, $D$  
and 
a unital \hm\, 
$\phi: D\to l^\infty(A)/I_{_{F, \varpi}}$ 
such that 
\beq
 \inf\{\|\Pi_{_{F, \varpi}}(\iota(x))-\phi(y)\|_{_{2, F_\varpi}}: y\in D^{\bf 1}\}<\ep\tforal x\in {\cal F}.
 \eneq
\end{prop}

\begin{proof}
Since $T(A)$ is assumed to be compact, $l^\infty(A)/I_{_{\varpi}}$ is unital (see \ref{Pcompact}) and hence,
$l^\infty(A)/I_{_{S, \varpi}}$ are all unital for all $S\subset T(A).$

Let $d={\rm trind}(F)$ be the transfinite dimension of $F.$ Note that $F$ is compact and metrizable.
Keep in mind that $\emptyset$ has dimension $-1.$ 
The proof is somewhat similar to that of Proposition 5.1 of \cite{S-2}, using, 
among other things, a standard (transfinite) induction for the boundaries on the dimension $d.$ 

Let $c$ be an ordinal number  such that $0\le c\le d.$ 
Let us assume that, for any compact subset  $F_0\subset \partial_e(T(A))$
with ${\rm trind}(F_0)<c,$ $\ep>0,$ and any finite subset ${\cal F}\subset A^{\bf 1},$ 
 there exist a  finite dimensional \CA\, $D_{0}'$ 
 and 
 a unital \hm\, 
$\phi_0: D_{0}'\to l^\infty(A)/I_{_{F_0, \varpi}}$ 
such that

(1) $\|\Pi_{_{F_0, \varpi}}(\iota(x))-\phi_0(y_x)\|_{_{F_0, \varpi}}<\ep\rforal x\in {\cal F}$ and some $y_x\in  (D_0')^{\bf 1}.$

Note that, if $c=0,$ $F_0=\emptyset.$ So the above holds automatically.

Now let $F\subset \partial_e(T(A))$ be a compact subset with ${\rm trind}(F)=c.$

By Corollary \ref{LnK}, 
%
for each $\tau\in F,$ 
there is   a finite dimensional \CA\, $D_\tau$ 
and a relatively open subset $U_{\tau}\subset F$  and 
a \hm\, 
$\wtd\phi_\tau: D_\tau\to l^\infty(A)/I_{_{F, \varpi}}$ such that 
$\bar \phi_\tau=\pi_{\bar U_\tau}\circ \wtd\phi_\tau: D_{\tau}\to l^\infty(A)/I_{_{{\bar U_\tau}, \varpi}}$
is a unital \hm\, and 


($i^0$) $\|\Pi_{_{\bar U_\tau, \varpi}}(\iota(x))-\bar \phi_\tau(y_{x, \tau})\|_{_{{\bar U}_\tau, \varpi}}<\ep/4^{3}\rforal x\in {\cal F}$ and for some $y_{x, \tau}\in  D_\tau^{\bf 1},$ 
where $\pi_{\bar U_\tau}: l^\infty(A)/I_{_{F, \varpi}}\to l^\infty(A)/I_{_{{\bar U_\tau}, \varpi}}$ is the quotient map
(and ${\bar U}=\overline{U\cap F}$).

Choose an order zero \cpc\, $\phi_{\tau}: D_{\tau}\to  A$ 
\st\, 

($i_0$) $\|x-\phi_\tau(y_{x, \tau})\|_{_{2, {\bar U}_\tau}}<\ep/4^{3}\rforal x\in {\cal F}$ and for some $y_{x, \tau}\in  D_\tau^{\bf 1},$  and

($i_1$) $\|1-\phi_\tau(1_{D_\tau})\|_{_{2, {\bar U}_\tau}}<\ep/4^3.$



Since $F$ is compact, there are $\tau_1, \tau_2,...,\tau_N\in F$ such that
$\cup_{i=1}^N U_{\tau_i}\supset F.$
Since 
${\rm trind}(F)$ is $c,$  there exist
relatively open subsets $V_1, V_2,...,V_N$ of $F$ such that 
$\cup_{i=1}^N V_i\supset F,$ $\overline{V_i}\subset U_i\cap F$ 
 and, if $c=(c-1)+1,$ 
${\rm trind}(bd_F(V_j))\le c-1,$ 
or, if $c$ is a limit ordinal, ${\rm trind}(bd_F(V_j))<c,$
where $bd_F(V_j)=(\overline{V_j}\setminus V_j)\cap F,$ 
$j=1,2,...,N.$  Set $F_0=\cup_{j=1}^N bd_F(V_j).$ Then ${\rm trind}(F_0)<c.$

Let $D_V=\bigoplus_{i=1}^N D_{\tau_i}.$
Choose $\dt_0>0$ (in place of $\dt$) which is given by Lemma \ref{Lsemip} for $\ep/16$  (in place of $\ep$)
associated with $D_V$ (in place of $D$).   We may choose $\dt_0$ such that 
$\dt_0<\ep/16.$

Choose, for each $s\in \{1,2,...,N\},$ a finite subset  ${\cal G}_s\subset  D_{\tau_s}^{\bf 1}$ which 
is $\dt_0/32N$-dense in $D_{\tau_s}^{\bf 1}.$
Let ${\cal G}={\cal F}\cup (\cup_{s=1}^N\{\phi_{\tau_s}(g): g\in {\cal G}_s\})$ 
(keep in mind that ${\cal G}$ is significantly larger than ${\cal F}$). 

By the transfinite inductive assumption,  there exist a finite dimensional \CA\, $D_0$
 and 
a unital \hm\, $\phi_0': D_0\to l^\infty(A)/I_{_{F_0, \varpi}}$ such that 


(i') $\inf\{\|\Pi_{_{F_0, \varpi}}(\iota(x))-\phi_0'(y)\|_{_{2, (F_0)_\varpi}}:y\in D_0^{\bf 1}\}<\dt_0/32 N \cdot 4^3 
\rforal x\in {\cal G}.$ 


By Lemma \ref{LnFO},
there is a relatively open subset $W_0'$ of $F$ such that $W_0'\supset F_0,$
a finite dimensional \CA\, $\bar D_0$
and a \hm\, $\wtd \phi_0: \bar D_0\to l^\infty(A)/I_{_{\varpi}}$ such that 
${\bar \phi_0}: \bar D_0\to l^\infty(A)/I_{_{{\overline W_0'}, \varpi}}$ is a unital \hm, 
where ${\bar \phi_0}=\pi_{\overline{W_0'}}\circ \wtd\phi_0$ and 
\beq\label{Pst-e-3}
&&\hspace{-0.2in}\|\Pi_{_{{\overline W_0'}, \varpi}}(\iota(x))-\bar \phi_0(y_{x,0}')\|_{_{{\overline W_0'}, \varpi}}<\dt_0/N4^{3}\rforal x\in {\cal G}\andeqn 
{\rm some}\,\, y_{x,0}'\in \bar D_0^{\bf 1}.
\eneq
%

Since $F$ is a compact metric space, there are relatively open subsets $W_{0,1}', W_{0,2}'\subset W_0'$
such that $F_0\subset W_{0,1}'\subset X_1={\overline{W_{0,1}'}}\subset  W_{0,2}'\subset  X_2=
{\overline{W_{0,2}'}}\subset W_0'\subset F.$ Note that
$F_0\subset X_1\subset X_2$ are compact subsets of $F.$ 

Choose $f_0'\in C(F)_+^{\bf 1}$ such that ${f_0}'|_{X_1}=1,$ ${f_0}'|_{F\setminus X_2}=0.$
By Lemma \ref{Lext}, there exist a finite dimensional \CA\, 
$D_1$ 
and a \hm\, $\psi_0: 
D_1\to l^\infty(A)/I_{_{F, \varpi}}$ satisfying the following
($\Pi_i: l^\infty(A)/I_{_{F, \varpi}}\to l^\infty(A)/I_{_{X_i, \varpi}}$ is the quotient map, $i=0,1$)

(1') $\|\Pi_2\circ \psi_0(1_{D_1})\Pi_{_{X_2, \varpi}}(\iota(x))-\Pi_2\circ \psi_0(y_{x,0})\|_{_{X_2, \varpi}}<\dt_0/N4^{2}$
for all $x\in {\cal G}$ and some $y_{x, 0}\in D_1^{\bf 1}.$

(2') $\psi_0(1_{D_1})\in I_{_{F\setminus X_2,\varpi}}/I_{_{F, \varpi}},$

(3') $\Pi_1\circ \psi_0: D_1\to l^\infty(A)/I_{_{X_1, \varpi}}$ is unital,  

(4') $\|[\Pi_{_{F,\varpi}}(\iota(x)),\, \psi_0(1_{D_1})]\|_{_{F, \varpi}}<\dt_0/N4^{2}$ for all $x\in {\cal G}.$  


Set $W_i=V_i\setminus (\cup_{j=1}^{i-1}\overline{V_j}\cup F_0),$ $1\le i\le N.$ 
Then $W_i\cap W_j=\emptyset$ if $i\not=j,$ ($1\le i,j \le N$) and 
$\cup_{i=1}^N W_i=\cup_{i=1}^N V_i\setminus F_0.$  Put $W_0=W_{0,1}'.$ 
Then  $\{W_0\}\cup\{W_i: 1\le i\le N\}$ is an open cover of $F.$

Let $\{ f_i: 0\le i\le N\}$ be a partition of unity subordinate to the open cover $\{W_i: 0\le i\le N\}.$
In other words, $f_i\in C(F)_+,$ ${\rm supp}(f_i)\subset W_i$ and 
$\sum_{0\le i\le N}f_i(t)=1$ for all $t\in F.$ 
Note that $f_1,f_2,...,f_N$ are mutually orthogonal ($f_0$ is not part of it). 

By Lemma \ref{Lbb} (recall that each $D_{\tau_j}$ is finitely generated), 
there are $b^{(1)}=\{b^{(1)}_n\}, b^{(2)}=\{b^{(2)}_n\},..., b^{(N)}=\{b_n^{(N)}\}\in 
{(\pi_\infty^{-1}(A'))}_+^{\bf 1}$
such that  

(1'')  $\Pi_{_{F, \varpi}}(b^{(1)}), \Pi_{_{F, \varpi}}(b^{(2)}),..., \Pi_{_{F, \varpi}}(b^{(N)})$ 
are mutually orthogonal, 

(2'') $\lim_{n\to \infty}\sup\{|\tau(b_n^{(j)})-\tau(f_j)|: \tau\in F\}=0,\,\,\, j=1,2,...,N,$

(3'') $\Pi_{_{F, \varpi}}(b^{(j)})\iota(\phi_{\tau_j}(y))=\iota(\phi_{\tau_j}(y))\Pi_{_{F, \varpi}}(b^{(j)})$ for all $y\in D_{\tau_j},$  $1\le j\le N,$  and 

(4'') $\Pi_{_{F, \varpi}}(b^{(j)})\psi_0(1_{D_1})=\psi_0(1_{D_1})\Pi_{_{F, \varpi}}(b^{(j)}),$  $1\le j\le N,$

(5'') $\lim_{n\to\varpi}\sup\{
|\tau(b_n^{(j)}\phi_{\tau_j}(y))-\tau(f_j)\tau(\phi_{\tau_j}(y))|:\tau\in F\}=0$
for all
$y\in D_{\tau_i},$ $1\le j\le N.$

%
Define $\psi_i':D_{\tau_i}\to l^\infty(A)/I_{_{F, \varpi}}$ by 
$\psi_i'(y)=\Pi_{_{F, \varpi}}((b^{(i)})^{1/2}\iota(\phi_{\tau_i}(y))(b^{(i)})^{1/2})$ for $y\in D_{\tau_i},$ $1\le i\le N,$ and 
$\psi_i'': D_{\tau_i}\to l^\infty(A)/I_{_{F, \varpi}}$ by 
$\psi_i''(y)=(1-\psi_0(1_{D_1}))\psi_i'(y)(1-\psi_0(1_{D_1}))$ for $y\in D_{\tau_i},$ $1\le i\le N.$
Note that $\psi'_i$ is an order zero \cpc,\, $1\le i\le N,$ and 
$\psi_1',\psi_2',...,\psi_N'$ are mutually orthogonal, as $b^{(1)}, b^{(2)},...,b^{(N)}$ are.

Let 
$\Psi': D_V\to  l^\infty(A)/I_{_{F, \varpi}}$ be defined by $\Psi'(y_1, y_2,...,y_N)=\sum_{i=1}^N \psi_i'(y_i)$
for $y_i\in D_{\tau_i},$ $1\le i\le N.$  Then $\Psi'$ is an order zero \cpc.
Define $\Psi'': D_V\to  l^\infty(A)/I_{_{F, \varpi}}$ by 
$\Psi''(d)=(1-\psi_0(1_{D_1}))\Psi'(d)(1-\psi_0(1_{D_1}))$ for $d\in D_V.$
Then $\Psi''$ is a \cpc.

By (4'), the choice of ${\cal G},$  and (4'') above, 
we have (recall that  $\psi_0(1_{D_1})$ is a projection),  for $y\in D_V^{\bf 1}, $ 
\beq\label{Pst-1-e15-1}
&&\|(1-\psi_0(1_{D_1}))\psi_j'(y)-\psi_j''(y)
\|_{_{2, {W_j}_\varpi}}<\dt_0/N4^{2} \andeqn\\\label{Pst-1-e15}
&&\|(1-\psi_0(1_{D_1}))\Psi'(y)-\Psi''(y)
\|_{_{2, F_\varpi}}<\dt_0/4^{2}.
\eneq
%

By (3'), 
$\Pi_{X_1, \varpi}(1-\psi_0(1_{D_1}))=0.$
Thus (by (1')), for all $x\in {\cal F},$  and any $y\in D_V,$ 
\beq
&&\hspace{-0.4in}\|\Pi_{X_1, \varpi}(\iota(x))-\Pi_1(\psi_0(y_{x,0})+\sum_{j=1}^N\psi_j''(y))\|_{_{2, {X_1}_\varpi}}\\\label{Pst-1-e-30}
&&\hspace{0.2in}=\|\Pi_1(\psi_0(1_{D_1}))\Pi_{_{X_1, \varpi}}(\iota(x))-\Pi_1(\psi_0(y_{x,0}))\|_{_{2, {X_1}_\varpi}}<\dt_0/N4^{2}.
\eneq

In what follows, denote by $\pi_\tau: l^\infty(A)/I_{_{F, \varpi}}\to l^\infty(A)/I_{_{\tau, \varpi}}$
the quotient map (if $\tau\in F$) and 
$\pi_Z: l^\infty(A)/I_{_{F, \varpi}}\to l^\infty(A)/I_{_{Z, \varpi}}$
the quotient map (if $Z\subset F$).

If $\tau\in X_2\setminus X_1,$ 
since ${\rm supp}(f_0)\subset W_0\subset X_1,$ 
$\sum_{i=1}^N f_i(\tau)=1.$ 
Since $\{W_j:1\le j\le N\}$ is  pairwisely disjoint, there is only one 
$j\in \{1,2,...,N\}$ such that $\tau\in W_j.$ 
Moreover, for  $\tau\in \Omega_j:=(F\setminus X_1)\cap W_j,$ $f_j(\tau)=1$ ($1\le j\le N$).
It follows  that (also using (3''))
\beq
&&\hspace{-0.7in}\|\pi_{\Omega_j}(\psi_j'(y))-\pi_{\Omega_j}(\iota(\phi_{\tau_j}(y)))\|_{_{2, {\Omega_j}_\varpi}}\\
&&\hspace{-0.3in}=\lim_{n\to\varpi}\sup\{\tau((1-(b_n^{(j)}))^2(\phi_{\tau_j}(y))^*(\phi_{\tau_j}(y)))^{1/2}: \tau\in (F\setminus X_1)\cap W_j\}\\
&&\hspace{-0.3in}\le \|y^*y\|^{1/2}\lim_{n\to\varpi}\sup\{\tau((1-(b_n^{(j)}))^2)^{1/2}:\tau\in \Omega_j\}\\
&&\hspace{-0.3in}\le \|y\| \lim_{n\to\varpi}\sup\{\tau((1-b_n^{(j)}))^{1/2}:\tau\in \Omega_j\}\\\label{Pst-1-e-10}
&&\hspace{-0.3in}=\|y\| \lim_{n\to\varpi}\sup\{\tau((1-f_j))^{1/2}:\tau\in \Omega_j\}=0.
\eneq
%
Then, for any $y_0\in D_1^{\bf 1},$ $y_j\in D_{\tau_j}^{\bf 1}$   and $x\in {\cal F},$   by \eqref{Pst-1-e15-1},
we have that
\beq\nonumber
&&\hspace{-0.4in}\|\Pi_{\Omega_j, \varpi}(\iota(x))-\pi_{\Omega_j}(\psi_0(y_0)+\psi_j''(y_j))\|_{_{2, {\Omega_j}_\varpi}}
\le \|\pi_{\Omega_j}(\psi_0(1_{D_1}))\Pi_{\Omega_j, \varpi}(\iota(x))-\pi_{\Omega_j}(\psi_0(y_0))\|_{_{2, {\Omega_j}_\varpi}}\\\label{Pst-1-e-99}
&&\hspace{0.5in}+\dt_0/4^{2}
+\|\pi_{\Omega_j}(1-\psi_0(1_{D_1}))(\Pi_{\Omega_j, \varpi}(\iota(x))-\pi_{\Omega_j}\circ \psi_j'(y_j))\|_{_{2, {\Omega_j}_\varpi}}.
\eneq
Hence,  
by \eqref{Pst-1-e-99},
 (1'), 
\eqref{Pst-1-e-10}, and then, by ($i_0$) above,  for $x\in {\cal F},$ 
\beq\nonumber
&&\hspace{-0.4in}\|\Pi_{\Omega_j, \varpi}(\iota(x))-(\pi_{\Omega_j}(\psi_0(y_{x,0})+\Psi''(y_{x, \tau_1},...,y_{x, \tau_j},..., y_{x, \tau_N}))\|_{_{2,{{\Omega_j}_\varpi}}}\\\nonumber
&&\hspace{-0.2in}\le \dt_0/N4^{2}+\dt_0/4^{2}+\|\pi_{\Omega_j}(1-\psi_0(1_{D_1}))(\Pi_{\Omega_j, \varpi}(\iota(x))-\pi_{\Omega_j}( \iota(\phi_{\tau_j}(y_{x,\tau_j}))))\|_{_{2,{{\Omega_j}_\varpi}}}\\
&&\hspace{-0.2in}< \dt_0/NR4^{2}+\dt_0/4^{2}+\ep/4^{3}.
\eneq
Put $d_x=(y_{x,\tau_1}, ...,y_{x, \tau_j},...,y_{x, \tau_N})\in D_V.$ 
Note that $X_2\setminus X_1\subset \cup_{j=1}^N \Omega_j.$ Hence, by Proposition \ref{Punion},  we have
\beq\label{Pst-1-e-33}
\|\Pi_{X_2\setminus X_1}(\iota(x))-\pi_{X_2\setminus X_1}(\psi_0(y_{x,0})+\Psi''(d_x))\|_{_{2, X_2\setminus X_1}}<\dt_0/NR4^{2}+\dt_0/4^{2}+\ep/4^{3}.
\eneq
If $\tau\in F\setminus X_2,$  we may assume that $\tau\in W_j$ for some $j\in \{1,2,...,N\}.$
Put $Y_j=(F\setminus X_2)\cap W_j\subset \Omega_j,$ $1\le j\le N.$
By  (2'),  \eqref{Pst-1-e15},  \eqref{Pst-1-e-10} and ($i^0$) above, for all $x\in {\cal F},$ 
\beq
&&\hspace{-0.6in}\|\Pi_{Y_j, \varpi}(\iota(x))-\pi_{Y_j}(\psi_0(y_{x,0})+\Psi''(d_x))\|_{_{2, {Y_j}_\varpi}}\\
&&=\|\pi_{Y_j}(1-\psi_0(1_{D_1}))(\Pi_{Y_j, \varpi}(\iota(x))-\pi_{Y_j}\circ \Psi''(d_x))\|_{_{2, {Y_j}_\varpi}}
\\
&&\le \|\Pi_{Y_j, \varpi}(\iota(x))-\pi_{Y_j}\circ \psi_j'(y_{x, \tau_j}))\|_{_{2,{Y_j}_\varpi}}+\dt_0/4^{2}\\
&&=\|\Pi_{Y_j, \varpi}(\iota(x))-\pi_{Y_j}\circ (\iota(\phi_{\tau_j}(y_{x, \tau_j})))\|_{_{2,{Y_j}_\varpi}}+\dt_0/4^{2}\\
&&\le \|\Pi_{{\bar U}_{\tau_j}, \varpi}(\iota(x))-\bar \phi_{\tau_j}(y_{x, \tau_j})\|_{_{2, ({\bar U}_j)_{\varpi}}}+\dt_0/4^2<\ep/4^{3}+\dt_0/4^{2}.
\eneq
Thus, applying Proposition \ref{Punion}, 
we obtain that
\beq\label{Pst-1-e-35}
\|\Pi_{_{F\setminus X_2, \varpi}}(\iota(x))-\pi_{F\setminus X_2}(\psi_0(y_{x,0})+\Psi''(d_x))\|_{_{2, (F\setminus X_2)_\varpi}}
<\ep/4^{3}+\dt_0/4^{2}.
\eneq
Combining \eqref{Pst-1-e-30}, \eqref{Pst-1-e-33}  and \eqref{Pst-1-e-35}, and by 
Proposition \ref{Punion},  we obtain that,  for all $x\in {\cal F},$ 
\beq\label{Pst-e-15}
&&\hspace{-0.6in}\|\Pi_{_{F, \varpi}}(\iota(x))-(\psi_0(y_{x,0})+\Psi''(d_x))\|_{_{2, F_\varpi}}
<\ep/4.
\eneq


Recall  that $\Psi'$ is an order zero \cpc. By (4'), the choice of ${\cal G}$ and (4''), we have 
\beq
\|[(1-\psi_0(1_{D_1})),\, \Psi'(y)]\|_{_{F, \varpi}}<\dt_0/2^{2}\rforal y\in D_2^{\bf 1}.
\eneq
Applying Lemma \ref{Lsemip}, we obtain a finite dimensional \CA\, $D_3$ and
 a unital \hm\, $\Psi''': D_3\to (1-\psi_0(1_{D_1}))(l^\infty(A)/I_{_{F, \varpi}})(1-\psi_0(1_{D_1}))$ 
 such that, for any $d\in  D_V^{\bf 1},$ there exists $z_d\in D_3^{\bf 1}$ such that
 \beq
 \|\Psi''(d)-\Psi'''(z_d)\|<\ep/16.
 \eneq
 Define $D=D_1\oplus D_3$
and define $\phi: D\to l^\infty(A)/I_{_{F, \varpi}}$ by
$\phi(d', d'')=\psi_0(d')\oplus \Psi'''(d'')$ for $d'\in D_1$ and $d''\in D_3.$ 
Then $\phi$ is a \hm.
 Moreover, by \eqref {Pst-e-15},
 for each $x\in {\cal F},$ there exists $z_x\in D^{\bf 1}$ 
 such that
 \beq
 \|\Pi_{_{F, \varpi}}(\iota(x))-\phi(z_x)\|_{_{2, F_\varpi}}<\ep.
 \eneq
 This completes the transfinite induction  and the lemma follows.
\end{proof}

\section{Countable dimensional extremal boundaries}

\begin{lem}\label{LMnemb-1}
Let $A$  be a separable algebraically simple amenable \CA\, with  strict comparison and 
T-tracial approximate oscillation zero and non-empty compact $T(A).$
Let $F\subset \partial_e(T(A))$ be a  compact subset with
 ${\rm trind}(F)=\af$ for some  ordinal $\af.$
Then, for any integer $n\in \N,$ any $\ep\in (0, 1/2),$ and any finite subset ${\cal F}\subset A^{\bf 1},$ 
there exists
a \hm\, $\phi: M_n\to l^\infty(A)/I_{_{T(A), \varpi}}$  such that 
\beq
\|[\Pi_{_{\varpi}}(\iota(x)),\, \phi(y)]\|_{_{2, T(A)_\varpi}}<\ep\tforal x\in {\cal F}\tand y\in M_n^{\bf 1},
 \eneq
 and $\pi_F\circ \phi$ is unital, where $\pi_F: l^\infty(A)/I_{_{\varpi}}\to l^\infty(A)/I_{_{F,\varpi}}$
 is the quotient map.
\end{lem}

\begin{proof}
Let $n\in \N,$ $\ep\in (0,1/2)$ and a finite subset ${\cal F}\subset A^{\bf 1}$ be given.
By Proposition \ref{Pst-1}, there exist a finite dimensional \CA\, $D_0$ and  a unital \hm\, 
$\psi: D_0\to l^\infty(A)/I_{_{F, \varpi}}$ such that, for any $x\in {\cal F},$ there is $y_x\in D_0^{\bf 1}$ 
such that
\beq
\|\Pi_{_{F, \varpi}}(\iota(x))-\psi(y_x)\|_{_{2, F_\varpi}}<\ep/2.
\eneq
By Proposition \ref{Ldvi-3},   there exists a \hm\, $\phi: M_n\to l^\infty(A)/I_{_{T(A), \varpi}}$
such that $\pi_F\circ \phi: M_n\to l^\infty(A)/I_{_{F, \varpi}}$ is unital (recall \ref{Pcompact}) and 
\beq
\|[\Pi_{_{ \varpi}}(\iota(x)),\, \phi(y)]\|_{_{2, T(A)_\varpi}}<\ep\rforal x\in {\cal F}\andeqn y\in M_n^{\bf 1}.
\eneq
\end{proof}

\begin{lem}\label{LMnemb-2}
Let $A$  be a separable algebraically simple amenable \CA\, with strict comparison,
T-tracial approximate oscillation zero and non-empty compact $T(A).$
Let $F\subset \partial_e(T(A))$ be  a compact subset with
${\rm trind}(F)=\af$ for some  ordinal $\af.$
Then,
there exists, for each $n\in \N,$ 
a  \hm\, $\psi:M_n\to \pi_\infty^{-1}(A')/I_{_{T(A), \varpi}}$ such that $\pi_F\circ \psi:
M_n\to \pi_\infty^{-1}(A')/I_{_{F, \varpi}}$ is unital, where $\pi_F: l^\infty(A)/I_{_{\varpi}}\to l^\infty(A)/I_{_{F,\varpi}}$
 is the quotient map.
%
\end{lem}

\begin{proof}
Fix an integer $n\in \N.$ 
Let $\{{\cal F}_k\}$ be an increasing sequence of finite subsets of $A^{\bf 1}$ whose union is dense in 
$A^{\bf 1}.$  By 
Lemma \ref{LMnemb-1},  for each $k\in \N,$ there is a \hm\, $\wtd\psi_k: M_n\to  l^\infty(A)/I_{_{T(A), \varpi}}$
such that $\pi_F\circ \wtd\psi_k: M_n\to
 l^\infty(A)/I_{_{F, \varpi}}$ is unital  and 
 \beq
 \|[\Pi_{_{ \varpi}}(\iota(x)),\, \widetilde \psi_k(y)]\|_{_{2, T(A)_\varpi}}<1/k^2\rforal x\in {\cal F}_k\andeqn y\in M_n^{\bf 1}.
 \eneq
Thus we obtain a sequence of order zero \cpc s 
$L_k: M_n\to A$ such that
\beq\label{LMnemb-2-2}
&&\lim_{k\to\infty}\|L_k(xy)-L_k(x)L_k(y)\|_{_{2, T(A)}}=0,\\\label{LMnemb-2-3}
&&\lim_{k\to\infty}\sup\{(1-\tau(L_k(1_n)): \tau\in F\}=0,\\\label{LMnemb-2-4}
&&\|[x,\, L_k(y)]\|_{_{2, T(A)}}<1/k^2\rforal  x\in {\cal F}_k \andeqn y\in M_n^{\bf 1}.
\eneq
Define  $\psi=\Pi_\varpi\circ \{L_k\}: M_n\to l^\infty(A)/I_{_{T(A), \varpi}}$  
and 
$\phi=\Pi_{_{F, \varpi}}\circ \{L_k\}: M_n\to l^\infty(A)/I_{_{F, \varpi}}.$ 
So $\phi=\pi_F\circ \psi.$
Then $\psi$ is a \hm\, (by  \eqref{LMnemb-2-2}) and  so is $\phi.$
By \eqref{LMnemb-2-3}, $\phi$ is unital.
By \eqref{LMnemb-2-4},   since $\cup_{k=1}^\infty{\cal F}_k$ is dense in $A^{\bf 1},$
\beq
\Pi_{_{\varpi}}(\iota(x))\psi(y)=\psi(y)\Pi_{_{\varpi}}(\iota(x))\rforal x\in A\andeqn y\in M_n.
\eneq
By the central surjectivity of Sato (see \cite[Lemma 2.1]{S12}), $\psi$ maps $M_n$ to $\pi_\infty^{-1}(A')/I_{_{T(A), \varpi}}.$ 
The lemma follows.
\end{proof}

\begin{df}\label{DOS}
Let $A$ be a separable simple \CA\, with $A={\rm Ped}(A),$   $S\subset \Tw$ be a compact subset 
and $\phi: M_k\to l^\infty(A)/I_{_{S, \varpi}}$ be an order zero  \cpc. 
We say that $\phi$ has  property (Os), if there exists 
an order zero \cpc\, $\Psi=\{\wtd\psi_n\}: M_k\to l^\infty(A)$ 
such that
$\{\wtd\psi_n(1_k)\}$ is 
a permanent projection lifting of $\phi(1_k),$
and, for any $\ep>0$  there exist  $\dt>0$  and 
${\cal Q}\in \varpi$
such that
\beq
d_\tau(\wtd \psi_n(1_k))-\tau(f_\dt(\wtd\psi_n(1_k)))<\ep\tforal \tau\in \Tw 
\tand  n\in {\cal Q}.
\eneq
\end{df}

\begin{prop}{\rm \cite[Proposition 3.11]{Linoz}}\label{LLliftos}
Let $A$ be a separable  non-elementary simple \CA\, with $A={\rm Ped}(A)$ 
and $T(A)\not=\emptyset.$
Suppose that $\phi: M_k\to l^\infty(A)/I_{_{\Tw, \varpi}}$
is a \hm.
Then $\phi$ has  property 
(Os).
%
\end{prop}
%
%
%

The following is also taken from \cite{Linoz} which allows us to add maps with property (Os).

\begin{lem}[{\rm \cite[Lemma 4.2]{Linoz}}]\label{LLnmaps}
Let $A$ be a separable amenable algebraically simple \CA\, with 
non-empty compact $T(A)$
and let $k\in \N.$
Let $\{S_n\}\subset 
\partial_e(T(A))$ be an increasing sequence of compact subsets.
%
Suppose that, for each $n,$
there exists a unital \hm\, 
 $\psi_n: M_k\to 
 \pi_\infty^{-1}(A')/I_{_{S_n, \varpi}}$ which has property (Os).

Then, for any $\ep>0$ and any finite subset 
${\cal F}\subset A,$ there exists a sequence 
of order zero \cpc s $\phi_n: M_k\to A$ such that
\beq 
&&\|[\phi_n(b), \, a]\|<\sum_{j=1}^n \ep/2^{i+2}\tforal  a\in {\cal F}\tand b\in M_k^{\bf 1},\\
&&\tau(\phi_n(1_k))>1-(\ep/2^{n+5})^2\tforal \tau\in \cup_{j=1}^n S_j,\\
\eneq
and, if $\tau(\phi_n(1_k))>1-(\sigma/2)^2
$ for some $\tau\in T(A)$ and $\sigma\in (0,1/2),$   then 
\beq
\tau(\phi_{n+1}(1_k))>1-(\ep/2^{n+4})^2-\sigma/2
\tforal n\in \N.
\eneq
\end{lem}

\begin{lem}\label{Lcover2}
Let $A$ be a separable amenable  algebraically simple \CA\, with non-empty  compact
$T(A),$ 
 strict comparison
and T-tracial approximate oscillation zero such that 
$\partial_e(T(A))=\cup_{n=1}^\infty S_n,$ where $S_n\subset S_{n+1},$ each $S_n$ is compact
and has transfinite dimension $\af_n.$ 
%
Then, 
for any integer $k\in\N,$ any $\ep>0$ and 
any finite subset ${\cal F}\subset A,$  there 
is an order zero \cpc\, $\phi: M_k\to A$ such that
\beq
\tau(\phi(1_k))>1-\ep\tforal \tau\in T(A)\tand \|[f, \, \phi(b)]\|<\ep
\eneq
for all $f\in {\cal F}$ and $b\in M_k^{\bf 1}.$
\end{lem}

\begin{proof}

Fix $k\in \N,$ $\ep\in (0,1/2)$ and a finite subset ${\cal F}\subset A.$ 

Applying Lemma \ref{LMnemb-2}
(with $S_n$ instead of $F$ for each $n$), Proposition \ref{LLliftos} and Lemma \ref{LLnmaps}, we obtain 
a sequence of  order zero \cpc s  $\phi_n: M_k\to A$ 
which satisfies the conclusion of Lemma \ref{LLnmaps} (with respective to $\{S_n\},$ ${\cal F}$ and $\ep$). 

Put $G_n=\{\tau\in T(A): \tau(\phi_n(1_k))>1-(\ep/4)^2\},$ $n\in \N.$

Let $\tau\in T(A).$ Then, by The Choquet  Theorem, there is a probability  Borel measure 
$\mu_\tau$ of $T(A)$ concentrated on $\partial_e(T(A))$ 
such that
\beq
\tau(f)=\int_{\partial_e(T(A))} f d\mu_\tau\rforal f\in \Aff(T(A)).
\eneq
Put $F_1=S_1$ and $F_{n+1}=S_{n+1}\setminus S_n,$ $n=1,2,....$ 
Let $\mu_{\tau, n}=\mu_\tau|_{F_n}.$ 
Then 
\beq
\tau(f)=\sum_{n=1}^\infty \int_{F_n} f d\mu_{\tau, n}\rforal f\in \Aff(T(A)).
\eneq
Since $\|\mu_\tau\|=1,$ there is $n_1\in \N$ such that
\beq\label{Lcover-n3}
\sum_{m>n_1}^{\infty} \|\mu_{\tau, m}\|<(\ep/8)^2\andeqn  \mu_\tau(S_{n_1})>1-(\ep/8)^2.
\eneq
We may assume that $n_1>2.$
Then (as $\{\phi_n\}$ satisfies the conclusion of Lemma \ref{LLnmaps}), if $n\ge n_1,$ 
\beq\label{Lcover-4}
\tau(\phi_n(1_k))&=&\int_{S_{n_1}} \widehat{\phi_n(1_k)}(s)d\mu_\tau
 +\sum_{m>n_1}\int_{F_m}\widehat{\phi_n(1_k)}(s) d\mu_{\tau, m}\\
&\ge & (1-(\ep/2^{n+5})^2)\mu_\tau(S_{n_1})  >1-(\ep/4)^2.
\eneq
In other words, $\tau\in G_n$ (for $n\ge n_1$). 
It follows that $\cup_{n=1}^\infty G_n\supset T(A).$ 
Since $T(A)$ is compact,  there exists $n_0\in \N$ such that
\beq
T(A)\subset \cup_{n=1}^{n_0}G_n.
\eneq
Let $\tau\in T(A).$ Suppose that $\tau\in G_j$ for some $j\le n_0.$
Then
\beq
\tau(\phi_j(1_k))>1-(\ep/4)^2.
\eneq
It follows from the conclusion of Lemma \ref{LLnmaps} that
\beq
\tau(\phi_{n_0}(1_k))>1-\sum_{i=j+1}^{n_0}\ep/2^{i+1+4}-\ep/4>1-\ep.
\eneq
Choose $\phi=\phi_{n_0}.$ Then, for all $x\in {\cal F}\andeqn b\in M_k^{\bf 1},$
\beq
\hspace{-0.2in}\|[x,\, \phi(b)]\|<\ep,
\andeqn
\tau(\phi(1_k))>1-\ep\rforal \tau\in T(A).
\eneq
The lemma follows.
\end{proof}

\begin{NN}({\bf Proof of Theorem \ref{TM-2}})\label{NNProof}


\begin{proof}
The equivalence of (1) and (2) is proved in \cite{FLosc} without assuming 
$\wtd T(A)$  has $\sigma$-compact  countable-dimensional  extremal  boundaries.
(3) $\Rightarrow$ (2) is proved in \cite{Rrzstable} for the unital case. 
In fact it is proved in \cite{Rrzstable} that, if $A$ is ${\cal Z}$-stable (unital or not),  then 
$A$ always has strict comparison.  For the non-unital case, it follows from \cite{FLL} that 
$A$ also has stable rank one.  None of the above requires the assumption that 
$\wtd T(A)$ has $\sigma$-compact  countable-dimensional extremal  boundaries.

We need to show that (1) $\Rightarrow$ (3). 

So we assume that $A$ has strict comparison and 
has T-tracial  approximate oscillation zero.  By 
Theorem  7.12  of \cite{FLosc},  the canonical map 
$\Gamma: \Cu(A)\to {\rm LAff}_+({\wtd{T}}(A))$ is surjective.
Choose $a\in {\rm Ped}(A)_+^{\bf 1}\setminus \{0\}$ such that 
$d_\tau(a)$ is continuous on ${\wtd{T}}(A).$ Define $A_1=\Her(a).$
Then $A_1$ has continuous scale, algebraically simple and $T(A_1)$ is compact (see Proposition 5.4 of \cite{eglnp}). 
Since $A_1\otimes {\cal K}\cong A\otimes {\cal K}$ (\cite{Br}), 
by Cor. 3.1 of \cite{TWst},
 it suffices to show that $A_1$ is ${\cal Z}$-stable.  Note that, since $T(A_1)$ 
 is compact, it forms a basis for the cone of $\wtd{T}(A\otimes {\cal K}).$ 
 Since $\wtd T(A)$ 
 has 
 a 
 $\sigma$-compact countable-dimensional extremal boundary, 
 by \cite[Proposition 2.17  (1)]{Linoz},  
 $\partial_e(T(A_1))$ is 
  $\sigma$-compact and 
  countable-dimensional.  
  Thus, we may assume that $T(A)$ is compact and 
  $\partial_e(T(A))$ is $\sigma$-compact and countably dimensional.  
  
  Write  $\partial_e(T(A)=\cup_{n=1}^\infty X_n,$
 where $X_n\subset X_{n+1}$ and each $X_n$ is compact and has transfinite dimension 
 $\af_n$ for some countable ordinal $\af_n,$ $n\in \N.$

Fix an integer $k\ge 2.$
Let $\{{\cal F}_n\}$ be an increasing sequence of finite subsets of $A$ such that 
$\cup_{n=1}^\infty {\cal F}_n$ is dense in $A.$ By Lemma 
\ref{Lcover2}, for each $n\in \N,$ there exists 
an order zero \cpc\, $\phi_n: M_k\to A_1$ such that
\beq\label{TM-1-1}
&&\|[a,\, \phi_n(b)]\|<1/n\rforal a\in {\cal F}_n\andeqn b\in M_k^{\bf 1},\andeqn\\\label{TM-1-2}
&&\sup\{\tau(\phi_n(1_k)): \tau\in T(A_1)\}>1-1/n,\,\,\,n\in \N.
\eneq
Define $\Phi: M_k\to l^\infty(A)$ by $\Phi(b)=\{\phi_n(b)\}.$ Then, by \eqref{TM-1-1},
$\Phi$ maps $M_k$ into $\pi_\infty^{-1}(A_1').$
By 
\eqref{TM-1-2},
\beq
\lim_{n\to\infty}\sup\{1-\tau(\phi_n(1_k)): \tau\in T(A_1)\}=0.
\eneq
It follows that $\Pi_{\varpi}\circ \Phi$ is a unital order zero \cpc. Therefore it is a unital \hm.
Hence   (3) follows from a result of Matui-Sato (see, explicitly, 
  Corollary 5.11 and Proposition 5.12  of \cite{KR}, for example) in the unital case.
For non-unital case,   
let us  use the result in \cite{CLZ}. In this case,
since $T(A)$ is compact, $A$ is uniformly McDuff (see  Definition 4.1 of \cite{CLZ},
or Definition 4.2 of \cite{CETW}). 
Since $A$ has strict comparison,  by Proposition 4.4 of \cite{CLZ} (also a  version of Matui-Sato's result),
we conclude that
$A\cong A\otimes {\cal Z}.$
%
  \end{proof}
  \end{NN}

\begin{cor}\label{CCM-1}
Let $A$ be a non-elementary separable amenable  simple \CA\, with ${\wtd{T}}(A)\setminus \{0\}\not=\emptyset$
which has a basis $S$ such that  $\partial_e(S)$ has countably many points.
Then  following are equivalent.

(1) $A$ has strict comparison and 

(2) $A\cong A\otimes {\cal Z},$


\end{cor}

\begin{proof}
It follows from \cite[Proposition 2.17]{Linoz} that 
there exists $e\in {\rm Ped}(A)_+^{\bf 1}$ such that 
$T_e=\{\tau\in \wtd T(A): \tau(e)=1\}$ has countably many extremal points.
%
By Theorem 5.9 of \cite{FLosc}, 
$A$ has T-tracial  approximate oscillation zero.
Then the corollary follows from Theorem \ref{TM-2}.
%
\end{proof}

\section{Epilogue}
This last section is an attempt to clarify some ideas behind previous sections and perhaps 
serves as an invitation for further study.

One may notice that much of Section 5 and 7 is aimed to show relevant \CA s  have the following 
tracial approximation property.

\begin{df}\label{DWTAC}
Let $A$ be a separable simple \CA\, with $\wtd T(A)\setminus \{0\}\not=\emptyset.$ We say that 
$A$ has property (WTAC), if, for any $a\in {\rm Ped}(A)_+,$ any
$\ep>0$ and any finite subset ${\cal F}\subset \Her(a)^{\bf 1},$
there are a finite dimensional \CA\, $D$ and 
\hm\, $\phi: C_0((0,1])\otimes D\to \Her(a)$ such that, for any 
$x\in {\cal F},$ there is $d_x\in (C_0((0,1])\otimes  D)^{\bf 1}$ such that
\beq
\sup\{\|x-\phi(d_x)\|_{_{2, \tau}}: \tau\in \overline{T(\Her(a))}^w\}<\ep.
\eneq
%

One may think that (WTAC) is an  abbreviation  of 
``weakly tracial approximation of cones."   It is a rather weak approximation property.
One easily sees that every separable simple \CA\, with tracial rank at most one has property (WTAC).
\end{df}

\begin{cor}\label{TWTAC}
Let $A$ be a separable amenable algebraically simple \CA\, with T-tracial approximate oscillation zero.
Suppose that $T(A)$ is a nonempty compact set  such that $\partial_e(T(A))$ 
is compact and has countable dimension. Then $A$ has property (WTAC).
\end{cor}

\begin{proof}
Recall that, since $\partial_e(T(A))$ is assumed to be compact and countable dimension, it has transfinite dimension 
(see \cite[Corollary 7.1.32]{En}).
The corollary follows immediately from Proposition \ref{Pst-1} by taking $F=\partial_e(T(A))$
(we also note that any finite dimensional \CA\, $D$  is a quotient of the cone $C_0((0,1])\otimes D$). 
\end{proof}

\begin{cor}\label{CWTAC}
Let $A$ be a separable algebraically simple amenable  \CA.
Suppose that $T(A)$ is a nonempty compact set  such that $\partial_e(T(A))$ is compact and   has countably many points.
 Then $A$ has property (WTAC).
\end{cor}

\begin{proof}
It follows from  \cite[Theorem 5.9]{FLosc}   and  \cite[Proposition 2.17]{Linoz} that $A$ has T-tracial approximate oscillation zero.   Hence Corollary \ref{TWTAC} applies.
\end{proof}

Recall (\cite[Definition 2.9]{Linsrk1}) that a separable simple \CA\, $A$ is called regular, if it is purely infinite, or 
if it has almost stable rank one and $\Cu(A)=(V(A)\setminus \{0\}) \sqcup {\rm LAff}_+(\wtd QT(A))$
(see \cite[Definition 2.10]{Linsrk1}).  It is proved  in \cite[Theorem 1.1]{FLosc} that
a finite separable simple \CA\, $A$ is regular if and only if $A$ has strict comparison and T-tracial approximate 
oscillation zero, and, if and only if $A$ has strict comparison and has stable rank one
(see  also \cite[Theorem 9.4]{FLosc}).  By (the proof of) \cite[Theorem 5.10]{FLosc},  a finite separable 
simple \CA\, $A$ is regular  if and only 
if $\Cu(A)=(V(A)\setminus \{0\}) \sqcup {\rm LAff}_+(\wtd QT(A))$  and $a\lesssim b$ (for any $a, b\in (A\otimes {\cal K})_+$) implies that there is 
$x\in A\otimes {\cal K}$ \st\, $x^*x=a$ and $xx^*\in \overline{bAb}.$
One of the starting points of this research is based on the following fact.

\begin{thm}\label{TLast}
Let $A$ be a finite separable non-elementary amenable regular  simple \CA.
Suppose that $A$ has property (WTAC). Then $A$  
is ${\cal Z}$-stable.
\end{thm}

\begin{proof}
Choose an element $e\in {\rm Ped}(A)_+\setminus \{0\}$ with $d_\tau(e)$ is continuous on $\wtd T(A).$
It suffices to show that $\Her(e)$ is ${\cal Z}$-stable.  Hence we may assume that $A=\Her(e).$ 
Note that, now $A=\Her(e)$ is algebraically simple and has property (WTAC), and $T(A)$
is compact.
%
Fix a finite subset ${\cal F}\subset A^{\bf 1}$ and $\ep>0.$
Since $A$ has property (WTAC), there are  a finite dimensional \CA\, $D$ and 
a \hm\, 
$\phi_c: C_0((0,1])\otimes D\to A$ such that, for any $x\in {\cal F},$ there is $d_x\in  (C_0(0,1])\otimes D)^{\bf 1}$ satisfying the following: 
\beq
\|x-\phi_c(d_x)\|_{_{2, T(A)}}<\ep/2.
\eneq
Define $\Phi: D\to l^\infty(A)$ by $\Phi(d)=\{\phi_c(d)\}_{n\in \N}$ for all $d\in D$ and 
$\phi: D\to l^\infty(A)/I_{_{\varpi}}$ by $\phi=\Pi_{_{\varpi}}\circ \Phi.$
Then 
\beq
\|\Pi_{_\varpi}(\iota(x))-\phi(d_x)\|_{_{2, T(A)_\varpi}}<\ep/2\rforal x\in {\cal F}.
\eneq
It follows from  Lemma \ref{LnFO} (by taking $K=\partial_e(T(A))$)
that there are a finite dimensional \CA\, $D_1$ and a \hm\, 
$h: D_1\to l^\infty(A)/I_{_\varpi}$ such that
\beq
\|\Pi_{_\varpi}(\iota(x))-h(y_x)\|_{_{2, T(A)_\varpi}}<\ep/2\rforal x\in {\cal F}\andeqn {\rm some}\,\, y_x\in D_1^{\bf 1}.
\eneq
By Proposition \ref{Ldvi-3}, for each integer $n\in \N,$ there exists a unital \hm\, $\psi': M_n\to l^\infty(A)/I_{_\varpi}$
such that
\beq
\|[\Pi_{_\varpi}(\iota(x)),\, \psi'(y)]\|_{_{2, T(A)_\varpi}}<\ep\rforal x\in {\cal F}\andeqn y\in M_n^{\bf 1}.
\eneq

Since the above holds for any $\ep>0$ and any finite subset ${\cal F}\subset M_n^{\bf 1},$ 
by exactly the same proof used in the proof of Lemma \ref{LMnemb-2}, we obtain a unital \hm\, 
$\psi: M_n\to \pi_\infty^{-1}(A')/I_{_\varpi}.$ 
%
In other words, $A$ is  uniformly McDuff.
Since $A$ has strict comparison, 
by Proposition 4.4 of \cite{CLZ}, $A$ is ${\cal Z}$-stable. 
\end{proof}

This leads us to the following questions:

{\bf Questions}: Does every separable simple amenable  \CA\,  with T-tracial approximate oscillation  zero
and with $\wtd T(A)\setminus \{0\}\not=\emptyset$ have
property 
(WTAC)?  
One may also ask:
 Does every finite separable simple amenable regular 
\CA\, have property (WTAC)?

\vspace{0.2in}

{\bf Added in January 2025}

Recently it has been shown that a simple \CA\, is regular  if and only if it is pure
(see \cite{Lin25}).


 \providecommand{\href}[2]{#2}
 %






\end{document}